\documentclass[10pt]{article}
\usepackage{amssymb,amscd,amsthm,tikz-cd,bbold,mathrsfs,textcomp,mathtools,setspace,yfonts,titlesec,leftidx, old-arrows, appendix, enumitem}

\usepackage[margin=1.00in]{geometry}
\usetikzlibrary{calc}

\usepackage{url}
\usepackage{titling}

\tikzset{curve/.style={settings={#1},to path={(\tikztostart)
    .. controls ($(\tikztostart)!\pv{pos}!(\tikztotarget)!\pv{height}!270:(\tikztotarget)$)
    and ($(\tikztostart)!1-\pv{pos}!(\tikztotarget)!\pv{height}!270:(\tikztotarget)$)
    .. (\tikztotarget)\tikztonodes}},
    settings/.code={\tikzset{quiver/.cd,#1}
        \def\pv##1{\pgfkeysvalueof{/tikz/quiver/##1}}},
    quiver/.cd,pos/.initial=0.35,height/.initial=0}

\tikzset{tail reversed/.code={\pgfsetarrowsstart{tikzcd to}}}
\tikzset{2tail/.code={\pgfsetarrowsstart{Implies[reversed]}}}
\tikzset{2tail reversed/.code={\pgfsetarrowsstart{Implies}}}
\tikzset{no body/.style={/tikz/dash pattern=on 0 off 1mm}}

\titleformat{\subsection}[runin]
  {\normalfont\bfseries}{\thesubsection}{1em}{}
\titleformat{\subsubsection}[runin]
    {\normalfont\bfseries}{\thesubsubsection}{1em}{}

\newcommand{\newoverline}[1]{\smash{\overline{#1}}\vphantom{#1}}

\newcommand{\A}{\mathbb{A}}
\newcommand{\G}{\mathbb{G}}
\newcommand{\cO}{\mathcal{O}}

\newcommand{\fD}{\mathfrak{D}}

\newcommand{\cG}{\mathcal{G}}
\newcommand{\cY}{\mathcal{Y}}
\newcommand{\cZ}{\mathcal{Z}}

\newcommand{\cX}{\mathcal{X}}

\newcommand{\sP}{\mathscr{P}}

\newcommand{\fL}{\mathfrak{L}}
\newcommand{\chlam}{\check{\lambda}}
\newcommand{\oZ}{\mathring{\mathcal{Z}}}
\newcommand{\pZ}{\newoverline{\mathcal{Z}}}

\newcommand{\pibar}{\overline{\pi}}

\newcommand{\ofp}{\overline{\mathfrak{p}}}
\newcommand{\ofq}{\overline{\mathfrak{q}}}

\newcommand\blfootnote[1]{
    \begingroup
    \renewcommand\thefootnote{}\footnote{#1}
    \addtocounter{footnote}{-1}
    \endgroup
}

\DeclareMathOperator{\IC}{IC}
\DeclareMathOperator{\sIC}{IC_Z^{\frac{\infty}{2}}}

\DeclareMathOperator{\sICR}{IC^{\frac{\infty}{2}}_{\Ran}}
\DeclareMathOperator{\sICC}{IC^{\frac{\infty}{2}}_{\Conf}}

\DeclareMathOperator{\Gr}{Gr}
\DeclareMathOperator{\id}{id}

\DeclareMathOperator{\op}{op}
\DeclareMathOperator{\Bun}{Bun}
\DeclareMathOperator{\dBun}{\newoverline{\Bun}}

\DeclareMathOperator{\fact}{fact}
\DeclareMathOperator{\Hom}{Hom}
\DeclareMathOperator{\Map}{Map}

\DeclareMathOperator{\add}{add}

\DeclareMathOperator{\corr}{corr}
\DeclareMathOperator{\gen}{gen}

\DeclareMathOperator{\ind}{ind}

\DeclareMathOperator{\Ran}{Ran}
\DeclareMathOperator{\str}{str}
\DeclareMathOperator{\untl}{untl}
\DeclareMathOperator{\eff}{eff}
\DeclareMathOperator{\Conf}{Conf}
\DeclareMathOperator{\Spec}{Spec}

\DeclareMathOperator{\colim}{colim}

\DeclareMathOperator{\Spc}{Spc}
\DeclareMathOperator{\pt}{pt}

\DeclareMathOperator{\Stk}{PrStk}
\DeclareMathOperator{\LStk}{LaxPrStk}
\DeclareMathOperator{\Cat}{Cat}
\DeclareMathOperator{\Groth}{Groth}
\DeclareMathOperator{\Fun}{Fun}

\DeclareMathOperator{\uMap}{\underline{Map}}
\DeclareMathOperator{\sif}{\mathscr{F}\ell_{\emph{x}}^{\frac{\infty}{2}}}

\DeclareMathOperator{\Div}{Div}

\DeclareMathOperator{\disj}{disj}
\DeclareMathOperator{\QC}{QCoh}
\DeclareMathOperator{\Shv}{Shv}

\DeclareMathOperator{\Tot}{Tot}
\DeclareMathOperator{\oblv}{oblv}

\DeclareMathOperator{\Aff}{AffSch}

\DeclareMathOperator{\un}{un}
\DeclareMathOperator{\act}{act}
\DeclareMathOperator{\ovF}{\newoverline{\mathscr{F}}}
\DeclareMathOperator{\Corr}{Corr}

\numberwithin{equation}{subsection}

\newtheorem{theorem}{Theorem}[subsection]
\newtheorem{corollary}{Corollary}[theorem]
\newtheorem{lemma}[theorem]{Lemma}
\newtheorem{proposition}[theorem]{Proposition}
\newtheorem{corollary'}[theorem]{Corollary}

\newtheorem{remark}[theorem]{Remark}
\theoremstyle{definition}
\newtheorem{definition}[theorem]{Definition}
\newtheorem{definition/construction}[theorem]{Definition/construction}

\title{Semi-infinite flags and Zastava spaces}
\author{Andreas Hayash}

\begin{document}

\maketitle

\begin{abstract}
    We give an interpretation of the semi-infinite intersection cohomology sheaf associated to a semisimple simply-connected algebraic group in terms of finite-dimensional geometry. Specifically, we describe a procedure for building factorization spaces over moduli spaces of finite subsets of a curve from factorization spaces over moduli spaces of divisors, and show that under this procedure the compactified Zastava space is sent to the support of the semi-infinite IC sheaf in the factorizable Grassmannian. We define ``semi-infinite t-structures" for a large class of schemes with an action of the multiplicative group, and show that, for the Zastava, the limit of these t-structure recovers the infinite-dimensional version. As an application, we also construct factorizable parabolic semi-infinite IC sheaves and a generalization (of the principal case) to Kac-Moody algebras.\blfootnote{\noindent \emph{$2020$ Mathematics Subject Classification} $14$D$24$.
    
     }
\end{abstract}

\tableofcontents
\newpage

\section{Introduction}

\subsection{Notation.} Fix an algebraically closed field $k$ of characteristic zero, and let $X$ be a connected smooth projective curve over $k$ of genus $g$. We will also fix a semisimple and simply-connected algebraic group\footnote{More generally, one can work with a reductive group with simply-connected derived subgroup. This is essentially because the semi-infinite orbit considered in the text only depends on the latter. However, for simplicity we will only consider semisimple groups.} $G$ over $k$ with Lie algebra $\mathfrak{g}$, together with a choice of Borel subgroup $B \subseteq G$. Choose a splitting $B=TN$ of $B$ into a maximal torus $T$ and unipotent radical $N$, and let $B^-$ be the Borel opposite to $B$ containing $T$. By $\Lambda$ we will denote the coweight lattice of the pair $(G,T)$, and our choice of Borel determines a subset $\{\alpha_i\} \subseteq \Lambda$ of simple coroots indexed by a set $\mathcal{I}_G$. Moreover, by $\Lambda^-$ we mean those coweights $\lambda$ which can be expressed as a sum $\sum_{i \in \mathcal{I}_G} (-n_i)\alpha_i$ with each $n_i$ a nonnegative integer. Dually, we will denote by $\check{\Lambda}$ the weight lattice of $G$, and by $\check{\Lambda}^+$ the subset of dominant weights. 

\subsection{Semi-infinite flags and Zastava spaces.} The purpose of this text is to establish a new connection between the semi-infinite flag variety of Feigin-Frenkel and the Zastava spaces of Drinfeld-Finkelberg-Mirkovi\'{c}. 

The semi-infinite flag variety $\sif$ at a point $x \in X$ is an object of great interest in geometric representation theory. It was first defined in \cite{FF} as the quotient \[\sif \coloneqq \mathfrak{L}_x G/\mathfrak{L}_x N \mathfrak{L}^+_x T\]where for an algebraic group $H$ we denote by $\mathfrak{L}_x H$ its loop group at $x$ and by $\mathfrak{L}_x^+ H$ its group of arcs\footnote{Strictly speaking, there are in the literature various inequivalent, though related definitions of the semi-infinite flag variety. For example, the approach taken in \cite{R2} is expected to yield a different object than the approach taken in the present text. Although the two are closely related, we will refrain from a discussion of the latter.}. Its category $\fD(\sif)$ of D-modules is expected to be the target of a localization theory for modules over the Kac-Moody algebra $\hat{\mathfrak{g}}$ \cite{Ga1}\cite{Ga2}, similar in spirit to the celebrated result of Beilinson-Bernstein \cite{BB}. For related reasons, $\fD(\sif)$ is also expected to be the $\mathfrak{L}_x G$-category associated to the trivial $\check{G}$-local system on $\mathring{D}_x$ under the conjectural local geometric Langlands correspondence \cite{FG2}. Here $\mathring{D}_x \coloneqq D_x \setminus x$ is the punctured formal disc about $x$.  

Early constructions of the category $\fD(\sif)^{\mathfrak{L}^+_x G}$ of $\mathfrak{L}^+ G$-invariants in $D(\sif)$ were proposed in \cite{FM}, and then later in \cite{ABBGM}. Both use in an essential way the \emph{Zastava spaces} $\cZ$, also first defined in \cite{FM}. The Zastava space $\cZ$ associated to $G$ is a finite-dimensional local model for $\sif$ and is defined as the moduli space of $G$-bundles together with a generalized $N$-reduction and a generically transverse $B^-$-reduction. Here by ``generalized" we mean that the reduction is allowed to have points of degeneracy. Geometrically, the Zastava spaces are intended as a way organize transverse slices to $\mathfrak{L}^+G$-orbits in $\sif$ into a single space \cite{FM}.

A crucial feature of $\cZ$ is that it factorizes over the configuration space $\Conf$ of points in $X$ with coefficients in $\Lambda^-$. Roughly, this means that the fiber of $\cZ$ over a disjoint union of divisors can be written as the product of fibers over each component divisor. Moreover, its intersection cohomology sheaf $\IC_{\cZ}$ has a canonical structure of a factorization algebra. Intuitively, the factorization structure on $\cZ$ should be thought of as making $\cZ$ into a kind of monoid\footnote{In fact, this is more than an analogy since a factorization structure is a monoid (or commutative algebra) structure in a category of correspondences.}, and $\IC_{\cZ}$ should be thought of as a multiplicative sheaf with respect to this monoid structure. For any factorization algebra there is also an associated category of factorization modules at the point $x$. In \cite{FM} and \cite{ABBGM}, the category $\fD(\sif)^{\mathfrak{L}^+_x G}$ is essentially defined to be the category of factorization modules for the intersection cohomology sheaf $\IC_{\cZ}$ at $x$. 

\subsection{Semi-infinite flags and the Ran sheaf.} More recently, another construction of $\fD(\sif)^{\mathfrak{L}^+_x G}$ was proposed in \cite{Ga1} and \cite{Ga2} using the affine Grassmannian $\Gr_{G,\Ran}$ over the moduli $\Ran$ of finite subsets of $X$. There is a natural action of the loop group $\mathfrak{L}_{\Ran} N$ of $N$ relative to the $\Ran$ space on $\Gr_{G,\Ran}$ whose orbits are indexed by the coweight lattice $\Lambda$ and denoted $S^{\lambda}_{\Ran}$ for $\lambda \in \Lambda$. In \cite{Ga2} Gaitsgory constructs a factorization algebra $\sICR$ in $D$-modules over the Beilinson-Drinfeld Grassmannian supported on the closure $\newoverline{S}^0_{\Ran}$ of the semi-infinite orbit $S^0_{\Ran}$ whose category of factorization modules is expected to possess the properties desired of $\fD(\sif)^{\mathfrak{L}^+_x G}$ \cite{Ga1}. 

By definition, $\sICR$ is the intermediate extension of the dualizing sheaf $\omega_{S^0_{\Ran}}$ relative to a \emph{semi-infinite} t-structure defined on the category of $\mathfrak{L}_{\Ran} N$-equivariant D-modules on $\newoverline{S}^0_{\Ran}$. The latter t-structure can loosely be thought of as a perverse t-structure on $D$-modules over $\newoverline{S}^0_{\Ran}$ which are ``constructible" with respect to the semi-infinite stratification. Note that this context is quite different from the traditional one since the strata $S^{\lambda}_{\Ran}$ are of infinite dimension and of infinite codimension. In particular, in the semi-infinite t-structure $\omega_{S^0_{\Ran}}$ lies in the heart whereas in the classical perverse t-structure it lies in cohomological degree minus infinity. 

\subsubsection{} The obvious question one might ask is if the early constructions from \cite{FM} and \cite{ABBGM} can be compared to the later one in \cite{Ga2}. A first step towards doing so is to reconstruct the semi-infinite intersection cohomology sheaf $\sICR$ from the geometry of the Zastava spaces. In the present paper such a reconstruction is performed (see Theorems \ref{effisun} and \ref{recovery}). We produce a sheaf $\sIC$ on a compactification $\pZ$ of $\cZ$ which is in a precise sense equivalent to $\sICR$. The construction of $\sIC$ follows a general pattern that occurs for a large class of schemes with a $\G_m$-action. Each such scheme admits a t-structure using its stratification into attracting loci, i.e. its \emph{Bia{\l}ynicki-Birula decomposition}. This will be discussed in detail in Section \ref{semiinfinitetstructures} (see also \ref{intrototstructure} of the Introduction for a discussion of the intuition behind the construction).

It is also natural to ask for a comparison of factorization modules for $\sIC$ and $\sICR$. Although such a comparison is beyond the scope of the paper and will be carried out in a future work, we can nevertheless say a few words. It is fairly clear that a suitably defined category of factorization modules with extra structure for $\sIC$ is equivalent to factorization modules for $\sICR$. Additionally, calculations carried out by the author indicate that this category is \emph{not} equivalent to factorization modules for $\IC_{\cZ}$, but that it is indeed an infinity categorical enhancement of the category defined in \cite{ABBGM}. As such, it does indeed seem that $\cZ$ on its own is insufficient to describe $\sif$.

\subsection{The reconstruction principle, aka accumulation of dust.} The origin of the idea for this paper can be found in \cite{M}. In \emph{loc. cit.} Mirkovi\'{c} observes that many interesting spaces $\cY_x$, usually ind-schemes, associated to the formal disc $D_x$ around a point $x \in X$ can be ``reconstructed" as a colimit from a family of objects, usually schemes, living over the $n$-th infinitesimal neighborhoods $D^n_x \simeq \Spec\big(k[z]/(z^n)\big)$ of $x$. 

More precisely, such a $\cY_x$ can often be endowed with an assignment of a space $\cY^n_x$ to each $D^n_x$, associative maps $\cY^n_x \to \cY^m_x$ whenever $n \leq m$, and an isomorphism \[\colim_n \cY^n_x \overset{\sim}{\longrightarrow} \cY_x\] as objects associated to $D_x$. As a toy example, note that the $D$-module $\cO_{D_x}^{\vee}$ of distributions on $D_x$ can be written as a colimit \[\colim_n \cO_{D^n_x}^{\vee} \overset{\sim}{\longrightarrow} \cO_{D_x}^{\vee}\] of distributions on the finite schemes $D^n_x$. 

It is often the case that such reconstructions also occur ``factorizably," i.e. where we allow the point $x$ to move and collide with other points of $X$. Although this idea is not made precise in \cite{M}, part of the motivation of the present work is to give such a construction. To do so, we employ a correspondence 
\begin{equation}\label{introcorr}\begin{tikzcd}
	& {\Conf_{\corr}} \\
	\Conf && \Ran
	\arrow[from=1-2, to=2-1]
	\arrow[from=1-2, to=2-3]
\end{tikzcd}\end{equation}
between the Hilbert space $\Conf$ of points on $X$ with coefficients in the negative cone $\Lambda^-$ and the moduli $\Ran$ of finite subsets of $X$. 

Whenever the fibers of a space $\cY \to \Conf$ vary functorially\footnote{One should think of $\cY$ as a sort of coCartesian fibration lying over $\Conf$. To make this precise, we use the language of \emph{lax prestacks} (see Section \ref{laxprstk}).} under inclusions of divisors, we explain in Section \ref{factoverconftoran} how to obtain a space $\cY_{\Ran}$ over $\Ran$ from $\cY$ using a pull-push procedure along the diagram \eqref{introcorr}. Explicitly, the fiber of $\cY_{\Ran}$ at a finite subset $x_I \subseteq X$ is given by the colimit \begin{equation}\label{colimitfiber}\colim_{D \subseteq x_I} \cY_D \overset{\sim}{\longrightarrow} (\cY_{\Ran})_{x_I} \end{equation} over the moduli of divisors $D$ which are set-theoretically supported on the formal completion in $X$ of the union of points contained in $x_I$. It is useful to think of the isomorphisms \eqref{colimitfiber} as endowing the space $\cY_{\Ran}$ with a sort of filtration relative to $\Ran$ indexed by the cone $\Lambda^-$. 

\subsection{Reconstruction and Zastava spaces.} The central example of a space equipped with a notion of functoriality with respect to inclusions of divisors which will be considered in the present text is the compactified Zastava $\pZ$. The scheme $\Conf$ is a monoid under addition of divisors, and one can equip $\pZ$ with the structure of an augmented module for $\Conf$. Such an action endows $\pZ$ with the necessary functoriality to plug it into the machinery discussed above. 

The fiber of $\pZ$ over a point $\lambda x \in \Conf$ is given by the intersection \[\pZ_{\lambda x} \coloneqq \newoverline{S}^0_x \cap \newoverline{S}^{-,\lambda}_x\] where $S^0_x$ is the orbit of $\mathfrak{L}_x N$ in the fiber $\Gr_{G,x}$ of the Beilinson-Drinfeld Grassmannian at $x$, and $S^{-,\lambda}_x$ is the orbit of the loop group at $x$ of the opposite unipotent $N^-$. Whenever $\lambda \leq \mu$ in the order determined by $B^-$, we have an inclusion $\newoverline{S}^{-,\lambda}_x \hookrightarrow \newoverline{S}^{-,\mu}_x$ induced by the closure relations of semi-infinite orbits, and it follows that we have an isomorphism \[\colim_{\lambda} \pZ_{\lambda x} \overset{\sim}{\longrightarrow} \newoverline{S}^0_x.\] These isomorphisms assemble into an equivalence \[\pZ_{\Ran} \overset{\sim}{\longrightarrow} \newoverline{S}^0_{\Ran}\] which respects factorization structures on both sides. As a result, the geometry of $\newoverline{S}^0_{\Ran}$ can be analyzed in terms of the (finite-dimensional) geometry of $\pZ$. 

\subsection{Zastava spaces and the semi-infinite IC sheaf.} A feature of the assignment $\cY \mapsto \cY_{\Ran}$ taking a space over $\Conf$ to a space over $\Ran$ is that it allows us to calculate the category $\mathfrak{D}(\cY_{\Ran})$ of $D$-modules on $\cY_{\Ran}$ in terms of the category $\mathfrak{D}(\cY)$ (see Corollary \ref{rancategoryfromconfcategory}). An advantage of this is that $\mathfrak{D}(\cY)$ is often ``easier" to work with than $\mathfrak{D}(\cY_{\Ran})$, especially in the case $\cY$ is finite-dimensional. 

More explicitly, using the isomorphism \eqref{colimitfiber}, one can describe a $D$-module on a fiber of $\cY_{\Ran}$ over $\Ran$ equivalently as a compatible system of $D$-modules on the filtered pieces. We call such systems \emph{effective sheaves} on $\cY$. As an application to the semi-infinite IC sheaf $\sICR$, we construct such a system of $D$-modules on $\pZ$ as an intermediate extension in a ``semi-infinite t-structure" on each connected component of $\pZ$. 

\subsection{Semi-infinite t-structures.}\label{intrototstructure} As was mentioned earlier, a key construction in producing the sheaf $\sIC$ is a new t-structure for any reasonable finite-dimensional scheme $\cY$ with a $\G_m$-action. Here, ``reasonable" means at least that $\cY$ has a decomposition into locally closed attracting loci, i.e. a \emph{Bia{\l}ynicki-Birula decomposition}. On such a $\cY$, one can consider a full subcategory of $\G_m$-monodromic sheaves which are pulled back from fixed points on each stratum. By Braden's theorem, this category acquires a t-structure called the \emph{semi-infinite t-structure}. Inside this category, we have a distinguished object called \emph{semi-infinite IC sheaf} $\IC^{\frac{\infty}{2}}(\cY)$. 

The idea behind $\IC^{\frac{\infty}{2}}(\cY)$ is that it ignores the singularities of the attracting loci, and so it will differ from the ordinary intersection cohomology sheaf in general. Nevertheless, this is exactly the construction needed to define $\sIC$. The reason for this can be thought of heuristically as follows. Although the Zastava space $\cZ$ models transverse slices to $\mathfrak{L}^+G$-orbits in $\sif$, unlike $\pZ$ it does not contain enough redundancy to apply the reconstruction principle. Indeed if $\pZ$ can be thought of as a filtration of $\newoverline{S}^0_{\Ran}$, the Zastava $\cZ$ is akin to its associated graded. Unfortunately, in compactifying $\cZ$ we introduce certain singularities that are not detected by $\sICR$. By taking an intermediate extension in the semi-infinite t-structure on $\pZ$, we manage to ignore these singularities. 

Since $\sIC$ pretends certain singularities do not exist, one might guess that $\sIC$ would become an ordinary IC sheaf after resolving those singularities. This is in fact possible, and is carried out in Section \ref{kontsevich}. Namely, we introduce a factorizable Deligne-Mumford stack $\pZ_K$ built out of the resolution of singularities for Drinfeld's compactification constructed in \cite{C}. It is equipped with a proper birational map to $\pZ$ which is a strata-wise resolution of singularities and such that the pullback of $\sIC$ to $\pZ_K$ is the ordinary IC sheaf. 

\subsection{Applications.} In Section \ref{parabolicicsheaves} we apply the machinery developed in the bulk of the text to construct \emph{parabolic semi-infinite IC sheaves} on the Beilinson-Drinfeld Grassmannian $\Gr_{G,\Ran}$. Namely, we consider an intermediate extension in the semi-infinite t-structre on the \emph{parabolic Zastava space} \cite{BFGM} and apply the reconstruction principle to obtain a sheaf on the affine Grassmannian.

We also construct generalizations of the principal semi-infinite IC sheaf to Kac-Moody algebras obtained from a graph with no edge loops. A notable triumph of \cite{M} together with \cite{MYZ} is a generalization of the compactified Zastava (and therefore of the Beilinson-Drinfeld Grassmannian) which uses only a bilinear form and a line bundle on $\Conf$. This is a possible first step in a generalization of the geometric Langlands program to Kac-Moody algebras. In Section \ref{quivericsheaves} we apply our methods to produce a semi-infinite IC sheaf on these spaces. This can possibly be viewed as a step towards producing geometric Eisenstein functors for Kac-Moody algebras. 

\subsection{The plan and statement of results.} We will now give a brief overview of the paper, as well as precise statements of some of the central results. 

In Section \ref{laxprstk} we review the theory of lax prestacks and their sheaves. Roughly, a lax prestack is a functor from commutative algebras to categories. The notion of lax prestack is a convenient framework in which to unify the theory of unital factorization structures, as well as their configuration space analogues \emph{effective factorization structures}, under a single banner. A review of the latter will occupy Section \ref{factorization}. 

In Section \ref{ransic} we will review Gaitsgory's construction of the Ran semi-infinite intersection cohomology sheaf $\sICR$. Specifically, we will sketch the construction of the t-structure inside whose heart $\sICR$ belongs, as well as the factorization structure on $\sICR$. Following this, we will review in Section \ref{zastava} the theory of Zastava spaces, and introduce the \emph{Zastava semi-infinite intersection cohomology sheaf} as Definition \ref{zsic}. 

Section \ref{genzastava} will contain our first main result. Namely, we will identify a category $\fD_{\eff}(\pZ)$ of sheaves on $\pZ$ with extra structure to which $\sIC$ will belong. The following will be proved as Theorem \ref{effisun}. 

\begin{theorem} There is an equivalence of categories \[\fD_{\eff}(\pZ) \overset{\sim}{\longrightarrow} \fD_{\untl}(\newoverline{S}^0_{\Ran})\] which sends $\sIC$ to $\sICR$. 
\end{theorem}

In the statement of the theorem, the category $\fD_{\untl}(\newoverline{S}^0_{\Ran})$ is the category of unital D-modules on $\newoverline{S}^0_{\Ran}$, i.e. D-modules which are equivariant for the action of $\Ran$ on $\newoverline{S}^0_{\Ran}$. The proof will proceed geometrically; we will identify a sheafification of the quotient $\pZ/\Conf$ with a sheafification of the quotient $\newoverline{S}^0_{\Ran}/\Ran$ (the latter is also referred to as the \emph{generic Zastava} $\cZ^{\gen}$). As a corollary, we will show that the category of D-modules on Drinfeld's compactification $\dBun_N$ embeds fully faithfully in $\fD_{\eff}(\pZ)$. Here $\dBun_N$ denotes the moduli of $G$-bundles together with a generalized $N$-reduction. 

Section \ref{semiinfinitetstructures} may be viewed as the heart of the paper. We will introduce \emph{semi-infinite t-structures} on categories of monodromic sheaves for a class of schemes with a $\G_m$-action which we call $\G_m$-complete. We will also define sheaves $\IC^{\frac{\infty}{2}}(\cY)$ for each such scheme, obtained as an intermediate extension in the semi-infinite t-structure. For each coweight $\lambda$, we will consider the resulting category $\operatorname{SI}(\pZ)$ and prove that restriction along the projection $\pZ^{\lambda} \to \dBun_N$ is t-exact on semi-infinite categories for every $\lambda$ (see Proposition \ref{texactzastava}). We will also show that the t-structures on the categories $\operatorname{SI}(\pZ^{\lambda})$ glue to a t-structure on a certain subcategory $\operatorname{SI}^{\leq 0}_Z$ of effective sheaves. The following is Theorem \ref{texactequivalence}, and it relates $\operatorname{SI}^{\leq 0}_Z$ to its global analogue $\operatorname{SI}^{\leq 0}_{\operatorname{glob}}$, a full subcategory of $\mathfrak{D}(\newoverline{\Bun}_N)$.

\begin{theorem}
    The functor \[\newoverline{\mathfrak{p}}^!: \operatorname{SI}^{\leq 0}_{\operatorname{glob}} \longrightarrow \operatorname{SI}^{\leq 0}_Z\] is a t-exact equivalence. In particular, we have a canonical isomorphism \[\newoverline{\mathfrak{p}}^!(\IC^{\operatorname{ren}}_{\dBun_N}) \simeq \IC^{\frac{\infty}{2}}(\pZ)\] extending the tautological identifications over $\cZ^-$. 
\end{theorem}

As a corollary, we will also show that the equivalence \[\operatorname{SI}^{\leq 0}_Z \overset{\sim}{\longrightarrow} \operatorname{SI}^{\leq 0}_{\Ran}\] obtained by restricting the equivalence of Theorem \ref{effisun} is t-exact.

In Section \ref{kontsevich} we will introduce the \emph{Kontsevich Zastava space} $\pZ_K$ in \ref{introofKont}, a Deligne-Mumford factorizable stack equipped with a map $\mathfrak{r}_Z: \pZ_K \to \pZ$ which is built out of a resolution of singularities for Drinfeld's compactification. 
We will show that $\mathfrak{r}_Z$ is a strata-wise resolution of singularities (Proposition \ref{zasissmooth}). We will conclude that $\pZ_K$ is a factorizable model for $\dBun_N$ on which $\sIC$ ``lives" as an ordinary IC sheaf. Our main result of the section (Theorem \ref{sicisic}) will be the following. 

\begin{theorem} There is a canonical isomorphism \[\mathfrak{r}^!_Z (\sIC) \overset{\sim}{\longrightarrow} \IC^{\operatorname{ren}}_{\pZ_K}\] of factorization algebras. Here $\IC^{\operatorname{ren}}_{\pZ_K}$ is a factorization algebra equivalent to $\IC_{\pZ_K}$ up to a cohomological shift.
\end{theorem}

Moreover, we will show in Theorem \ref{recovery} that the sheaf $\sIC$, and by extension the sheaf $\sICR$, can be recovered from $\IC_{\pZ_K}$. More precisely, we will construct a monad $\mathscr{M}_K$ acting on the category $\fD(\pZ_K)$ together with an equivalence \[\mathscr{M}_K -\operatorname{mod}(\fD(\pZ_K)) \overset{\sim}{\longrightarrow} \fD_{\eff}(\pZ)\] which sends (a shift of) $\IC_{\pZ_K}$ to the Zastava semi-infinite IC sheaf. 

\subsubsection{} We will end with an application of the methods developed in this paper to parabolic subgroups of $G$ as well as to Kac-Moody algebras. Namely, we will define in Definition \ref{parabolicicdefinition} analogues of $\sICR$ for every proper parabolic subgroup $P$ of $G$ which recovers the usual factorization algebra when $P=B$. We note also that a (non-factorizable) parabolic semi-infinite IC sheaf has been constructed in \cite{DL}. By the local-to-global compatibility of Corollary \ref{localtoglobalrancomp} and Theorem $4.2.2$ of \cite{DL}, we see that our sheaf recovers the Dhillon-Lysenko sheaf after restriction to a fiber $\Gr_{G,x}$ of $\Gr_{G,\Ran}$. 

In Section \ref{quivericsheaves} we will discuss applications to Kac-Moody algebras. Namely, we will review the theory of Mirkovi\'{c} Zastava spaces \cite{MYZ} and their $\G_m$-action and produce a semi-infinite IC sheaf for each one. By taking the underlying graph to be an affine ADE quiver we obtain a semi-infinite IC sheaf associated to each Kac-Moody algebra of type ADE. 

%Lastly, in Section \ref{snops} we will make contact with \cite{FM}. For a marked point $x \in X$, we will construct a category $\Sigma(G)_x$ of factorization modules for $\sIC$, which we also call \emph{snops}. We will show that $\Sigma(G)_x$ is naturally equivalent to the category of unital factorization modules for $\sICR$ at $x$. 

\subsection{Acknowledgments.} I wish to thank my graduate advisor Ivan Mirkovi\'{c} for his continued support and encouragement during my dissertation (for which part of this text was written), as well as for introducing me to many beautiful ideas in mathematics. I would also like to thank Tom Braden, Justin Campbell, Michael Finkelberg, Dennis Gaitsgory, Owen Gwilliam, Sergey Lysenko, and Alexei Oblomkov for many useful and interesting discussions, some of which led to the ideas developed in this paper. Special thanks are due to Joakim Faergeman and Sam Raskin, both of whom contributed crucial pieces of technical help in the paper in addition to many interesting conversations. In particular, I learned the argument for conservativity in Lemma \ref{conserv} from Joakim, and Sam taught me the proof of Lemma \ref{dmpushforward}. I would also like to thank the anonymous referee for their careful reading of the text as well as their helpful comments. Part of this work was completed while being supported by the Max Planck Institute for Mathematics in Bonn, Germany. 

\section{Lax prestacks and their sheaves}\label{laxprstk}

In this section we will review the concept of \emph{lax prestack}. For a general survey on lax prestacks see \cite{R1}. Let $\Aff$ denote the category of (non-derived) affine schemes over $k$, and let $\Cat$ denote the $\infty$-category of $\infty$-categories. 

\subsection{Lax prestacks.} A \emph{lax prestack} is a functor $\Aff^{op} \to \Cat$. The category $\Fun(\Aff^{op}, \Cat)$ of lax prestacks  will be denoted $\LStk$. 

\begin{remark} In this paper most functors $\Aff^{op} \to \Cat$ we consider will take values in (nerves of) ordinary categories. Nevertheless, it is still desirable to fit these objects into the broader framework of higher category theory, in particular to have a well-behaved theory of sheaves. 
\end{remark}

Let $\Spc$ denote the $\infty$-category of $\infty$-groupoids. Then there is a natural inclusion $\Spc \to \Cat$ and so every prestack $\Aff^{op} \to \Spc$ may be thought of as a particular kind of lax prestack. We will denote the left adjoint  to the inclusion $\Stk \to \LStk$ by $\str: \LStk \to \Stk$. Roughly, for an affine scheme $S$, the space $\cY_{\str}(S) \coloneqq \str(\cY)(S)$ is the universal groupoid admitting a map from $\cY(S)$. 

Here is an equivalent formulation of $\cY_{\str}$. There is a fully faithful functor $\Cat \to \Fun(\Delta^{op}, \Spc)$, whose essential image consists of the complete Segal spaces. Hence to a lax prestack $\cY$ we may associate a simplicial prestack $s\cY: \Aff^{op} \to \Fun(\Delta^{op},\Spc)$ satisfying the Segal conditions. The following lemma gives a more concrete construction of $\cY_{\str}$.

\begin{lemma}
\label{str}
Let $\cY$ be a lax prestack and let $s\cY$ be the associated simplicial prestack. Then there is a canonical equivalence of prestacks
\[\cY_{\str} \longrightarrow |s\cY|,\]
where $|s\cY|$ is the prestack obtained from $s\cY$ by composing with the geometric realization functor $\Fun(\Delta^{op}, \Spc) \to \Spc$. 
\end{lemma}
\begin{proof}
This follows from the fact that the straightening functor $\Cat \to \Spc$ is equivalent to the composition \[\Cat \longrightarrow \Fun(\Delta^{op},\Spc) \longrightarrow \Spc,\]
where the first arrow is the functor taking a category to its associated complete Segal space, and the second arrow is geometric realization. 
\end{proof}

Recall that a category $\mathscr{C}$ is \emph{weakly contractible} if $\mathscr{C}_{\str}$ is equivalent to a point in $\Spc$. We will say that a lax prestack $\cY$ is weakly contractible if for all affine schemes $S$, the category $\cY(S)$ is weakly contractible. 

\begin{remark}
The straightening functor does not preserve homotopy groups in general. This means that even if $\cY$ is a lax prestack taking values in ordinary categories, the straightening $\cY_{\str}$ may take values in higher groupoids. 
\end{remark}

Consider the following example of a lax prestack, which will be important throughout the text. Let $\mathcal{M}$ be a presheaf of commutative monoids on $\Aff_k$. That is, for every affine scheme $S$ the groupoid $\mathcal{M}(S)$ is a set, and $\mathcal{M}$ has a factorization through the forgetful functor from commutative monoids to sets. 

The functor $\mathcal{M}$ can be upgraded to a lax prestack as follows. Define $\mathcal{M}^{\to}(S)$ to be the category whose objects are given by the set $\mathcal{M}(S)$ and such that there is a unique arrow $x \to y$ for every $m$ such that $mx=y$. For every morphism $S \to T$ of affine schemes, the morphism $\mathcal{M}(T) \to \mathcal{M}(S)$ of monoids given by restriction can be upgraded to a functor \[\mathcal{M}^{\to}(T) \longrightarrow \mathcal{M}^{\to}(S)\] in the obvious way. It is easy to see that the associated simplicial prestack is given by the simplicial object \[\begin{tikzcd}
	\ldots & {\mathcal{M} \times \mathcal{M} \times \mathcal{M}} & {\mathcal{M} \times \mathcal{M}} & {\mathcal{M}}
	\arrow[shift left=1, from=1-3, to=1-4]
	\arrow[shift right=1, from=1-3, to=1-4]
	\arrow[shift left=2, from=1-2, to=1-3]
	\arrow[shift right=2, from=1-2, to=1-3]
	\arrow[from=1-2, to=1-3]
\end{tikzcd}\] induced by multiplication in $\mathcal{M}$. 

\begin{remark} Note that $\mathcal{M}^{\to}$ is weakly contractible if $\mathcal{M}$ is a presheaf of \emph{cancellative} monoids. Explicitly, $\mathcal{M}$ is cancellative if for all $x,y,z \in \mathcal{M}(S)$ if $xy=xz$ then $y=z$. In this case, $\mathcal{M}^{\to}$ is a functor of filtered categories, and all such are known to be weakly contractible. \end{remark}

\subsection{Sheaves on lax prestacks.} We will now briefly review sheaves on lax prestacks \cite{R1}. Let us assume we are given a \emph{sheaf theory} $\Shv: \Aff^{op} \to \Cat$. For all intents and purposes $\Shv(S)$ will be the category of quasicoherent sheaves, ind-coherent sheaves, or $D$-modules (see \cite{Ga-.5} for a survey of ind-coherent sheaves and \cite{GR} for a survey of $D$-modules). For a map $f: S \to T$ of affine schemes, denote the resulting functor between categories of sheaves by $f^!$. We can then define the category of sheaves on an arbitrary lax prestack.

\begin{definition}
Let $\cY$ be a lax prestack. Then the category $\Shv(\cY)$ of \emph{sheaves on} $\cY$ is the category of natural transformations $\cY \to \Shv$. In particular, the assignment $\cY \mapsto \Shv(\cY)$ takes arbitrary colimits to limits. 
\end{definition}

Unwinding the definition of $\Shv(\cY)$, we see that a sheaf $c$ on $\cY$ is the following data.
\begin{enumerate}
    \item For every $S$-point $f:S \to \cY$ of $\cY$ we have a sheaf $c_{S,f} \in \Shv(S)$.
    \item For every map $\varphi: (S,f) \to (T,g)$ of affine schemes over $\cY$, we have an isomorphisms $\varphi^!(c_{T,g}) \simeq c_{S,f}$ with homotopy coherent compatibilities. 
    \item For every morphism $f \to g$ of $S$-points, we have morphisms $c_{S,f} \to c_{S,g}$ with homotopy coherent compatibilities. 
\end{enumerate}

By definition of $\cY_{\str}$, it follows that the category $\Shv(\cY_{\str})$ is equivalent to the full subcategory $\Shv_{\str}(\cY)$ of sheaves on $\cY$ where all the morphisms $c_{S,f} \to c_{S,g}$ in $3$ are isomorphisms. Denote the resulting full subcategory of $\Shv(\cY)$ by $\Shv_{\str}(\cY)$. There is the following equivalent simplicial perspective on $\Shv_{\str}(\cY)$. 

\begin{lemma}
\label{strcat}
Let $\cY$ be a lax prestack, and let $s\cY$ be the associated simplicial prestack. Then there is an equivalence of categories
\[\Shv_{\str}(\cY) \longrightarrow \Tot\big(\Shv(s\cY)\big), \]
where $\Tot\big(\Shv(s\cY)\big)$ denotes the totalization of the cosimplicial category $\Shv(s\cY)$.
\end{lemma}
\begin{proof}
Since $\Shv_{\str}(\cY)$ is equivalent to $\Shv(\cY_{\str})$, it suffices to show the claim with the left hand side replaced by the latter. But by Lemma \ref{str} we have a canonical isomorphism $\cY_{\str} \simeq |s\cY|$, and so the claim follows from the fact that $\Shv$ sends arbitrary colimits of prestacks to limits. 
\end{proof}

\subsection{Left fibrations of lax prestacks.} Recall that a functor $F: \mathscr{C} \to \mathscr{D}$ of $\infty$-categories is a \emph{left fibration} if it satisfies the right lifting property with respect to all horn inclusions, except possibly the right outer ones. That is, whenever the outer square of \[\begin{tikzcd}
	{\Lambda^k[n]} & {\mathscr{C}} \\
	{\Delta[n]} & {\mathscr{D}}
	\arrow[from=1-1, to=1-2]
	\arrow["F", from=1-2, to=2-2]
	\arrow[from=1-1, to=2-1]
	\arrow[from=2-1, to=2-2]
	\arrow[dashed, from=2-1, to=1-2]
\end{tikzcd}\]commutes, there exists a dotted arrow as above for all $0 \leq k < n$ making the diagram commute. In higher category theory, the left fibrations play a role similar to that of categories cofibered in groupoids in ordinary category theory. In fact, any left fibration is also a coCartesian fibration, and if the base is a Kan complex then it is also a Kan fibration. 

%We will call a map $\varphi: \cY \to \cG$ of lax prestacks a \emph{left fibration} if for every affine scheme $S$, the functor $\varphi_S: \cY(S) \to \cG(S)$ is a left fibration of $\infty$-categories and for every morphism $S \to T$ the map $\cG(T) \to \cG(S)$ sends coCartesian arrows to coCartesian arrows. More generally, we will call $\varphi$ a \emph{coCartesian fibration} (resp. \emph{Cartesian fibration}) if $\varphi_S$ is a coCartesian fibration (resp. Cartesian fibration). 

\subsubsection{} We will define here an analogue of left fibrations for lax prestacks. Consider the lax prestack $\Spc_{/-}$ which sends an affine scheme to the slice category $\Stk_{/S}$. For a lax prestack $\cY$, define $\Spc_{/\cY}$ to be the category of natural transformations \[ \Phi: \cY \longrightarrow \Spc_{/-}\] One should think of $\Phi$ as assigning to an $S$-point $f: S \to \cY$ the fiber $\Phi_f$ of the associated left fibration. For any morphism $f \to g$ of $S$-points, the structure of natural transformation $\cY \to \Spc_{/-}$ gives a morphism \[\Phi_f \longrightarrow \Phi_g\] as prestacks over $S$. By passing to global sections and using the Grothendieck construction, for each $\Phi$ there is an associated lax prestack $\Phi_{\Spc} \to \cY$ lying over $\cY$. For an affine scheme $S$, the functor $\Phi_{\Spc}(S) \to \cY(S)$ is a left fibration of categories.

We will refer to a natural transformation $\Phi: \cY \to \Spc/-$ as a \emph{left fibration of lax prestacks} or simply as a \emph{left fibration}. Since left fibrations are examples of coCartesian fibrations, we will also sometimes call such functors $\Phi$ \emph{coCartesian fibrations of lax prestacks}. Note that we have the lax prestack $\cY^{\operatorname{op}}$ taking an affine scheme $S$ to $\cY(S)^{\operatorname{op}}$. Hence we also have the notion of \emph{right fibration of lax prestacks} or \emph{Cartesian fibrations of lax prestacks}. Additionally, one can define a \emph{pointwise (co)Cartesian fibration of lax prestacks} which is a map $\cX \to \cY$ of lax prestacks such that for all affine schemes $S$ the induced functor \[ \cX(S) \longrightarrow \cY(S)\] is a (co)Cartesian fibration and for any map $T \to S$ the induced map $\cX(S) \to \cX(T)$ sends (co)Cartesian arrows to (co)Cartesian arrows. Although for $\Phi$ as above the map $\Phi_{\Spc} \to \cY$ is a pointwise coCartesian fibration, it is not clear that the converse holds. 

An advantage of this perspective is that one can immediately see how to assign to a lax prestack $\cY$ a sheaf of categories, given the existence of such a $\Phi$. Recall that a sheaf of categories is defined to be a lax natural transformation \[\cY \longrightarrow \operatorname{ShvCat}_/\] where $\operatorname{ShvCat}$ is the lax prestack which assigns to an affine scheme $S$ the category of modules for the symmetric monoidal category $\QC(S)$. Note there is a natural transformation \[\fD^{\Spc \to \Cat}: \Spc_{/-} \longrightarrow \operatorname{ShvCat}_{/-}\] which takes a prestack $\cX \to S$ over an affine scheme to $\fD(\cX)$. Hence to any left fibration $\Phi$ over $\cY$ we may assign a sheaf of categories given as the composition $\fD^{\Spc \to \Cat}(\Phi) \coloneqq \fD^{\Spc \to \Cat} \circ \Phi$.

Note whenever there is a map $f:\cX \to \cY$ of lax prestacks, we can define a pullback functor \[f^{\Spc,*}: \Spc_{/\cY} \to \Spc_{/\cX}\] simply by composing a lax natural transformation $\cY \to \Spc_/$ with $f$. We therefore obtain a functor \[\LStk^{\op} \longrightarrow \Cat\] which takes a lax prestack $\cY$ to $\Spc_{/\cY}$. 

\subsection{Correspondences.} In this section we will discuss categories of correspondences. These will be important in defining notions of factorization spaces and their variants (e.g. unital, effective). Of course, the definition of factorization spaces is well-known, and we are not adding anything new to the theory here. However, since we are considering factorization over several different bases, it is convenient to have a unified language to discuss all of them. 

Let $\mathscr{C}$ be a category with all finite limits, and let $\Corr(\mathscr{C})$ be the $2$-\emph{category of correspondences}\footnote{We will mostly sweep all $2$-categorical complexities under the rug. The ideas can all be made precise, however, using the approach in \cite{R1}, mutatis mutandis.} in $\mathscr{C}$. Objects in $\Corr(\mathscr{C})$ are just objects of $\mathscr{C}$, but morphisms $\cY \to \cG$ are given by correspondences
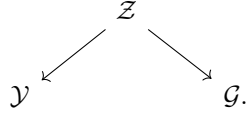
\[\begin{tikzcd}
	& \cZ \\
	\cY && \cG.
	\arrow[from=1-2, to=2-1]
	\arrow[from=1-2, to=2-3]
\end{tikzcd}\] Composition of morphisms $\cY \to \cG \to \mathcal{H}$ is given by the outer legs of the diagram 
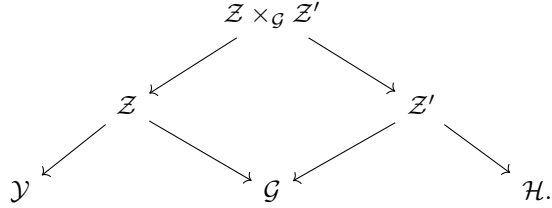
\[\begin{tikzcd}
	&& {\cZ \times_{\cG} \cZ'} \\
	& \cZ && {\cZ'} \\
	\cY && \cG && \mathcal{H}.
	\arrow[from=2-2, to=3-1]
	\arrow[from=2-2, to=3-3]
	\arrow[from=2-4, to=3-3]
	\arrow[from=2-4, to=3-5]
	\arrow[from=1-3, to=2-2]
	\arrow[from=1-3, to=2-4]
\end{tikzcd}\]Note that $\Corr(\mathscr{C})$ is naturally symmetric monoidal under the Cartesian monoidal structure. Hence we may consider commutative algebras (aka commutative monoids) in $\Corr(\mathscr{C})$. Roughly, a commutative algebra in $\Corr(\mathscr{C})$ is a pair of correspondences 
\begin{equation}\label{algebraincorr}\begin{tikzcd}
	& {\operatorname{mult}_{\cY}} &&&& {\operatorname{unit}_{\cY}} \\
	{\cY \times \cY} && \cY && \pt && \cY
	\arrow[from=1-2, to=2-1]
	\arrow[from=1-2, to=2-3]
	\arrow[from=1-6, to=2-5]
	\arrow[ from=1-6, to=2-7]
\end{tikzcd}\end{equation} satisfying natural associativity and unitality conditions up to coherent homotopy. 

We will apply the above picture to both $\Stk$ and $\LStk$ to obtain the categories $\Corr(\Stk)$ and $\Corr(\LStk)$. Note that $\Spc_{/(-)}$ is naturally a commutative algebra in lax prestacks, where multiplication is induced by the functors \[\Spc_{/S} \times \Spc_{/S} \longrightarrow \Spc_{/S}\] taking a pair of prestacks $(\cY,\cX)$ to the fiber product $\cY \times_S \cX$. The unit is provided by the inclusion of the identity $S \to S$ as an object of $\Spc_{/S}$. 

%Note that the constant lax prestack $\underline{\Spc}$ is naturally a commutative algebra in $\Corr(\LStk)$ with correspondences given by \[\begin{tikzcd}
	%& {\underline{\Spc} \times \underline{\Spc}} &&& \pt \\
	%{\underline{\Spc} \times \underline{\Spc}} && \underline{\Spc} & \pt && \underline{\Spc}
	%\arrow["{\operatorname{id}_{\underline{\Spc}} \times \operatorname{id}_{\underline{\Spc}}}"', from=1-2, to=2-1]
	%\arrow["{-\times -}", from=1-2, to=2-3]
	%\arrow[from=1-5, to=2-4]
	%\arrow[from=1-5, to=2-6]
%\end{tikzcd}\]where $- \times -$ denotes Cartesian product and the map $\pt \to \underline{\Spc}$ is the inclusion of the contractible space. 

\subsection{Multiplicative spaces.} In this section, we will review the category of multiplicative spaces over a commutative algebra in correspondences. A general survey of the notion of multiplicative spaces can be found in \cite{Bu}. As we will later see, when the commutative algebra is the Ran space or the configuration space, multiplicative spaces will coincide with the familiar notion of factorization spaces. Our definition of multiplicative space proceeds by adapting Raskin's construction of multiplicative sheaves of categories to a geometric context.

Denote by \[\Spc_{/(-)}: \LStk^{\op} \longrightarrow \Cat\] the functor taking a lax prestack $\cY$ to its category $\Spc_{/\cY}$ of left fibrations over $\cY$. As in \emph{Chiral categories}, we can construct the category $\operatorname{Groth}(\Spc_{/(-)})$ whose objects are pairs $(\cY,\Phi)$, where $\cY$ is a lax prestack and $\Phi: \cY \to \Spc_{/(-)}$ is a left fibration over $\cY$. Morphisms $(\cY,\Phi) \to (\cX, \Psi)$ are given by a correspondence \[\begin{tikzcd}
	& \cZ \\
	\cY && \cX
	\arrow["g", from=1-2, to=2-3]
	\arrow["f"', from=1-2, to=2-1]
\end{tikzcd}\] together with a map \[f^{\Spc,*}(\Phi) \to g^{\Spc,*}(\Psi)\] satisfying various homotopy coherent compatibilities.

Since $\Spc_{/(-)}$ is lax symmetric monoidal with respect to the Cartesian monoidal structure on both sides, the category $\operatorname{Groth}(\Spc_{/(-)})$ inherits a symmetric monoidal structure $\otimes$, given pointwise on objects by \[(\cY,\Phi) \otimes (\cX,\Psi) =(\cY \times \cX, \Phi \times \Psi).\] There is a natural forgetful functor \[\oblv_{\Spc}: \operatorname{Groth}(\Spc_{/(-)}) \longrightarrow \Corr(\LStk)\] which has the structure of a symmetric monoidal functor. 

Now let $\cY$ be a commutative algebra in $\Corr(\LStk)$. Define a \emph{weakly multiplicative space} over $\cY$ to be a commutative algebra object of $\operatorname{Groth}(\Spc_{/(-)})$ mapping to $\cY$ under $\oblv_{\Spc}$ as a commutative algebra. Note that, among other features, a weakly multiplicative space on $\cY$ is a left fibration $\Phi$ over $\cY$ together with a morphism \begin{equation}\label{1operation}(\iota_{\cY})^*(\Phi \times \Phi) \longrightarrow (\mu_{\cY})^*(\Phi)\end{equation} where $\iota_{\cY}$ and $\mu_{\cY}$ are the left and right legs, respectively, of the multiplication correspondence in \eqref{algebraincorr}. We also have analogous maps for all $n$-ary operations of $\cY$. 

\begin{definition} Let $\cY$ be a commutative algebra in correspondences of lax prestacks. Define a \emph{multiplicative space} over $\cY$ to be a weakly multiplicative space such that all $n$-ary operations are isomorphisms. In particular, the map \eqref{1operation} is an isomorphism. Let $\operatorname{MultSpc}(\cY)$ denote the category of multiplicative spaces over $\cY$. 
\end{definition}

Let us unwind this definition a bit more. By the Grothendieck construction, a left fibration $\Phi: \cY \to \Spc_/$ corresponds to a particular kind of lax prestack $\Phi_{\tau}$ with a map \[\Phi^{\Groth}: \Phi_{\tau} \to \cY,\] where here the notation $\tau$ stands for \emph{total space}, i.e. the total space of the fibration $\Phi$. Then if $\Phi$ has the structure of a multiplicative space over $\cY$ we have isomorphisms \[(\Phi_{\tau} \times \Phi_{\tau}) \times_{\cY \times \cY} \operatorname{mult}_{\cY} \overset{\sim}{\longrightarrow} \Phi_{\tau} \times_{\cY} \operatorname{mult}_{\cY}\] satisfying various commutativity and associativity compatibilities.

\subsection{Compatibility of straightening with algebras.} In this section, we will discuss when it is possible to push a multiplicative space forward along the canonical map $\cY \to \cY_{\str}$ for a commutative algebra $\cY$ in $\Corr(\LStk)$.

Note that the straightening functor \[\str: \LStk \longrightarrow \Stk\] preserves products but not fiber products in general. As such, it is not clear how to upgrade $\str$ to a functor between categories of correspondences. To remedy this, we will restrict the allowable morphisms in $\LStk$. 

Recall that a functor $\mathscr{C} \to \mathscr{D}$ between categories $\mathscr{C}$ and $\mathscr{D}$ is called a \emph{realization fibration} if for all functors $\mathscr{E} \to \mathscr{D}$ the diagram \[\begin{tikzcd}
	{(\mathscr{C} \times_{\mathscr{D}} \mathscr{E})_{\str}} & {\mathscr{E}_{\str}} \\
	{\mathscr{C}_{\str}} & {\mathscr{D}_{\str}}
	\arrow[from=1-1, to=1-2]
	\arrow[from=1-2, to=2-2]
	\arrow[from=1-1, to=2-1]
	\arrow[from=2-1, to=2-2]
\end{tikzcd}\] is Cartesian. We will call a map $\cX \to \cY$ of lax prestacks a \emph{realization fibration} if for all affine schemes $S$ the functor \[\cX(S) \longrightarrow \cY(S)\] is a realization fibration of categories. Denote by $\LStk_{\operatorname{rf}}$ the subcategory of $\LStk$ consisting of lax prestacks and realization fibrations as morphisms. Note that, by an application of the pasting lemma, realization fibrations are stable under base change and hence it makes sense to consider the category $\Corr(\LStk_{\operatorname{rf}})$.

Since the functor $\str$ preserves fiber products when restricted to $\LStk_{\operatorname{rf}}$, we may upgrade $(-)_{\str}$ to a functor \[\operatorname{\str}^{\operatorname{enh}}: \operatorname{CAlg}\big(\Corr(\LStk_{\operatorname{rf}})\big) \longrightarrow \operatorname{CAlg}\big(\Corr(\Stk)\big),\] where for a category $\mathscr{C}$ with all finite limits we define $\operatorname{CAlg}(\mathscr{C})$ to be the category of commutative algebras in $\mathscr{C}$. In other words, given a commutative algebra $\cY$ in correspondences of lax prestacks where each leg of the multiplication and unit correspondences is a realization fibration, the straightening $\cY_{\str}$ has a natural structure of a commutative algebra in correspondences of prestacks. 

\begin{proposition}\label{multspctomultspc} Let $\cY$ be a commutative algebra in correspondences such that all maps in the diagrams \eqref{algebraincorr} are realization fibrations. Then there is a natural functor \[\operatorname{LKE_{\cY \to \cY_{\str}}}: \operatorname{MultSpc}(\cY) \longrightarrow \operatorname{MultSpc}\big(\str^{\operatorname{enh}}(\cY)\big)\] taking multiplicative spaces over $\cY$ to multiplicative spaces over the commutative algebra $\str^{\operatorname{enh}}(\cY)$. 
\end{proposition}
\begin{proof} Let $\Phi: \cY \to \Spc_{/(-)}$ be a multiplicative space over $\cY$, and denote also by $\Phi_{\tau}$ the total space of the associated left fibration over $\cY$. Define $\operatorname{LKE}_{\cY \to \cY_{\str}}(\Phi)$ by left Kan extension as follows: \[\begin{tikzcd}
	{\cY} & {\Spc_{(-)}} \\
	{\cY_{\str}}
	\arrow["\Phi", from=1-1, to=1-2]
	\arrow[from=1-1, to=2-1]
	\arrow["{\operatorname{LKE}_{\cY \to \cY_{\str}}}(\Phi)"', dashed, from=2-1, to=1-2]
\end{tikzcd}\]where by left Kan extension we mean the functor $\cY_{\str} \to \Spc_/$ which on an affine scheme $S$ is the left Kan extension of $\cY(S) \to \Spc_{/S}$ along $\cY(S) \to \cY_{\str}(S)$. Note that this gives a well-defined morphism of lax prestacks because $\Stk$ is locally Cartesian closed. 

We need to show that the resulting diagram \begin{equation}\label{commute}\begin{tikzcd}
	& {\operatorname{mult}_{\cY_{\str}}} \\
	{\cY_{\str} \times \cY_{\str}} & {\Spc_{/(-)} \times \Spc_{/(-)}} & {\cY_{\str}} \\
	{\Spc_{/(-)} \times \Spc_{/(-)}} && {\Spc_{/(-)}}
	\arrow[from=1-2, to=2-3]
	\arrow["{\operatorname{LKE}_{\cY \to \cY_{\str}}}", from=2-3, to=3-3]
	\arrow[from=1-2, to=2-1]
	\arrow[from=1-2, to=2-2]
	\arrow[from=2-2, to=3-3]
	\arrow[from=2-2, to=3-1]
	\arrow[from=2-1, to=3-1]
\end{tikzcd}\end{equation} and all the analogous ones for higher operations commute (and similarly for the unit diagram, which we omit). By the universal property of Kan extensions, we know that \eqref{commute} commutes up to a possibly non-invertible $2$-morphism. We wish to show that this $2$-morphism is an equivalence.

Let $(\Phi_{\tau})_{\str} \to \cY_{\str}$ denote the left fibration associated to $\operatorname{LKE}_{\cY \to \cY_{\str}}$ via the Grothendieck construction. Equivalently, $(\Phi_{\tau})_{\str}$ is the straightening of the lax prestack $\Phi_{\tau}$. To show that \eqref{commute} commutes, it suffices to show that the corresponding map \[((\Phi_{\tau})_{\str} \times (\Phi_{\tau})_{\str}) \times_{\cY_{\str} \times \cY_{\str}} \operatorname{mult}_{\cY_{\str}} \longrightarrow (\Phi_{\tau})_{\str} \times_{\cY_{\str}} \operatorname{mult}_{\cY_{\str}}\] is an equivalence, and similarly for the higher operations. But this follows from the definitions together with our assumption on $\cY$. 
\end{proof}

\section{Bases of factorization}\label{factorization}

In this section we will introduce the bases over which factorization will occur. Namely, we consider the moduli of finite subsets of our curve $X$, as well as a graded version in which we remember multiplicity. We will define what it means to have a factorization structure over both. Later, we will show how to pass from factorization structures over the configuration space to factorization structures over the Ran space. This will allow us to reconstruct sheaves on the affine Grassmannian from sheaves on Zastava spaces. 

\subsection{The Ran space.} Define a prestack $\Ran$ as follows (see \cite{R1} for a more detailed overview). For an affine scheme $S$, a map $S \to \Ran$ is a nonempty finite subset of $\Map(S,X)$. Equivalently, we can write \[\Ran \overset{\sim}{\longrightarrow} \colim_{\operatorname{FinSet}^{\op}_{\operatorname{surj}}} X^I \] where the colimit is over the category opposite to the one whose objects are (possibly empty) finite sets and whose morphisms are surjective functions. Here $X^I$ denotes the scheme of maps $x_I : I \to X$ from a finite set $I$ to $X$. 

For a point $x$ of $X$, define $\Gamma_x$ to be its graph. Accordingly, for a point $x_I=(x_i)_{i \in I}$ of $\Ran$, write $\Gamma_{x_I}$ for the union $\cup_{i \in I} \Gamma_{x_i}$ of graphs. Inside the product $\Ran \times \Ran$ we have the subprestack $[\Ran \times \Ran]_{\disj}$ given by pairs $(x_I, x_J)$ such that $\Gamma_{x_I} \cap \Gamma_{x_J} = \emptyset$. Note that $\Ran$ is equipped with the structure of a monoid \[\add_{\Ran} : \Ran \times \Ran \to \Ran\] given by union of finite subsets. By abuse of notation, we will denote by $\add_{\Ran}$ the restriction of the same named map to $[\Ran \times \Ran]_{\disj}$. 

\subsubsection{} An important feature of $\Ran$ is that it has a structure of a commutative algebra in correspondences. Namely, there is a correspondence 
\begin{equation}
\label{rancorr}
\begin{tikzcd}
	& {[\Ran \times \Ran]_{\disj}} \\
	{\Ran \times \Ran} && \Ran && {}
	\arrow[from=1-2, to=2-1]
	\arrow[from=1-2, to=2-3]
\end{tikzcd}
\end{equation} whose second leg is given $\add_{\Ran}$. The unit diagram for $\Ran$ is just given by the inclusion \[\{\emptyset\} \longhookrightarrow \Ran\] of the empty subset of $X$. 

Define a \emph{factorization space} over $\Ran$ to be a multiplicative space for the correspondence \eqref{rancorr}. A bit more explicitly, a factorization space is a prestack $\cY \to \Ran$ together with unital and associative isomorphisms \[\fact_{\cY}: (\cY \times \cY) \times_{\Ran \times \Ran} [\Ran \times \Ran]_{\disj} \overset{\sim}{\longrightarrow} \cY \times_{\Ran} [\Ran \times \Ran]_{\disj}.\] We will often denote the left hand side by $[\cY \times \cY]_{\disj}$ and the right hand side by $\cY_{\disj}$. 

\subsection{The unital Ran space.} We will also need an enhancement of $\Ran$ as a lax prestack. Define \[\Ran^{\un}: \Aff_k^{\op} \longrightarrow \Cat\] to be the functor sending $S$ to the following category. Objects in $\Ran^{\un}(S)$ are given by finite subsets $x_I \subseteq \Map(S,X)$ and for any such pair $x_I$ and $x_J$ there is a single morphism $x_I \to x_J$ precisely if $x_I \subseteq x_J$. Note we are not requiring that $x_I$ be nonempty. We call the lax prestack $\Ran^{\un}$ the \emph{unital Ran space}. 

As in the case of the usual Ran space, there is a correspondence 
\[\begin{tikzcd}
	& {[\Ran^{\un} \times \Ran^{\un}]_{\disj}} \\
	{\Ran^{\un} \times \Ran^{\un}} && \Ran^{\un} && {}
	\arrow[from=1-2, to=2-1]
	\arrow[from=1-2, to=2-3]
\end{tikzcd} \] which makes $\Ran^{un}$ into a commutative algebra in correspondences of lax prestacks. Define a \emph{unital factorization space} to be a multiplicative space over $\Ran^{\un}$.

More explicitly, we may think of a unital factorization space as a \emph{prestack} $\cY \to \Ran$ together with associative and unital maps $\cY_{x_I} \to \cY_{x_J}$ between fibers whenever we have an inclusion $x_I \subseteq x_J$. By abuse of notation, we will often call $\cY$ a unital factorization space when the associated structure of a left fibration is clear from context. 

\subsubsection{} Let $\cY^{\un} \to \Ran^{\un}$ be a unital factorization space, and let $\cY$ be the underlying ordinary factorization space. Since $\cY^{\un}$ is a lax prestack, we may consider the category $\fD_{\str}(\cY^{\un})$ of \emph{unital D-modules} on $\cY$. There is a natural forgetful functor $\oblv_{\cY}: \fD_{\str}(\cY^{\un}) \to \fD(\cY)$ obtained by $!$-pull back along the composition \[\cY \to \cY^{\un} \to \cY_{\str}\] where the first map is the projection $\cY= \cY^{\un} \times_{\Ran^{\un}} \Ran \to \cY^{\un}$. The following proposition is standard. 

\begin{proposition} The forgetful functor \[ \oblv_{\cY}: \fD_{\str}(\cY^{\un}) \longrightarrow \fD(\cY)\] is fully faithful.
\end{proposition}
\begin{proof} This follows directly from the homological contractibility of $\Ran$ proved in \cite{Ga-2}. 
\end{proof}

Hence we may call $\fD_{\str}(\cY^{\un})$ the unital \emph{subcategory} of $\fD(\cY)$, and often denote it by $\fD_{\untl}(\cY)$. 

\subsection{An alternative description of the unital category.} A unital factorization space $\cY^{\un}$ is determined by an ordinary factorization space $\cY \to \Ran$ together with an action \[ \act_{\cY}: \Ran \times \cY \to \cY \] which extends the natural map over the disjoint locus provided by factorization. In particular, we obtain a simplicial prestack \[ \cY^{\bullet} \coloneqq \, \ldots \begin{tikzcd}
	{\Ran \times \Ran \times \cY} & {\Ran \times \cY} & \cY
	\arrow["{\operatorname{pr_{\cY}}}"', shift right=2, from=1-2, to=1-3]
	\arrow["{\act_{\cY}}", shift left=2, from=1-2, to=1-3]
	\arrow[shift right=3, from=1-1, to=1-2]
	\arrow[shift left=3, from=1-1, to=1-2]
	\arrow[from=1-1, to=1-2]
\end{tikzcd}\]
whose associated geometric realization $|\cY^{\bullet}|$ is canonically isomorphic to $\cY^{\un}_{\str}$. It follows that there is a canonical equivalence of categories \[\Tot(\fD(\cY^{\bullet})) \overset{\sim}{\longrightarrow} \fD_{\str}(\cY^{\un})\] which is compatible with factorization. One should think of $\Tot(\fD(\cY^{\bullet}))$ as a category of $\Ran$-equivariant sheaves on $\cY$. 

Note that by homological contractibility of Ran, the category $\Tot(\fD(\cY^{\bullet}))$ may be explicitly described as follows. It is the full subcategory of $\fD(\cY)$ of sheaves $c$ such that \emph{there exists} a (necessarily unique) isomorphism \[\operatorname{pr}_{\cY}^!(c) \overset{\sim}{\longrightarrow} \act_{\cY}^!(c),\] i.e. equivariance is a property rather than an extra structure. 

\subsection{The configuration space.} We will now discuss a ``graded" analogue of $\Ran$ (see \cite{R2}). Let $\Div^{-\eff}$ denote the moduli of \emph{anti}-effective divisors on $X$ and let $\check{\Lambda}^+$ be the monoid of dominant weights of $G$. Define the \emph{configuration space} $\Conf$ as the space of homomorphisms $\check{\Lambda}^+ \to \Div^{-\eff}$. Equivalently, $\Conf$ is the moduli of sums $\sum_j \lambda_j x_j$ where $\lambda_j \in \Lambda^-$ is a negative coweight for each $j$ and the $x_j$'s are pairwise distinct. We will refer to the points of $\Conf$ as \emph{colored divisors} on $X$. 

The prestack $\Conf$ is actually a scheme with $\pi_0(\Conf)= \Lambda^-$ where each connected component $\Conf^{\lambda}$ consists of colored divisors with total degree $\lambda$. Explicitly, a choice of simple coroots determines an isomorphism \[ \Conf^{\lambda} \overset{\sim}{\longrightarrow} \prod_{\alpha_i \in \mathcal{I}_G} X^{(n_i)}\] where $X^{(n_i)}= X^{n_i}//S_{n_i}$ is the GIT quotient of $X^{n_i}$ by the symmetric group. We will often denote $\Conf^{\lambda}$ by $X^{\lambda}$ and call it a \emph{partially symmetrized power} of $X$. 

\subsubsection{} Consider the incidence divisor in the product $X \times \Conf$. Given an $S$-point $D: S \to \Conf$, define $\Gamma_D$ to be the relative effective divisor in $X \times S$, equal to the pullback along \[\operatorname{pr_X} \times D: X \times S \to X \times \Conf\] of the incidence divisor. Here $\operatorname{pr}_X$ is the projection onto $X$. As in the case of the Ran space, there is a subscheme $[\Conf \times \Conf]_{\disj}$ of $\Conf \times \Conf$ consisting of pairs $(D,E)$ such that $\Gamma_D \cap \Gamma_E = \emptyset$. As before, there is a correspondence
\[\begin{tikzcd}
	& {[\Conf \times \Conf]_{\disj}} \\
	{\Conf \times \Conf} && \Conf && {}
	\arrow[from=1-2, to=2-1]
	\arrow[from=1-2, to=2-3]
\end{tikzcd}\] which makes $\Conf$ into a commutative algebra in correspondences. Here the second leg of the correspondence is the restriction of the map $\add_{\Conf}: \Conf \times \Conf \to \Conf$ given by addition of divisors. Hence we may define a \emph{factorization space} over $\Conf$ to be a multiplicative space for $\Conf$. 

\subsubsection{} Parallel to what happens for the Ran space, the configuration space may be upgraded to a lax prestack \[ \Conf^{\to}: \Aff^{\op}_k \longrightarrow \Cat\] where there is a unique arrow $D \to E$ if and only if $\Gamma_D \subseteq \Gamma_E$. In typical fashion, there is a natural structure of a commutative algebra in correspondences on $\Conf^{\to}$. As before, we may consider left fibrations over $\Conf^{\to}$ endowed with a multiplicative structure. To distinguish them from the Ran situation, we will call such objects \emph{effective factorization spaces}. 

Given an effective factorization space $\cY^{\to} \to \Conf^{\to}$ with underlying factorization space $\cY \to \Conf$, we may consider the category $\fD_{\eff}(\cY) \coloneqq \fD(\cY^{\to}_{\str})$ of \emph{effective sheaves} on $\cY$. There is also a description of this category in terms of equivariant sheaves with respect to a corresponding \emph{defect action} $\Conf \times \cY \to \cY$. Note however that since $\Conf$ is \emph{not} contractible, the forgetful functor $\fD_{\eff}(\cY) \to \fD(\cY)$ is no longer fully-faithful. 

\section{The Ran semi-infinite IC sheaf}\label{ransic}

We will recall here the construction of the \emph{semi-infinite intersection cohomology sheaf} of \cite{Ga1} and \cite{Ga2}. Note the definitions of the affine Grassmannian and related structures that will be used in the following may be found in \emph{loc. cit.} and the references therein.

\subsubsection{} Recall the \emph{affine Grassmannian} $\Gr_{G, \Ran}$ is the factorization space over $\Ran$ whose $S$-points are given by triples $(x_I, \sP_G,\alpha)$, where $x_I$ is an $S$-point of $\Ran$, $\sP_G$ is a $G$-bundle on $X \times S$, and \begin{equation}\label{iso} \alpha: \sP^0_G |_{X \times S \setminus \Gamma_{x_I}} \overset{\sim}{\longrightarrow} \sP_G |_{X \times S \setminus \Gamma_{x_I}} \end{equation} is a trivialization of $\sP_G$ away from the complement of $\Gamma_{x_I}$. Here $\sP^0_G$ denotes the trivial bundle. 

The factorization structure on $\Gr_{G,\Ran}$ is given by gluing $G$-bundles. Moreover, this upgrades to a unital factorization structure given as follows. Let $x_I \subseteq x_J$ and let $(x_I, \sP_G, \alpha)$ be a point of $\Gr_{G,\Ran}$ over $x_I$. By restricting $\alpha$ to the complement of $\Gamma_{x_J}$ we obtain a point over $x_J$, and hence the structure of a left fibration over $\Ran^{\un}$. 

\subsubsection{} Suppose we have an $S$-point $(x_I,\sP_G, \alpha)$ of $\Gr_{G,\Ran}$. Then for each dominant coweight $\chlam \in \check{\Lambda}^+$ the trivialization $\alpha$ induces meromorphic maps 
\begin{equation}
\label{pluckeriso}
\cO_{X \times S} \longrightarrow V^{\chlam}_{\sP^0_G} \overset{\alpha}{\longrightarrow} V^{\chlam}_{\sP_G} 
\end{equation}
where $V^{\chlam}$ is the irreducible representation of $G$ with highest weight $\chlam$ and $V^{\chlam}_{\sP_G}$ is the associated vector bundle. Here, the first map in \eqref{pluckeriso} is given by the inclusion of the highest weight line. Define a closed, factorizable sub-indscheme $\newoverline{S}^0_{\Ran}$ of $\Gr_{G,\Ran}$ given by the condition that the compositions \eqref{pluckeriso} extend to regular maps \begin{equation}\label{pluckerfors}\cO_{X \times S} \longrightarrow V^{\chlam}_{\sP_G}\end{equation} for all dominant coweights. The prestack $\newoverline{S}^0_{\Ran}$ is naturally closed under the unital structure of $\Gr_{G,\Ran}$. 

Let $\dBun_N$ denote \emph{Drinfeld's compactification}, whose definition will be recalled in detail in Section \ref{drincomp}. There is a natural map \[\newoverline{\mathfrak{p}}_{\Ran}: \newoverline{S}^0_{\Ran} \longrightarrow \dBun_N\] taking a point of the left hand side to the (possibly degenerate) reduction to $N$ defined by the trivialization. More precisely, the maps \eqref{pluckerfors} satisfy the Pl\"{u}cker relations and hence define a point of $\dBun_N$. 

\subsubsection{} For an algebraic group $H$, let $\mathfrak{L}_{\Ran} H$ denote the group prestack over $\Ran$ whose $S$-points are given by an $S$-point $x_I$ of $\Ran$ together with a map $\mathring{D}_{x_I} \to H$. By bounding the degrees of zeros of the maps \eqref{pluckerfors} we obtain an $\fL_{\Ran} N$-equivariant stratification of $\newoverline{S}^0_{\Ran}$, with strata $S^{\lambda}_{\Ran}$ indexed by negative coweights $\lambda \in \Lambda^-$. The stratum \[j^0: S^0_{\Ran} \hookrightarrow \newoverline{S}^0_{\Ran}\] labeled by $0$ is given by the \emph{semi-infinite orbit}, an open dense subspace of $\newoverline{S}^0_{\Ran}$ obtained by the condition that the sections \eqref{pluckerfors} have no zeros. Denote by $\mathfrak{p}^0_{\Ran}$ the restriction of $\newoverline{\mathfrak{p}}_{\Ran}$ to $S^0_{\Ran}$. 

We then consider the category $\operatorname{SI}_{\Ran}$ of $\fL_{\Ran} N$-equivariant sheaves on the affine Grassmannian. By considering $\fL_{\Ran} N$-equivariant sheaves with support on $\newoverline{S}^0_{\Ran}$ we obtain the category $\operatorname{SI}^{\leq 0}_{\Ran}$. 

Similarly, there is a full subcategory $\operatorname{SI}_{\Ran}^{=\lambda}$ of $\fD(S^{\lambda}_{\Ran})$ for each negative coweight $\lambda$. A feature of these categories is that they are preserved by ($!$ and $*$) pullback and pushforward between the $S^{\lambda}_{\Ran}$'s and $\newoverline{S}^0_{\Ran}$. The following summarizes some of the central results of \cite{Ga2}. 

\begin{theorem} (Gaitsgory) \begin{enumerate} \item{For each negative coweight $\lambda$, there is a t-structure on $\operatorname{SI}_{\Ran}^{=\lambda}$ characterized by the fact that (up to a cohomological shift) the dualizing sheaf $\omega_{S^{\lambda}_{\Ran}}$ is in the heart $(\operatorname{SI}^{=\lambda}_{\Ran})^{\heartsuit}$. In particular, for $\lambda=0$ the dualizing sheaf $\omega_{S^0_{\Ran}}$ itself is an object of the heart $\operatorname{SI}^{=0,\heartsuit}_{\Ran}$.} 
\item{The t-structures on the categories $\operatorname{SI}^{=\lambda}_{\Ran}$ glue to one on $\operatorname{SI}^{\leq 0}_{\Ran}$ and the intermediate extension \[\sICR \coloneqq j^0_{!*}(\omega_{S^0_{\Ran}}) \in (\operatorname{SI}^{\leq 0}_{\Ran})^{\heartsuit}\] possesses a canonical structure of a factorization algebra. This factorization structure is uniquely determined by the requirement that it extends the tautological one on $\omega_{S^0_{\Ran}}$.} 
\end{enumerate}
\end{theorem}

 \begin{remark}It should be emphasized that the t-structure on $\operatorname{SI}^{\leq 0}_{\Ran}$ is \emph{not} the perverse t-structure. Indeed, the latter does not make sense with respect the stratification of $\newoverline{S}^0_{\Ran}$ since the strata $S^{\lambda}_{\Ran}$ are both infinite dimensional and of infinite codimension. 

  The fact that the Gaitsgory t-structure behaves differently from the perverse t-structure is reflected, for example, in the fact that $\omega_{S^0_{\Ran}}$ lies in the heart of $\operatorname{SI}^{=0}_{\Ran}$. This is because for an ind-scheme $\cY$ the dualizing sheaf $\omega_{\cY}$ typically lies in cohomological degree $-\infty$ in the perverse t-structure. 
 \end{remark}

 Following \cite{Ga2}, we call the object $\sICR$ the \emph{semi-infinite intersection cohomology sheaf}. We will not say much more about the t-structure for which $\sICR$ is an intermediate extension, but it is worth mentioning that the t-structure on $\operatorname{SI}^{=\lambda}_{\Ran}$ is pulled back from the usual one on the scheme of $T$-fixed points of $S^{\lambda}_{\Ran}$. In other words, the scheme of $T$-fixed points of $S^{\lambda}_{\Ran}$ is isomorphic to the partially symmetrized power $X^{\lambda}$ of the curve and we have an equivalence \[\operatorname{SI}^{=\lambda}_{\Ran} \overset{\sim}{\longrightarrow} \mathfrak{D}(X^{\lambda})\] which is t-exact up to a cohomological shift \cite{Ga2}.  

The following is Theorem 3.3.3 in \cite{Ga2} and relates the semi-infinite intersection cohomology sheaf to its global version lying over Drinfeld's compactification $\dBun_N$. The latter is simply given by the ordinary IC sheaf $\IC_{\dBun_N}$ up to a cohomological shift. 

\begin{proposition}\label{sicispulledbackfromdrin} Let $d=(g-1)\operatorname{dim}(N)$. Then there is a canonical isomorphism \[\newoverline{\mathfrak{p}}_{\Ran}^!(\IC_{\dBun_N})[d] \overset{\sim}{\longrightarrow} \sICR\] extending the tautological isomorphism $(\mathfrak{p}_{\Ran}^0)^!(\omega_{\Bun_N}) \simeq \omega_{S^0_{\Ran}}$.
\end{proposition}

\subsection{The configuration space semi-infinite IC sheaf.} We will also need a version of $\sICR$ which lies over the configuration space. To this end, define the configuration space affine Grassmannian $\Gr_{G,\Conf}$ to be the prestack whose $S$-points are given by triples $(D, \sP_G, \alpha)$, where 
\begin{enumerate}
    \item $D$ is an $S$-family of colored divisors;
    \item $\sP_G \to X \times S$ is a $G$-bundle;
    \item $\alpha$ is a trivialization of $\sP_G$ on $(X \times S) \setminus \Gamma_D$. 
\end{enumerate} As before, we define a closed substack $\newoverline{S}^0_{\Conf}$ inside $\Gr_{G, \Conf}$ by a positivity condition. More precisely, the trivialization $\alpha$ induces meromorphic maps \[\cO_{X \times S} \longrightarrow V^{\chlam}_{\sP_G} \] for every dominant weight $\chlam$ and we require that these extend to $X \times S$. The space $\newoverline{S}^0_{\Conf}$ is also equipped with a canonical map $\newoverline{\mathfrak{p}}_{\Conf}$ to $\dBun_N$. 

We may endow $\Gr_{G,\Conf}$ with the structure of an effective factorization space as follows. Let $D$ be a colored divisor and let $D'$ be another colored divisor such that $\Gamma_D \subseteq \Gamma_{D'}$. Suppose we have a point $(D,\sP_G,\alpha)$ of $\Gr_{G,\Conf}$ lying over $D$. The effective structure is given by the map 
\[ (D,\sP_G, \alpha) \mapsto (D',\sP_G, \alpha'),\] where $\alpha'$ is the restriction of $\alpha$ to $X \times S \setminus \Gamma_{D'}$. The factorization structure is again given by gluing $G$-bundles. It is evident that $\newoverline{S}^0_{\Conf}$ is closed under both structures. 

\subsubsection{}\label{correspondencesection}Define the \emph{correspondence configuration space} $\Conf_{\corr}$ as the subspace of $\Conf \times \Ran$ given by the condition that the divisor is set-theoretically supported on the point of $\Ran$. 

We have a correspondence 
\begin{equation}\label{confcorr}\begin{tikzcd}
	& {\Conf_{\corr}} \\
	\Conf && \Ran
	\arrow["\mathfrak{c}_L"', from=1-2, to=2-1]
	\arrow["\mathfrak{c}_R",from=1-2, to=2-3]
\end{tikzcd}\end{equation} defined in the obvious way. For a space $\cY \to \Conf$ lying over the configuration space, define \[\cY_{\corr} \coloneqq \cY \times_{\Conf} \Conf_{\corr}\] as a space over $\Conf_{\corr}$. Similarly, define \[\cY_{\corr,R} \coloneqq \cY \times_{\Ran} \Conf_{\corr} \] for a space $\cY \to \Ran$. 

We also have a correspondence analogous to \eqref{confcorr} whose left leg is $\Conf^{\to}$ and whose total space is denoted $\Conf^{\to}_{\corr}$. For a space $\cY \to \Conf^{\to}$, denote by $\cY^{\to}_{\corr}$ its pullback to $\Conf^{\to}_{\corr}$. Similarly, for a space $\cY \to \Ran$ we have the space $\cY^{\to}_{\corr,R}$ over $\Conf^{\to}_{\corr}$. 

\subsubsection{} Let us show that $\Conf^{\to}_{\corr}$ is a commutative algebra in $\Corr(\LStk)$. Define \[[\Conf^{\to}_{\corr} \times \Conf^{\to}_{\corr}]_{\disj} \coloneqq (\Conf^{\to}_{\corr} \times \Conf^{\to}_{\corr}) \times_{\Ran \times \Ran} [\Ran \times \Ran]_{\disj} \] and note that in addition to the obvious map \[\operatorname{inc}_{\corr}: [\Conf^{\to}_{\corr} \times \Conf^{\to}_{\corr}]_{\disj} \to \Conf^{\to}_{\corr} \times \Conf^{\to}_{\corr}\] we have another map \[\add_{\corr}: [\Conf^{\to}_{\corr} \times \Conf^{\to}_{\corr}]_{\disj} \longrightarrow \Conf^{\to}_{\corr}\] which sends a pair $\big( (D,x_I), (E,x_J) \big)$ to $(D+E,x_I \cup x_J)$. 

The inclusion of the pair \[(0,\emptyset) \longhookrightarrow \Conf^{\to}_{\corr}\] provides the necessary unit for $\Conf^{\to}_{\corr}$, and it is easy to see that the above maps endow $\Conf^{\to}_{\corr}$ with the structure of a commutative algebra in $\Corr(\LStk)$. Note that the inclusion of $(\emptyset,\emptyset)$ into $\Conf^{\to}_{\corr}$ is an inclusion of a connected component.

%Given a factorizable effective space $\cY^{\to} \to \Conf^{\to}$, denote by $[\cY^{\to}_{\corr} \times \cY^{\to}_{\corr}]_{\disj}$ the base change of $\cY^{\to}_{\corr} \times \cY^{\to}_{\corr}$ along $\operatorname{inc}_{\corr}$. Similarly, denote by $\cY^{\to}_{\corr,\disj}$ the base change of $\cY^{\to}_{\corr}$ along $\add_{\corr}$.

Now the map $\mathfrak{c}_L$ upgrades to a lax morphism of commutative algebras, while the second projection $\mathfrak{c}_R$ is a strict morphism of commutative algebras. It follows that for an effective factorization space $\cY^{\to} \to \Conf^{\to}$, the pullback $\cY^{\to}_{\corr}$ to $\Conf^{\to}_{\corr}$ obtains a structure of a multiplicative space over $\Conf^{\to}_{\corr}$. Similarly, if $\cY$ is a factorization space over $\Ran$ then $\cY_{\corr,R}$ obtains a structure of a multiplicative space over $\Conf^{\to}_{\corr}$. 

\subsubsection{}\label{localizationstuff} In addition to the usual structure on $\Conf^{\to}_{\corr}$ of a commutative algebra in correspondences, there is a richer structure ${'}{\Conf_{\corr}^{\to}}$. The underlying prestack of ${'}{\Conf_{\corr}^{\to}}$ is just $\Conf_{\corr}$, whereas morphisms are given as follows. For an affine scheme $S$, there is precisely one morphism $(D,x_I) \to (E,x_J)$ of $S$-points whenever there is a morphism $D \to E$ in $\Conf^{\to}(S)$ and $x_I \to x_J$ in $\Ran^{\un}(S)$. Note the lax prestack ${'}{\Conf_{\corr}^{\to}}$ also acquires a structure of a commutative algebra in correspondences such that the inclusion \[\Conf_{\corr}^{\to} \longrightarrow {'}{\Conf_{\corr}^{\to}} \] is a map of commutative algebras. 

The maps $\mathfrak{c}_L$ and $\mathfrak{c}_R$ are compatible with the structures of the previous paragraph, and moreover $\mathfrak{c}_L$ is a pointwise Cartesian fibration. For a multiplicative space $\cY$ (resp. $\cX$) over $\Conf^{\to}_{\corr}$ (resp. over $\Ran^{\un}$), the space $\mathfrak{c}_L^{\Spc,*}(\cY)=\cY_{\corr}$ (resp. $\mathfrak{c}^{\Spc,*}_R(c)=\cY_{\corr,R}$) over ${'}{\Conf^{\to}_{\corr}}$ has a canonical structure of a multiplicative space. Note also the projection \begin{equation}\label{localizationmap} \cY_{\corr} \longrightarrow \cY \end{equation} is a pointwise Cartesian fibration by base change. 

Let $\fD_{\untl}(\cY_{\corr})$ denote the full subcategory of $\fD(\cY_{\corr})$ consisting of sheaves $c$ such that $c_{S,f} \to c_{S,g}$ is an isomorphism for any arrow $f \to g$ of $S$-points in $\cY_{\corr}$ whose image in $\cY$ is the identity. For any $\cY$ such that $\cY(S)$ is discrete for all $S$, by Proposition $7.1.12$ of \cite{Cis} the map \eqref{localizationmap} is a localization and there is a canonical equivalence \begin{equation}\label{localizationequivalence} \fD(\cY) \overset{\sim}{\longrightarrow} \fD_{\untl}(\cY_{\corr}).\end{equation} Indeed, a sheaf on $\cY_{\corr}$ is in particular the data of a functor \[\cY_{\corr}(S) \longrightarrow \mathfrak{D}(S)\] for every affine scheme $S$, and by our assumptions any such which sends the appropriate morphisms to inverses admits a unique factorization through \eqref{localizationmap}. Note an analogous discussion holds over $\Ran$ and we obtain a category $\fD_{\untl}(\cY_{\corr,R})$ for any multiplicative space $\cY \to \Ran^{\un}$.

\subsubsection{} We claim there is a map $\iota_{\corr}: \newoverline{S}^0_{\Conf,\corr} \to \newoverline{S}^0_{\Ran,\corr,R}$ which makes the triangle
\[\begin{tikzcd}[column sep=tiny, row sep=large ]
	{\newoverline{S}^0_{\Conf,\corr}} && {\newoverline{S}^0_{\Ran,\corr,R}} \\
	& {\Conf_{\corr}}
	\arrow[from=1-1, to=2-2]
	\arrow[from=1-3, to=2-2]
	\arrow["{\iota_{\corr}}", from=1-1, to=1-3]
\end{tikzcd}\] commute. Indeed, $\iota_{\corr}\big((x_J,D,\sP_G,\alpha)\big)$ is obtained by restricting $\alpha$ to the complement of $\Gamma_{x_J}$. In the following proposition, denote by $t_{\Conf}$ (resp. $t_{\Ran}$) the projection $\newoverline{S}^0_{\Conf,\corr} \to \newoverline{S}^0_{\Conf}$ (resp. $\newoverline{S}^0_{\Ran,\corr} \to \newoverline{S}^0_{\Ran}$).

\begin{proposition}\label{prop: conf semi-infinite ic sheaf} There exists an effective factorization algebra $\sICC$ over $\newoverline{S}^0_{\Conf}$ together with an isomorphism \[t_{\Conf}^!(\sICC) \overset{\sim}{\longrightarrow} (t_{\Ran} \circ \iota_{\corr})^!(\sICR)\] which is compatible with factorization. 
\end{proposition}
\begin{proof}
Let us apply the discussion in \ref{localizationstuff} to our present situation. Consider the multiplicative spaces $\newoverline{S}^0_{\Conf,\corr}$ and $\newoverline{S}^0_{\Ran,\corr,R}$ over ${'}{\Conf^{\to}_{\corr}}$. It is easy to see that the map $\iota_{\corr}: \newoverline{S}^0_{\Conf,\corr} \to \newoverline{S}^0_{\Ran,\corr,R}$ is compatible with multiplicative structures on both sides. It follows that the shriek pullback $\iota_{\corr}^!$ restricts to a functor \[\iota_{\corr}^!: \fD_{\untl}(\newoverline{S}^0_{\Ran,\corr,R}) \longrightarrow \fD_{\untl}(\newoverline{S}^0_{\Conf,\corr}).\] Hence we will obtain the object $\sICC \in \fD(\newoverline{S}^0_{\Conf})$ corresponding uniquely under the equivalence \eqref{localizationequivalence} to $(\iota_{\Ran} \circ \iota_{\corr})^!(\sICR)$ provided $\sICR$ belongs to the unital subcategory of $\fD(\newoverline{S}^0_{\Ran})$. The latter is the content of Section $4$ of \cite{Ga2}. Note that $\sICC$ immediately inherits the structure of an effective factorization algebra from $\sICR$. 
\end{proof}

\section{Zastava spaces and the semi-infinite IC sheaf}\label{zastava}

In this section we will review the theory of Drinfeld's compactifications and Zastava spaces. More information about both objects can be found in \cite{BFGM} and a review of Zastava spaces via the affine Grassmannian can be found in \cite{Ga0}.

\subsection{Drinfeld's compactification.}\label{drincomp} Let $\Bun_B$ denote the moduli of $B$-bundles on $X$, and consider the induction morphism $\operatorname{pr}_{B,G}:\Bun_B \longrightarrow \Bun_G.$ Although the fibers of this map are not proper in general, there is a proper morphism \[\overline{\operatorname{pr}}_{B,G}: \dBun_B \longrightarrow \Bun_G\] together with an open embedding $j: \Bun_B \to \dBun_B$ with dense image which commutes with the projection to $\Bun_G$. The stack $\dBun_B$ is known as \emph{Drinfeld's compactification}, and is explicitly defined as follows. Let $S$ be an affine scheme. A map $S \to \dBun_B$ is the data of a $G$-bundle $\sP_G$ and a $T$-bundle $\sP_T$ on $X \times S$, together with embeddings of coherent sheaves \[\check{\lambda}(\sP_T) \longrightarrow V^{\chlam}_{\sP_G}\] for every $\chlam \in \check{\Lambda}$ which satisfy the Pl\"{u}cker relations. There is an evident map $\dBun_B \to \Bun_T$ whose preimage over $\Bun_T^{\lambda}$ will be denoted $\dBun_B^{\lambda}$. Here $\Bun_T^{\lambda}$ is the connected component of $\Bun_T$ corresponding to $\lambda \in \Lambda$ consisting of $T$-bundles of total degree $\lambda$. 

\subsubsection{} There is a stratification of $\dBun_B$ by smooth substacks indexed by positive coweights $\Lambda^+$. Each stratum ${_{=\lambda}}{\dBun_B}$ is canonically isomorphic\footnote{Whenever convenient we will identify $X^{\lambda} \times \Bun_B$ and its image in $\dBun_B$.} to $X^{\lambda} \times \Bun_B$, and we will write $\iota_{\lambda}$ for the inclusion ${_{=\lambda}}{\dBun_B} \to \dBun_B$. The union of strata indexed by coweights which are less than or equal to some fixed coweight $\mu$ is an open substack of $\dBun_B$ and will be denoted by ${_{\leq \mu}}{\dBun_B}$. 

The maps $\iota_{\lambda}$ extend to finite morphisms \begin{equation}\label{compactact}\overline{\iota}_{\lambda}: X^{\lambda} \times \dBun_B \longrightarrow \dBun_B\end{equation} which assemble into an action of the configuration space on $\dBun_B$. Moreover, the diagram 
\begin{equation}\label{actcomm}\begin{tikzcd}[column sep=tiny]
	{X^{\lambda} \times \dBun_B} && {\dBun_B} \\
	& {\Bun_G}
	\arrow["{\overline{\iota}_{\lambda}}", from=1-1, to=1-3]
	\arrow["{a_{X^{\lambda}} \times \overline{\operatorname{pr}}_{B,G}}"', from=1-1, to=2-2]
	\arrow["{\overline{\operatorname{pr}}_{B,G}}", from=1-3, to=2-2]
\end{tikzcd} \end{equation} commutes, where $a_{X^{\lambda}}: X^{\lambda} \to \pt$ is the unique map. In practice, we will more often use the corresponding structure for $B^-$. 

\subsubsection{} We will also need a variant $\dBun_N$ of Drinfeld's compactification for the unipotent radical $N$ of $B$. By definition, $\dBun_N$ sits in a pullback square
\[\begin{tikzcd}
	{\dBun_N} & {\dBun_B} \\
	\Spec(k) & {\Bun_T}
	\arrow[from=1-1, to=1-2]
	\arrow[from=1-2, to=2-2]
	\arrow[from=1-1, to=2-1]
	\arrow[from=2-1, to=2-2]
\end{tikzcd}\]
where $\Spec(k) \to \Bun_T$ is the $k$-point of $\Bun_T$ corresponding to the trivial $T$-bundle. Since $T$ acts on $N$ by the adjoint action, it also naturally acts on the stack $\dBun_N$. 

\subsubsection{} We will denote by $\ovF_{N,B^-}$ the fiber product \[\dBun_N \times_{\Bun_G} \dBun_{B^-}. \] Define the \emph{compactified Zastava space} $\pZ$ to be the open locus of $\ovF_{N,B^-}$ where the two generalized reductions are generically transverse. Explicitly, an $S$-point of $\pZ$ consists of a $G$-bundle $\sP_G$ and a $T$-bundle $\sP_T$ on $X \times S$ together with diagrams 
\begin{equation}
\label{zastavaeq}
\cO_{X \times S} \overset{\kappa^{\chlam}}{\longrightarrow} V^{\chlam}_{\sP_G} \overset{(\kappa^-)^{\chlam}}{\longrightarrow} \chlam(\sP_T)
\end{equation}
for every dominant weight $\chlam$, where the collections $\{\kappa^{\chlam}\}$ and $\{(\kappa^-)^{\chlam}\}$ satisfy the Pl\"{u}cker equations and each composition $(\kappa^-)^{\chlam} \circ \kappa^{\chlam}$ is nonzero. 

Denote by $\ofp$ (resp. $\ofq$) the projection $\pZ \to \dBun_N$ (resp. $\pZ \to \dBun_{B^-}$). Note that $\pZ$ splits into connected components $\pZ^{\lambda}$ indexed by \emph{negative} coweights $\lambda$, where $\pZ^{\lambda}$ is given as the preimage of $\dBun_{B^-}^{\lambda}$ under $\ofq$. By taking the zeros of the composition of \eqref{zastavaeq} we obtain a map 
\[\overline{\pi}^{\lambda}: \pZ^{\lambda} \longrightarrow X^{\lambda}.\] It is well-known that each $\pZ^{\lambda}$ is a scheme of finite type, and that $\overline{\pi}^{\lambda}$ is proper. Moreover, there is a natural factorization structure on $\pZ$ with respect to $\pibar$. 

By commutativity of the diagram \eqref{actcomm} for $B$ replaced by $B^-$, the maps $\overline{\iota}_{\lambda}$ give rise to an action of $\Conf$ on the fiber product $\ovF_{N,B^-}$. Moreover, since the condition defining $\pZ$ is generic, the compactified Zastava is preserved by this action. This equips $\pZ$ with the structure of an effective factorization space. 

\subsubsection{}\label{definition of zastava} By requiring the maps $(\kappa^-)^{\chlam}$ in \eqref{zastavaeq} to be surjective, we obtain the \emph{affine Zastava} $\cZ$. The space $\cZ$ admits a factorization structure and the dense open embedding \[j_Z : \cZ \hookrightarrow \pZ\] is compatible with factorization on both sides. Note however that the the factorization structure on $\cZ$ is \emph{not} effective. Similarly, by requiring that the maps $\kappa^{\chlam}$ are regular embeddings of vector bundles, we obtain the factorizable open subscheme \[j_Z^-: \cZ^- \hookrightarrow \pZ\] which we call the \emph{opposite affine Zastava}. The intersection $\oZ \coloneqq \cZ \cap \cZ^-$ is called the \emph{open Zastava} and is known to be smooth. Note that each variation of the Zastava space carries a $T$-action.

\subsection{Zastava spaces via the affine Grassmannian.}\label{section: zastava spaces via the affine grassmannian}

Suppose we have a point of the compactified Zastava space, consisting of two generically transverse generalized $N$ and $B^-$ reductions of a $G$-bundle $\sP_G$. Over the locus of transversality both reductions are genuine, and intersect in a single point over each fiber of the projection $\sP_G \to X \times S$. It follows that we may equip $\sP_G$ with a section away from the support of a colored divisor $D$, and hence we obtain a point of $\Gr_{G,\Conf}$. The result is a closed embedding \[\pZ \longhookrightarrow \Gr_{G,\Conf}\] which we will now describe more explicitly. 

Define $\newoverline{S}^{-}_{\Conf}$ to be the subspace of $\Gr_{G,\Conf}$ whose $S$-points consist of triples $(D, \sP_G, \alpha)$ such that for each dominant weight $\chlam \in \check{\Lambda}^+$ the composition \[V^{\chlam}_{\sP_G} \overset{\alpha}{\longrightarrow} V^{\chlam}_{\sP_G^0} \longrightarrow \cO_{X \times S}\] of meromorphic maps extends to a regular map \[V^{\chlam}_{\sP_G^0} \longrightarrow \cO_{X \times S}(-\chlam(D)),\] where the map $V^{\chlam}_{\sP_G^0} \to \cO_{X \times S}$ corresponds to the unique $N^-$-invariant functional on $V^{\chlam}$. Note that $\newoverline{S}^-_{\Conf}$ inherits an effective structure from the one on $\Gr_{G,\Conf}$.

The image of the compactified Zastava space in $\Gr_{G,\Conf}$ may be described as the intersection \[\newoverline{S}^0_{\Conf} \times_{\Gr_{G,\Conf}} \newoverline{S}^-_{\Conf} \longhookrightarrow \Gr_{G,\Conf}.\] As a result of the definitions, we see that the embedding $\pZ \hookrightarrow \Gr_{G,\Conf}$ is compatible with the effective factorization structures on both sides. We are now ready to define our central object of study.

\begin{definition}\label{zsic} Denote by $i_Z: \pZ \hookrightarrow \newoverline{S}^0_{\Conf}$ the inclusion. Define the \emph{Zastava semi-infinite intersection cohomology sheaf} $\sIC$ to be the factorization algebra $i_Z^!(\sICC)$ in $\fD_{\eff}(\pZ)$. Note that $\sIC$ inherits a natural effective structure from that of $\sICC$. 
\end{definition} 

\begin{remark}
    In Section \ref{semiinfiniteforzastava} we will give a construction of $\sIC$ that does not depend on the prior existence of $\sICR$.
\end{remark}

In the remaining sections we will give a more explicit description of $\sIC$ and show that it is, in a precise sense, equivalent to $\sICR$. 

\section{The effective category and the generic Zastava}\label{genzastava}

In this section we will establish the following theorem.

\begin{theorem}\label{effisun} There is an equivalence of categories \[\fD_{\eff}(\pZ) \overset{\sim}{\longrightarrow} \fD_{\untl}(\newoverline{S}^0_{\Ran})\] which is compatible with factorization and sends $\sIC$ to $\sICR$ up to a cohomological shift. 
\end{theorem}

In fact, we will actually show a slightly stronger statement from which Theorem \ref{effisun} can be deduced as a corollary. We will begin with some preparations. The following well-known result is an infinity-categorical analogue of Quillen's theorem A and is an immediate consequence of Theorem 4.1.3.1 of \cite{Lu}. 

\begin{lemma}\label{quillenthm} Let $F:\mathscr{C} \to \mathscr{D}$ be a functor such that for every object $d \in \mathscr{D}$, the straightening $(\mathscr{C}_{/d})_{\str}$ of the slice category $\mathscr{C}_{d/}$ of $F$ over $d$ is contractible. Then the induced functor \[F_{\str}: \mathscr{C}_{\str} \longrightarrow \mathscr{D}_{\str}\] is an equivalence.
\end{lemma}

\subsubsection{} Let $\varphi:\cY \to \Conf$ be an effective space over $\Conf$ and let $\varphi^{\to}: \cY^{\to} \to \Conf^{\to}$ be the associated lax prestack. Recall the correspondence \[\begin{tikzcd}
	& {\Conf^{\to}_{\corr}} \\
	{\Conf^{\to}} && \Ran
	\arrow[from=1-2, to=2-1]
	\arrow[from=1-2, to=2-3]
\end{tikzcd}\] and define \[ \cY^{\to}_{\corr} \coloneqq \cY^{\to} \times_{\Conf^{\to}} \Conf^{\to}_{\corr} \] to be the lax prestack whose $S$-objects are triples $(y,x_I,D)$, where $y$ is an $S$-point of $\cY$, $D$ is the image of $y$ under $\varphi^{\to}: \cY^{\to} \to \Conf^{\to}$, and $x_I$ is a point of $\Ran$ on which $D$ is set-theoretically supported. 

Since the fibers of the map $\mathfrak{c}_R: \Conf^{\to}_{\corr} \to \Ran$ are functors of directed sets, we see by Lemma \ref{quillenthm} that the induced map \[(\Conf^{\to}_{\corr})_{\str} \longrightarrow \Ran\] is an equivalence. As a result, the projection $\cY^{\to}_{\corr} \to \Ran$ factors through the canonical map $\cY^{\to}_{\corr} \to (\cY^{\to}_{\corr})_{\str}$. Let $\cY_{\Ran}$ denote the prestack $(\cY^{\to}_{\corr})_{\str}$ and write \[\varphi_{\Ran}: \cY_{\Ran} \longrightarrow \Ran\] for the resulting morphism. 

\subsection{Factorization over the configuration space to Ran.}\label{factoverconftoran} In this section we will show that the assignment $\cY \mapsto \cY_{\Ran}$ sends effective factorization spaces to unital factorization spaces. 
%Given a factorizable effective space $\cY^{\to} \to \Conf^{\to}$, denote by $[\cY^{\to}_{\corr} \times \cY^{\to}_{\corr}]_{\disj}$ the base change of $\cY^{\to}_{\corr} \times \cY^{\to}_{\corr}$ along $\operatorname{inc}_{\corr}$. Similarly, denote by $\cY^{\to}_{\corr,\disj}$ the base change of $\cY^{\to}_{\corr}$ along $\add_{\corr}$.

\begin{proposition}\label{conftoran} Let $\cY^{\to} \to \Conf^{\to}$ be an effective factorization space. Then the prestack $\cY_{\Ran} \to \Ran$ has a natural structure of a unital factorization space. 
\end{proposition}
\begin{proof} By Proposition \ref{multspctomultspc}, it suffices to show that the maps $\operatorname{inc}_{\corr}$ and $\add_{\corr}$ (as well as the corresponding maps for the unital diagram) are realization fibrations. By Theorem 2.1 of \cite{S} it is enough to show that $\operatorname{inc}_{\corr}$ and $\add_{\corr}$ are both pointwise Cartesian and coCartesian fibrations. Note that while the right leg of the unital diagram is not a pointwise (co)Cartesian fibration, it is still a realization fibration since it is the inclusion of a disjoint base-point.

We will show that $\add_{\corr}$ is a pointwise coCartesian fibration, and the other cases will be left to the reader. Let \[f: (D+E,x_I \cup x_J) \longrightarrow (D',x_I \cup x_J)\] be a morphism in $\Conf^{\to}_{\corr}$ with $\big( (D,x_I), (E,x_J) \big)$ in $[\Conf^{\to}_{\corr} \times \Conf^{\to}_{\corr}]_{\disj}$. Since $x_I$ and $x_J$ are disjoint, we can write \[D'=D'(x_I)+D'(x_J).\] where $D'(x_I)$ is set-theoretically supported on $x_I$ and $D'(x_J)$ is set-theoretically supported on $x_J$. Then it is easy to see that there exists \[f':\big((D,x_I),(E,x_J) \big) \longrightarrow \big((D'(x_I),x_I),(D'(x_J),x_J) \big)\] which is a coCartesian lift of $f$.  

To see that $\cY_{\Ran}$ is a unital factorization space, first recall the action of $\Ran$ on $\cY^{\to}_{\corr}$ constructed in Section \ref{localizationstuff}. Clearly, this action is compatible with the morphisms in $\cY^{\to}_{\corr}$ and hence descends to an action on $\cY_{\Ran}$. Lastly, it is easy to see that this action is compatible with the factorization structure constructed above.
\end{proof}

\begin{corollary}\label{rancategoryfromconfcategory} There is a natural equivalence of categories \[\fD_{\str}(\cY^{\to}_{\corr}) \overset{\sim}{\longrightarrow} \fD(\cY_{\Ran})\] which sends factorization algebras to factorization algebras.
\end{corollary}
\begin{proof} The equivalence of categories follows by definition of $\cY_{\Ran}$, and the fact that it preserves factorization algebras follows from the construction of the factorization structure on $\cY_{\Ran}$ from Proposition \ref{conftoran}. 
\end{proof}

The next proposition relates the effective category of $\cY^{\to}$ with the unital subcategory of $\fD(\cY_{\Ran})$. Note that the natural projection $\operatorname{pr}_{\cY}: \cY^{\to}_{\corr} \to \cY^{\to}$ is a morphism of lax prestacks, and hence we obtain a functor \[\operatorname{pr}^!_{\cY_{\str}}: \fD_{\eff}(\cY) \longrightarrow \fD(\cY_{\Ran})\] from the category of effective sheaves on $\cY$ and the category of D-modules on $\cY_{\Ran}$ given by !-pullback along the induced map of straightenings. 

\begin{proposition}\label{effectivetounital} Let $\varphi^{\to}:\cY^{\to} \to \Conf$ be an effective factorization space and assume the underlying prestack $\cY$ takes values in sets. Then the functor \[\operatorname{pr}_{\cY_{\str}}^!: \fD_{\eff}(\cY) \longrightarrow \fD(\cY_{\Ran}) \] is fully faithful and its essential image is the unital subcategory of $\fD(\cY_{\Ran})$. 
\end{proposition}
\begin{proof} By Section \ref{localizationstuff} we may endow $\cY^{\to}_{\corr}$ with an action of $\Ran$ by functors. This gives us a modified lax prestack ${'}{\cY^{\to}_{\corr}}$ whose $S$-objects are given by $\cY^{\to}_{\corr}(S)$ with a morphism \[(y,x_I) \longrightarrow (y',x_J)\] for every morphism $y \to y'$ provided $x_I$ is contained in $x_J$. By definition of the unital structure on $\cY_{\Ran}$, the straightening of ${'}{\cY^{\to}_{\corr}}$ is tautologically isomorphic to $(\cY_{\Ran})_{\str}$. 

We claim that the obvious map \[{'}{\cY^{\to}_{\corr}} \longrightarrow \cY^{\to}\] induces an isomorphism on the level of straightenings after sheafification in the fppf topology. The reason sheafification is needed is because $\Conf_{\corr} \to \Conf$ is only surjective locally in the fppf topology by results of \cite{Bar}. Now let $y_0$ be an $S$-object of $\cY^{\to}$. By Lemma \ref{quillenthm}, it suffices to show that the category \[\cY^{\to}_{\corr,y_0} \coloneqq \cY^{\to}_{\corr}(S) \times_{\cY^{\to}(S)} \cY^{\to}_{y_0/}(S)\] is weakly contractible. 

By our assumption on $\cY$, the category $\cY^{\to}_{\corr,y_0}$ is equivalent to the category whose objects are pairs $(y,x_J)$ as in the definition of $\cY^{\to}_{\corr}$ and $y_0$ admits a (necessarily unique) map to $y$. Let $(y,x_I)$ and $(y',x_J)$ be two such objects. Since $\cY^{\to} \to \Conf^{\to}$ is a pointwise coCartesian fibration, we may lift the morphism \[\varphi^{\to}(y) \longrightarrow \varphi^{\to}(y)+ \varphi^{\to}(y') \] to a coCartesian morphism $y \to y_1$. We obtain an analogous morphism $y' \to y_1'$. 

Let $\lambda$ be the degree of $\varphi^{\to}(y)+ \varphi^{\to}(y')$. Then, as points of $\cY$, $y_1$ and $y_1'$ lie in $\cY^{\lambda}$ over $\Conf^{\lambda}$, and since $y_1$ and $y_1'$ admit a map from $y_0$, we conclude that there is a (necessarily unique) isomorphism $y_1 \overset{\sim}{\to} y_1'$. Hence we obtain maps \[\begin{tikzcd}
	{(y,x_I)} \\
	& {(y_1,x_I \cup x_J)} \\
	{(y',x_J)}
	\arrow[from=1-1, to=2-2]
	\arrow[from=3-1, to=2-2]
\end{tikzcd}\] and by our hypothesis on $\cY$, we conclude that $\cY^{\to}_{\corr,/y_0}$ is a directed set and is therefore weakly contractible.  

To conclude the proof, notice we have constructed a commutative diagram \[\begin{tikzcd}
	& {{'}{\cY^{\to}_{\corr}}} \\
	{\cY^{\to}_{\str}} && {(\cY^{\to}_{\Ran})_{\str}}
	\arrow[from=1-2, to=2-1]
	\arrow[from=1-2, to=2-3]
	\arrow["\sim", from=2-1, to=2-3]
\end{tikzcd}\] where the bottom arrow is an isomorphism. It follows that $\operatorname{pr}^!_{\cY_{\str}}$ is fully faithful and has essential image given by $\fD_{\untl}(\cY_{\Ran})$. 
\end{proof}

\subsubsection{} We will now apply the machinery above to the affine Grassmannian. An application of Lemma \ref{quillenthm} shows that $(\Gr_{G,\Conf})_{\Ran}$ is canonically isomorphic to $\Gr_{G,\Ran}$. Moreover, the unital structure on $\Gr_{G,\Ran}$ is inherited from the one on $\Gr^{\to}_{G,\Conf,\corr}$ via Proposition \ref{conftoran}. Likewise, we have a canonical isomorphism \[(\newoverline{S}^0_{\Conf})_{\Ran} \overset{\sim}{\longrightarrow} \newoverline{S}^0_{\Ran}\] inducing the unital structure on the right hand side.

The next proposition shows that we can actually obtain all of $\Gr_{G,\Ran}$ from $\newoverline{S}^{-,\to}_{\Conf}$, where we recall that $\newoverline{S}^{-,\to}_{\Conf}$ is the effective factorization space obtained from the space $\newoverline{S}^-_{\Conf}$ defined in Section \ref{section: zastava spaces via the affine grassmannian}. Denote by $\iota^-_{\corr}$ the inclusion $\newoverline{S}^{-,\to}_{\Conf,\corr} \hookrightarrow \Gr^{\to}_{G,\Conf,\corr}$.

\begin{lemma}\label{strisequiv} The inclusion \[\iota^-_{\corr}: \newoverline{S}^{-,\to}_{\Conf,\corr} \hookrightarrow \Gr^{\to}_{G,\Conf,\corr}\] induces an equivalence \[(\newoverline{S}^{-,\to}_{\Conf})_{\Ran} \overset{\sim}{\longrightarrow} (\Gr_{G,\Conf})_{\Ran} \] and hence $(\newoverline{S}^{-,\to}_{\Conf})_{\Ran}$ is canonically equivalent to $\Gr_{G,\Ran}$. 
\end{lemma}
\begin{proof} Let $(D,x_I,\sP_G,\alpha)$ be an $S$-point of $\Gr^{\to}_{G,\Conf,\corr}$. We will construct a point $(D',x_I, \sP_G,\alpha')$ of $\newoverline{S}^{-,\to}_{\Conf}$ together with a morphism $(D,x_I,\sP_G,\alpha) \to (D',x_I,\sP_G,\alpha')$. To this end, let $\omega_i$ be a fundamental weight of $G$, and consider the meromorphic map \[\beta_i: V^{\omega_i}_{\sP_G} \longrightarrow \cO_{X \times S}\] induced by $\alpha$. 

Since $\beta_i$ is defined away from the divisor $\Gamma_D$, there exists an integer $n_i$ large enough such that the composition \[V^{\omega_i}_{\sP_G} \longrightarrow \cO_{X \times S} \longrightarrow \cO_{X \times S}(n_i \Gamma_D)\] is regular. Define the colored divisor \[D' \coloneqq D + \sum_{i \in \mathcal{I}_G} (-n_i)\alpha_i \Gamma_D\] and note that there is an obvious map $D \to D'$. Define $\alpha'$ to be the restriction of the map $\alpha$ to the complement of $\Gamma_{D'}$. Tautologically the tuple $(D',x_I,\sP_G,\alpha')$ admits a map from $(D,x_I, \sP_G, \alpha)$ and $(D',x_I,\sP_G,\alpha')$ is a point of $\newoverline{S}^{-,\to}_{\Conf}$. 

Now consider the composition of \[\newoverline{S}^{-,\to}_{\Conf,\corr} \longrightarrow \Gr^{\to}_{G,\Conf,\corr} \longrightarrow (\Gr_{G,\Conf})_{\Ran}, \] where the first arrow is the inclusion and the second arrow is the natural map. By the discussion above, this map is essentially surjective for every affine scheme $S$, and the fact that it induces an equivalence \[(\newoverline{S}^{-}_{\Conf})_{\Ran} \overset{\sim}{\longrightarrow} (\Gr_{G,\Conf})_{\Ran}\] is a direct application of Lemma \ref{quillenthm}. Indeed, the fibers of the functor $\newoverline{S}^{-,\to}_{\Conf,\corr}(S) \to \Gr_{G,\Ran}(S)$ are easily seen to be directed sets for every affine scheme $S$. 
\end{proof}

We are now ready to prove the theorem. 

\begin{proof}[Proof of Theorem \ref{effisun}] Recall we have a natural embedding $\pZ^{\to} \hookrightarrow \Gr_{G,\Conf}^{\to}$ whose essential image is the intersection $\newoverline{S}^{0,\to}_{\Conf} \cap \newoverline{S}^{-,\to}_{\Conf}$. It is easy to see that $\newoverline{S}^{-,\to}_{\Conf,\corr} \to \Gr_{G,\Ran}$ is both a Cartesian and coCartesian fibration, and hence a realization fibration. By Lemma \ref{strisequiv} the inclusion \[\pZ \hookrightarrow \newoverline{S}^0_{\Conf}\] induces a canonical isomorphism \[\pZ_{\Ran} \overset{\sim}{\longrightarrow} (\newoverline{S}^0_{\Conf})_{\Ran} \simeq \newoverline{S}^0_{\Ran}\] on the level of fppf sheafifications in such a way that $\cZ^-_{\Ran}$ is identified with $S^0_{\Ran}$. We conclude the existence of a natural equivalence $\mathscr{F}:\fD_{\eff}(\pZ) \overset{\sim}{\longrightarrow} \fD_{\untl}(\newoverline{S}^0_{\Ran})$ by Proposition \ref{effectivetounital}.

By construction, the equivalence $\mathscr{F}$ is the composition of two equivalences \[\mathfrak{D}_{\eff}(\pZ) \overset{\mathscr{F}_0}{\longrightarrow} \mathfrak{D}_{\eff}(\newoverline{S}^0_{\Conf}) \overset{\mathscr{F}_1}{\longrightarrow} \mathfrak{D}_{\untl}(\newoverline{S}^0_{\Ran})\] and tracing through the proof of Proposition \ref{prop: conf semi-infinite ic sheaf} we see that $\mathscr{F}_1(\IC^{\frac{\infty}{2}}_{\Conf})$ is canonically equivalent to $\sICR$ as a factorization algebra. Moreover, by definition of $\sIC$ and the fact that the inclusion $\pZ^{\to} \hookrightarrow \newoverline{S}^{0,\to}_{\Conf}$ induces an equivalence on the level of straightenings, we have that $\mathscr{F}_0(\sIC)$ is canonically equivalent to $\sICC$ as an effective factorization algebra, whence the theorem. 

%Note that the projection $\newoverline{\mathfrak{p}}: \pZ \to \dBun_N$ commutes with the defect action of $\Conf$ on $\pZ$ and the trivial action on $\dBun_N$, and therefore upgrades to a morphism $\newoverline{\mathfrak{p}}: \pZ^{\to} \to \dBun_N$. Hence we have a commutative diagram \[\begin{tikzcd}[column sep=tiny]
	%{\pZ^{\to}_{\corr}} && {\newoverline{S}^0_{\Ran}} \\
	%& {\dBun_N} & {}
	%\arrow[from=1-1, to=1-3]
	%\arrow["{\newoverline{\mathfrak{p}}_{\corr}}"', from=1-1, to=2-2]
	%\arrow["{\newoverline{\mathfrak{p}}^0_{\Ran}}", from=1-3, to=2-2]
%\end{tikzcd}\] where $\newoverline{\mathfrak{p}}_{\corr}$ is the composition of the projection $\pZ^{\to}_{\corr} \to \pZ^{\to}$ with $\newoverline{\mathfrak{p}}$. We conclude that $\mathscr{F}(\sIC)$ is canonically isomorphic to $\sICR$ up to a cohomological shift.

%It remains to check that $\sICR$ is isomorphic to $\mathscr{F}(\sIC)$ as a factorization algebra. However, it is easy to check that the isomorphism $\mathscr{F}(\sIC) \overset{\sim}{\to} \sICR$ as sheaves extends the tautological isomorphism of factorization algebras over $S^0_{\Ran}$. Since the factorization structure on $\sICR$ is uniquely determined by its restriction to $S^0_{\Ran}$, we are done. 
\end{proof}

\subsection{The generic Zastava.} The following sections on the generic Zastava will not be used in the rest of the text. Nevertheless, the results contained therein may be of independent interest as they connect the theory of Zastava spaces and the semi-infinite intersection cohomology sheaf to moduli spaces of rationally defined maps and principal bundles (see \cite{Bar} as well as Appendix A of \cite{Ga2}). 

Denote by $\Gr^{\gen}_G$ the straightening of the lax prestack $\Gr_{G,\Ran}^{\to}$ associated to the canonical unital structure. We wish to give a more explicit description of $\Gr^{\gen}_G$. Let $S$ be an affine scheme, and recall that an open subscheme $U \subseteq X \times S$ is called a \emph{domain} if it is universally dense with respect to the projection \[p_S: X \times S \longrightarrow S.\] This means that for any closed point $s \in S$, the intersection of the fiber $X_s$ of the projection $p_S$ over $s$ with $U$ is nonempty.

It is well known that $\Gr^{\gen}_G$ is the moduli space whose $S$-points are triples $(\sP_G, U, \alpha)$ where $\sP_G \to X \times S$ is a $G$-bundle, $U \subseteq X \times S$ is a domain, and $\alpha$ is a trivialization of $\sP_G$ defined over $U$. An isomorphism of such triples is an isomorphism of $G$-bundles which commutes with the trivializations over some subdomain of the intersection of their domains of definition. 

\subsubsection{} Let us generalize the above construction. For $H \subseteq G$ a subgroup, let $\Bun_{G,H}^{\gen}$ denote the following moduli space (see Appendix A of \cite{Ga2} for more information). For an affine scheme $S$, a morphism \[S \to \Bun_{G,H}^{\gen}\] is a triple $(\sP_G, U, \kappa)$, where $\sP_G \to X \times S$ is a $G$-bundle, $U \subseteq X \times S$ is a domain, and $\kappa$ is a reduction of $\sP_G$ to $H$ over $U$. An isomorphism \[(\sP_G,U,\kappa) \overset{\sim}{\longrightarrow} (\sP_G',U',\kappa')\] is an isomorphism $\sP_G \overset{\sim}{\to} \sP_G'$ over $X \times S$ which commutes with $\kappa$ and $\kappa'$ over a domain contained in $U \cap U'$. In this notation, note that we have a tautological isomorphism \[\Gr_G^{\gen} \overset{\sim}{\longrightarrow} \Bun_{G,1}^{\gen},\] where $1 \subseteq G$ denotes the trivial subgroup. 

For any subgroup $H$ of $G$, there is a morphism \[\Gr_G^{\gen} \longrightarrow \Bun_{G,H}^{\gen}\] which takes a triple $(\sP_G,U,\alpha)$ to $(\sP_G,U,\alpha_H)$, where $\alpha_H$ is the isomorphism given by the composition \[\sP_G \overset{\sim}{\longrightarrow} \sP_G^0 \simeq \ind_1^G(\sP_1^0) \overset{\sim}{\longrightarrow} \ind_H^G(\sP_H^0) \] defined over $U$. 

\subsubsection{} When $H=N$, there is also a morphism $\dBun_N \to \Bun_{G,N}^{\gen}$ obtained by restricting a generalized reduction $\kappa$ of a $G$-bundle $\sP_G$ to its \emph{non-degeneracy domain}, i.e. the maximal domain over which the Pl\"{u}cker maps are injective maps of vector bundles. Now define a prestack $\newoverline{S}^{\gen}$ by the condition that it sits in a Cartesian square \[\begin{tikzcd}
	& {} \\
	{\newoverline{S}^{\gen}} & {\Gr_G^{\gen}} \\
	{\dBun_N} & {\Bun_{G,N}^{\gen}.}
	\arrow[from=2-1, to=2-2]
	\arrow[from=2-2, to=3-2]
	\arrow["\mathfrak{p}^{\gen}"', from=2-1, to=3-1]
	\arrow[from=3-1, to=3-2]
\end{tikzcd}\]

Explicitly, an $S$-point of $\newoverline{S}^{\gen}$ is given by a $G$-bundle $\sP_G \to X \times S$, a domain $U \subseteq X \times S$, and a trivialization $\alpha$ of $\sP_G$ defined over $U$ such that the following condition holds. For all dominant weights $\lambda \in \check{\Lambda}^+$, the meromorphic maps \[ \cO_{X \times S} \longrightarrow V^{\chlam}_{\sP_G^0} \overset{\alpha}{\longrightarrow} V^{\chlam}_{\sP_G}\] induced by $\alpha$ extend to regular maps defined over all of $X \times S$. As usual, the map $\cO_{X \times S} \to V^{\chlam}_{\sP_G^0}$ is the inclusion of the highest weight line. 

\subsubsection{} In what follows, denote by $\uMap^{\gen}(X,N)$ the group prestack of generically defined maps $X \to N$ (see \cite{Bar} for a careful definition). In other words, an $S$-point of $\uMap^{\gen}(X,N)$ is a map $U \to N$, where $U \subseteq X \times S$ is a domain. Two such maps $U \to N$ and $U' \to N$ are equivalent if they agree on a subdomain of $U \cap U'$. 

By Theorem $1.8.2$ in \cite{Ga-2}, the space $\uMap^{\gen}(X,N)$ is homologically contractible\footnote{Note Theorem $1.8.2$ essentially states that the homology of $\uMap^{\gen}(X,N)$ is scalars, but as noted in \emph{loc. cit.} this is an equivalent condition to homological contractibility.}. As the next proposition shows, the projection $\mathfrak{p}^{\gen}: \newoverline{S}^{\gen} \to \dBun_N$ is very close to being an equivalence. Note the following result is closely related to arguments from Appendix A of \cite{Ga2}. 

\begin{proposition}\label{torsor} The map $\mathfrak{p}^{\gen}: \newoverline{S}^{\gen} \to \dBun_N$ is a torsor for the group prestack $\uMap^{\gen}(X,N)$. In particular, $\mathfrak{p}^{\gen}$ is universally homologically contractible.
\end{proposition}
\begin{proof} For a subgroup $H$ of $G$, define a prestack $\Bun_H^{\gen}$ by declaring a morphism $S \to \Bun_H^{\gen}$ to be a domain $U \subseteq X \times S$ as well as an $H$-bundle \[ \sP_H \to U \subseteq X \times S.\] An isomorphism $(U, \sP_H) \to (U',\sP'_H)$ of pairs as above is an isomorphism $\sP_H \to \sP'_H$ defined over a domain contained in $U \cap U'$. There is an obvious morphism $\Bun_H^{\gen} \to \Bun_G^{\gen}$. Moreover, it is easy to see that there is a canonical isomorphism \[ \psi_H:\Bun_{G,H}^{\gen} \overset{\sim}{\longrightarrow} \Bun_H^{\gen} \times_{\Bun_G^{\gen}} \Bun_G.\]

Now since $N$ is unipotent, for any affine scheme $S$ the \'{e}tale cohomology group $H^1_{\acute{e}t}(S,N)$ vanishes, and hence there is an isomorphism \[\beta: \Bun_N^{\gen} \overset{\sim}{\longrightarrow} B\big(\uMap^{\gen}(X,N)\big) \] which commutes with the obvious map from $\Spec(k)$ to both sides. Therefore under the equivalences $\psi_1$, $\psi_N$, and $\beta$, the map $\Gr_G^{\gen} \to \Bun_{G,N}^{\gen}$ corresponds to the map \[(u \times \id_{\Bun_G}): \Spec(k) \times_{\Bun_G^{\gen}} \Bun_G \longrightarrow B\big(\uMap^{\gen}(X,N)\big) \times_{\Bun_G^{\gen}} \Bun_G\] obtained by base change from \[u: \Spec(k) \to B\big(\uMap^{\gen}(X,N)\big).\] Since $\mathfrak{p}^{\gen}$ is a further base change of $(u \times \id_{\Bun_G})$, we conclude that it is a torsor for $\uMap^{\gen}(X,N)$.

The claim that $\mathfrak{p}^{\gen}$ is universally homologically contractible follows from the following general observation. Let $\cY$ be a prestack and let $\pi: \cX \to \cY$ be a torsor for a homologically contractible group prestack $\mathcal{H}$. Then $\pi$ is universally homologically contractible (see Lemma 6.2.10 of \cite{Bar}). 
\end{proof}

\subsection{Equivalence between the generic Zastava and a moduli space of rational data.} Define the \emph{generic Zastava} $\cZ^{\gen}$ to be the straightening $\pZ^{\to}_{\str}$. Denote by \[\mathfrak{t}^{\gen}_Z: \pZ \longrightarrow \cZ^{\gen}\] the quotient map.

In this section we will identify $\cZ^{\gen}$ with $\newoverline{S}^{\gen}$. First note that there is a canonical equivalence \[(\Gr_{G,\Conf}^{\to})_{\str} \longrightarrow \Gr_G^{\gen}.\] Indeed, this follows from the fact that for every domain $U \subseteq X \times S$ there exists a Zariski cover $S' \to S$ such that $U \times_{X \times S} (X \times S')$ is a divisor complement, as well as an application of Lemma \ref{quillenthm}. 

\subsection{The generic Zastava in detail.} In this section we will give an alternative description of $\cZ^{\gen}$ using the moduli space of rational reductions to the opposite Borel $B^-$. 

Recall the stack $\Bun_{G,B^-}^{\gen}$ with its map to $\Bun_G$. Then we can form the fiber product \[\dBun_N \times_{\Bun_G} \Bun_{G,B^-}^{\gen}\] together with its open locus $\cZ^{\gen}_0$ consisting of a generalized $N$-reduction with a generically transverse generically defined $B^-$-reduction. Note this is a well-defined condition. We claim there is an isomorphism \begin{equation}\label{genericisgeneric0}\cZ^{\gen} \overset{\sim}{\longrightarrow} \cZ^{\gen}_0.\end{equation} Indeed, $\cZ^{\gen}$ is the quotient of $\pZ$ by the action of the configuration space. But $\Conf$ inherits its action from a similar action on $\dBun_{B^-}$, and it is easy to see that the quotient of the latter is equivalent to $\Bun_{G,B^-}^{\gen}$.

From this perspective, we get another understanding of Proposition \ref{torsor}. The fiber of the map \[\Bun_{G,B^-}^{\gen} \longrightarrow \Bun_G\] over a $G$-bundle $\sP_G \to X \times S$ is given by the prestack \[\operatorname{GSect}_{X \times S}((G/B^-)_{\sP_G}) \overset{\sim}{\longrightarrow} \Bun_{G,B^-}^{\gen} \times_{\Bun_G} S\] of generically defined sections of the bundle $(G/B^-)_{\sP_G}$ associated to $\sP_G$ with fiber $G/B^-$ (see \cite{Bar} for a definition). 

In particular, the fiber over the trivial bundle is tautologically equivalent to the prestack $\uMap^{\gen}(X \times S, G/B^-)$. Inside the latter we have the open subprestack consisting of generically defined maps $X \times S \to G/B^-$ which factor through the open cell $N \hookrightarrow G/B^-$. This is precisely the fiber of the map \[\mathfrak{p}^{\gen}: \cZ^{\gen} \longrightarrow \dBun_N\] over the trivial reduction of $\sP_G^0$ exhibiting $\mathfrak{p}^{\gen}$ as a torsor for $\uMap^{\gen}(X,N)$. 

Let us also identify a stratification on $\cZ^{\gen}$. First recall that $\Bun^{\gen}_{G,B^-}$ has a stratification whose strata are isomorphic to the connected components of $\Bun_{B^-}$. By the equivalence \eqref{genericisgeneric0} we see that $\cZ^{\gen}$ has a (non-smooth) stratification whose strata are isomorphic to the connected components of $\cZ$. In particular, the composition \[\cZ \longrightarrow \cZ^{\gen}\] of the open embedding $j_Z : \cZ \hookrightarrow \pZ$ with the quotient map $\mathfrak{t}^{\gen}_Z: \pZ \to \cZ^{\gen}$ is a bijection on $k$-points. 

\section{Finite-dimensional semi-infinite t-structures}\label{semiinfinitetstructures}

In this section, we will construct an analogue of Gaitsgory's t-structure for a large class of schemes with an action of the multiplicative group. Let $\cY$ be an irreducible separated scheme of finite-type with an action by $\G_m$. We further assume that its fixed subscheme $\cY^0 \coloneqq \uMap_{\G_m}(\operatorname{pt},\cY)$ is smooth. 

\subsection{The t-structure on schemes with a contracting action.} Recall that the category $\fD(\cY)^{\G_m-\operatorname{mon}}$ of $\G_m$-monodromic sheaves on $\cY$ is defined to be the full subcategory of $\fD(\cY)$ generated under colimits by the essential image of the forgetful functor \[\fD(\cY)^{\G_m} \longrightarrow \fD(\cY)\] from $\G_m$-equivariant sheaves to all $D$-modules. 

Assume for a moment that $\G_m$ contracts $\cY$ to its, by assumption irreducible, fixed subscheme. By this we mean that the action map \[\G_m \times \cY \longrightarrow \cY\] extends to a map \[\A^1 \times \cY \longrightarrow \cY\] whose precomposition with the inclusion \[\{0\} \times \cY \hookrightarrow \A^1 \times \cY\] maps onto $\cY^0$ and is the identity on $\{0\} \times \cY^0 \simeq \cY^0$. Let $\pi$ denote the projection $\cY \to \cY^0$ and $\sigma$ the inclusion of $\cY^0$ into $\cY$. 

It is known that the left adjoint $\pi_!$ of $\pi^!$ is defined on $\G_m$-monodromic sheaves and is explicitly given by $\sigma^!$. It follows that the functor \[\pi^!: \fD(\cY^0) \longrightarrow \fD(\cY)^{\G_m-\operatorname{mon}}\] is fully-faithful. Let $\operatorname{SI}(\cY)$ denote the full subcategory of $\fD(\cY)^{\G_m-\operatorname{mon}}$ equal to the essential image of $\pi^!$. Define a t-structure on $\operatorname{SI}(\cY)$ by declaring an object $\pi^!(c)$ to be connective/coconnective if and only if $c[\operatorname{dim}(\cY)-\operatorname{dim}(\cY^0)]$ is connective/coconnective in the perverse t-structure on $\cY^0$. For example, if $\cY$ is smooth, then the projection $\pi$ is a locally trivial affine bundle \cite{B}, and the t-structure above reduces to the usual one on $\cY$. 

\subsection{The general case.} Let us return to the case of a general $\cY$ satisfying the assumptions at the beginning of the section. We will further assume that $\cY$ admits a \emph{positive Bia{\l}ynicki-Birula decomposition}. This means $\cY$ admits a stratification into locally closed subschemes, each of which is an \emph{attracting locus} (see \cite{B} where this was first defined for smooth schemes). Explicitly, attracting loci are the connected components of the attracting scheme \[\cY^+ \coloneqq \uMap_{\G_m}(\A^1, \cY)\] where here $\A^1$ is equipped with its usual action by $\G_m$ contracting it to $0$. Let us index these connected components by a set $I$, and for $\mu \in I$ denote by $\cY^+_{\mu}$ the corresponding attracting locus and by $j_{\mu}: \cY^+_{\mu} \to \cY$ the inclusion. 

We will call any scheme $\cY$ satisfying the above assumptions $\G_m$-\emph{complete}. For example, any scheme with a contracting action is $\G_m$-complete. Any projective normal variety acted on by $\G_m$, or any scheme $\cY$ admitting an equivariant closed embedding \[ \cY \longhookrightarrow \mathbb{P}(V)\] where $V$ is a vector space and $\G_m$ acts on $\mathbb{P}(V)$ linearly is also $\G_m$-complete (see Appendix $B$ of \cite{D} and the references therein).  

As a non-example, consider $\A^2$ with the action $t \cdot (x,y) = (tx, t^{-1}y)$. The only fixed point is the origin, and the only points in its attracting locus are those of the form $(x,0)$. Another scheme which should not be considered is $\mathbb{P}^1$ with $0$ and $\infty$ identified.   

From now on, fix a $\G_m$-complete scheme $\cY$. Define a full subcategory \[\operatorname{SI}(\cY) \longhookrightarrow \fD(\cY)^{\G_m-\operatorname{mon}}\] consisting of objects $c$ such that for every $\mu \in I$, the pullback $j^!_{\mu}(c)$ belongs to $\operatorname{SI}(\cY^+_{\mu})$. Note that $\operatorname{SI}(\cY)$ is cocomplete and for every $c \in \operatorname{SI}(\cY^+_{\mu})$ the pushforward $(j_{\mu})_*(c)$ belongs to $\operatorname{SI}(\cY)$. We have the following basic proposition establishing the existence of semi-infinite t-structures. Note points $1$ and $2$ are essentially\footnote{Although the author of \emph{loc.cit.} does not prove Proposition \ref{tstructure} in our general context, the arguments are essentially the same.} proved in \cite{Ga2}. 

\begin{proposition}\label{tstructure} Let $\cY$ and $\cX$ be $\G_m$-complete schemes, and let $f: \cY \to \cX$ be a $\G_m$-equivariant \'{e}tale morphism such that for any connected component of $\cX^0$ its preimage in $\cY$ consists of a single connected component of $\cY^0$. Then

    \begin{enumerate} \item For every $\mu \in I$, the left adjoint $j^*_{\mu}$ of the pushforward functor \[(j_{\mu})_*: \operatorname{SI}(\cY^+_{\mu}) \longrightarrow \operatorname{SI}(\cY)\] is well-defined; 
    \item the category $\operatorname{SI}(\cY)$ has a t-structure characterized by the condition that an object $c$ is connective (resp. coconnective) if $j^*_{\mu}(c)[-\operatorname{dim}(\cY)]$ (resp. $j^!_{\mu}(c)[-\operatorname{dim}(\cY)]$) is connective\footnote{The reason for the cohomological shift by the dimension of $\cY$ is so that the dualizing sheaf of the open stratum lies in the heart. This will also be clarified when we apply the situation to the semi-infinite IC sheaf later.} (resp. coconnective) for all $\mu \in I$;
    \item the functor $f^!$ induces a functor \[f^!: \operatorname{SI}(\cX) \longrightarrow \operatorname{SI}(\cY)\] which is t-exact with respect to the above t-structures.
    \end{enumerate}
\end{proposition}
\begin{proof}
    If we assume $1$, we also obtain the existence of left-adjoint to the pullback functor \[j^!_{\mu}: \operatorname{SI}(\cY) \longrightarrow \operatorname{SI}(\cY^+_{\mu})\] by a Cousin argument. Indeed, by induction, we can assume there are exactly two strata. Let the open embedding be denoted by $j$ and the closed by $i$. it is enough to verify that $j^!$ admits a left adjoint on semi-infinite categories. To this effect, define $j_!(c)$ as the fiber of the counit \[j_*(c) \longrightarrow i_*i^*j_*(c)\] where $i^*$ is the left adjoint to $i_*$ which we are assuming exists on semi-infinite categories. The unit $c \to j^!j_!(c)$ is given explicitly by the canonical isomorphism \[c \overset{\sim}{\longrightarrow} \operatorname{fib}\big(j^!j_*(c) \longrightarrow j^!i_*i^*j_*(c)\big)\] using the fact that $j^!i_* \simeq 0$. Now the required condition can be checked by mapping into objects of the form $j_*(d)$ and $i_*(d')$ separately. 
    
    Since each stratum $\cY^+_{\mu}$ is locally closed inside $\cY$, both $!$ and $*$ pushforwards along the embedding \[\cY^+_{\mu} \longhookrightarrow \cY \] are fully faithful. As a result, we deduce that $2$ follows formally by recollement of t-structures. Hence it suffices to verify points $1$ and $3$. 

    For $1$ let us first note that the existence of the desired left adjoint follows from Braden's theorem \cite{Br} \cite{DG0}. Indeed, let $\pi_{\mu}$ denote the projection $\cY^+_{\mu} \to \cY^0_{\mu}$. Then by construction of the semi-infinite categories it suffices to show that the functor $(j_{\mu})_* \circ \pi^!_{\mu}$ has a left adjoint, which is exactly the statement of Braden's theorem.

    For $3$ observe that we can immediately reduce to the case where $\cY$ and $\cX$ have a contracting $\G_m$ action. In this case we have a commuting square \[\begin{tikzcd}
	{\cY^0} & \cY \\
	{\cX^0} & \cX
	\arrow[from=1-1, to=1-2]
	\arrow["f", from=1-2, to=2-2]
	\arrow["{f_0}"', from=1-1, to=2-1]
	\arrow[from=2-1, to=2-2]
\end{tikzcd}\]which we claim is Cartesian. Indeed, any point of $\cY$ mapping to a fixed point of $\cX$ must itself be fixed since $f$ is \'{e}tale and $\G_m$ is connected. Hence $f_0$ is also \'{e}tale and we conclude that $f^!:\operatorname{SI}(\cY) \to \operatorname{SI}(\cX)$ is t-exact. 
\end{proof}

\subsection{Semi-infinite sheaf for $\mathbb{G}_m$-complete schemes.} Let $\cY$ be $\G_m$-complete. By irreducibility, $\cY$ has a unique open stratum $\cY^+_0$ with fixed locus $\cY^0_0$. Note that by definition, the dualizing sheaf $\omega_{\cY^+_0}$ is an object of the heart $\operatorname{SI}(\cY^+_0)^{\heartsuit}$. We have the following definition.

\begin{definition}
    Define the \emph{semi-infinite sheaf} for a $\G_m$-complete scheme $\cY$ to be the intermediate extension \[\IC^{\frac{\infty}{2}}(\cY) \coloneqq \operatorname{im}\big({\mathbb{H}^0_{\operatorname{si}}(j_0)_!(\omega_{\cY^+_0})} \longrightarrow \mathbb{H}^0_{\operatorname{si}}(j_0)_*(\omega_{\cY^+_0})\big) \]
in the semi-infinite t-structure on $\cY$. Here $\mathbb{H}^*_{\operatorname{si}}(-)$ denotes hypercohomology in this t-structure. 
\end{definition}

In what follows denote by $\operatorname{SI}_{\operatorname{hol}}(\cY)$ the preimage of the category $\fD_{\operatorname{hol}}(\cY)$ of \emph{holonomic} $D$-modules under the inclusion \[\operatorname{SI}(\cY) \hookrightarrow \fD(\cY).\] We will also denote by $\IC^{\operatorname{ren}}_{\cY}$ for a scheme (or algebraic stack) $\cY$ of finite type the sheaf $\IC_{\cY}[\operatorname{dim}(\cY)]$ and call it the \emph{renormalized intersection cohomology sheaf} of $\cY$. The following proposition in particular identifies $\IC^{\frac{\infty}{2}}(\cY)$ in a large class of examples. 

\begin{proposition}\label{semiinfiniteisusual}
Assume that \begin{enumerate} \item the projections $\pi_{\mu}: \cY^+_{\mu} \to \cY^0_{\mu}$ are all smooth;  \item the renormalized intersection cohomology sheaf $\IC^{\operatorname{ren}}_{\cY}$ belongs to $\operatorname{SI}(\cY)$. \end{enumerate} Then the inclusion functor \[\operatorname{SI}_{ }(\cY) \longhookrightarrow \fD_{ }(\cY)\] is left t-exact up to cohomological shift by $-\operatorname{dim}(\cY)$ and sends $\IC^{\frac{\infty}{2}}(\cY)$ to $\IC^{\operatorname{ren}}_{\cY}$. 
\end{proposition}
\begin{proof}
    Let $c \in \operatorname{SI}_{ }(\cY)^{\geq 0}$ be coconnective. Define for brevity \[d_{\mu} = -(\operatorname{dim}(\cY^0_{\mu})-\operatorname{dim}(\cY^+_{\mu}))\] to be the relative dimension of $\pi_{\mu}$. By definition, for every $\mu \in I$ we have \[j_{\mu}^!(c) \in \operatorname{SI}(\cY)^{\geq 0}.\] Since $j_{\mu}^!(c)$ belongs to the semi-infinite category, we may write \[j_{\mu}^!(c) \overset{\sim}{\longrightarrow} \pi_{\mu}^!(c')\] for some $c' \in \fD(\cY^0_{\mu})$. The fact that $j_{\mu}^!(c)$ is coconnective in the semi-infinite t-structure means that $c'[-\operatorname{dim}(\cY)+d_{\mu}]$ is coconnective in the \emph{perverse} t-structure on $\fD(\cY^0_{\mu})$. In other words, $c'$ lies in perverse degrees $ \geq -\operatorname{dim}(\cY)+d_{\mu}$. Since $\pi_{\mu}$ is smooth, the functor $\pi^!_{\mu}[-d_{\mu}]$ is conservative and t-exact for the perverse t-structure. We conclude that $\pi^!_{\mu}(c')$ lies in perverse degrees $-\operatorname{dim}(\cY)$ in $\fD(\cY)$ and that the inclusion \[ \operatorname{SI}(\cY) \longhookrightarrow \fD(\cY)\] is left t-exact up to the desired shift.

    To see that $\IC^{\operatorname{ren}}_{\cY}$ coincides with $\IC^{\frac{\infty}{2}}(\cY)$, its enough to verify that for all $\mu \in I$, we have \[j^!_{\mu}(\IC^{\operatorname{ren}}_{\cY}) \in \operatorname{SI}(\cY)^{\geq 1} \,\,\, \text{and} \,\,\, j^*_{\mu}(\IC^{\operatorname{ren}}_{\cY}) \in \operatorname{SI}(\cY)^{\leq 1}.\] Note however that $\operatorname{SI}(\cY) \hookrightarrow \fD(\cY)$ is conservative and left exact by the discussion above. Hence we just need to check the statement about $*$-pullbacks. 

    Since $\IC^{\operatorname{ren}}_{\cY} \in \operatorname{SI}_{\operatorname{hol}}(\cY)$, the partially defined left adjoint $j^*_{\mu,\operatorname{hol}}:\fD_{\operatorname{hol}}(\cY) \to \fD_{\operatorname{hol}}(\cY^+_{\mu})$ to \[(j_{\mu})_*: \fD(\cY^+_{\mu}) \longrightarrow \fD(\cY)\] is defined on $\IC^{\operatorname{ren}}_{\cY}$. A similar statement holds for the left adjoint $(\pi_{\mu})_!$ of $\pi^!_{\mu}$. It follows that we may compute \[j^*_{\mu}(\IC^{\operatorname{ren}}_{\cY}) \overset{\sim}{\longrightarrow} (\pi_{\mu})_!j^*_{\mu,\operatorname{hol}}(\IC^{\operatorname{ren}}_{\cY}) \] viewing both sides as objects of $\fD(\cY^0_{\mu})$. Note that $\IC^{\operatorname{ren}}_{\cY}$ lies in perverse degrees $-\operatorname{dim}(\cY)$, and hence $j^*_{\mu,\operatorname{hol}}(\IC^{\operatorname{ren}}_{\cY})$ lies in perverse degrees $\leq -\operatorname{dim}(\cY) - 1$. But $\pi_{\mu}$ is smooth and so $\pi^!_{\mu}[-d_{\mu}]$ is perverse t-exact, implying $(\pi_{\mu})_![d_{\mu}]$ is perverse right t-exact. Therefore $j^*_{\mu}(\IC^{\operatorname{ren}}_{\cY})$ lies in perverse cohomological degrees $ \leq -\operatorname{dim}(\cY)+d_{\mu}-1$ as desired.    
\end{proof}

Let us apply Proposition \ref{semiinfiniteisusual} to some cases of interest. Let $\cY$ be the flag variety $G/B$ with an action of $\G_m$ obtained by picking a regular dominant coweight $\G_m \to T$. In this case, the Bia{\l}ynicki-Birula decomposition coincides with the Schubert stratification. Consider a Schubert variety $X_w= \newoverline{NwB}$ corresponding to an element $w$ of the Weyl group $W$. Then $X_w$ is stable under the $N$ action on $G/B$ and $\IC^{\operatorname{ren}}_{X_w}$ is $N$-invariant. Hence $\IC^{\operatorname{ren}}_{X_w}$ belongs to the semi-infinite category and by Proposition \ref{semiinfiniteisusual} it is canonically equivalent to $\IC^{\frac{\infty}{2}}(X_w)$. Indeed, requiring that a sheaf on a stratum $Nw'B$ is pulled back from the unique $\G_m$-fixed point $w'B$ is tautologically equivalent to imposing equivariance for $N$.

As another example, consider the affine Grassmannian $\Gr_G$ with its action by $\G_m$ via \emph{loop rotation}, and consider the closure $\newoverline{\Gr}^{\lambda}_G$ of a $\mathfrak{L}^+G$-orbit $\Gr^{\lambda}_G$ associated to a dominant coweight $\lambda$. For more information on the following setup, see \cite{Zhu}. The Bia{\l}ynicki-Birula decomposition of $\newoverline{\Gr}^{\lambda}_G$ for loop rotation coincides with its stratification into $\mathfrak{L}^+G$-orbits $\Gr^{\mu}_G$ indexed by dominant coweights $\mu \leq \lambda$. Let $K$ denote the kernel of the evaluation map \[\mathfrak{L}^+G \longrightarrow G\] a pro-unipotent group scheme. Then the action of $K$ on $\newoverline{\Gr}^{\lambda}_G$ preserves the fibers of the projections \[\Gr^{\mu}_G \longrightarrow G/P_{\mu}=(\Gr^{\mu}_G)^{\G_m} \] where $P_{\mu}$ is the parabolic subgroup associated to $\mu$. Explicitly, $P_{\mu}$ is the subgroup of $G$ generated by the root subgroups $U_{\check{\alpha}} \subseteq G$ for which the root $\check{\alpha}$ satisfies $\langle \check{\alpha},\mu \rangle \leq 0$. As in the case of $X_w$, the renormalized IC sheaf for $\newoverline{\Gr}^{\lambda}_G$ is $K$-equivariant and hence belongs to $\operatorname{SI}(\newoverline{\Gr}^{\lambda}_G)$ implying by Proposition \ref{semiinfiniteisusual} that we have $\IC^{\frac{\infty}{2}}(\newoverline{\Gr}^{\lambda}_G) \simeq \IC^{\operatorname{ren}}_{\newoverline{\Gr}^{\lambda}_G}$.

From the above examples, we see that $\IC^{\frac{\infty}{2}}(\cY)$ is likely to give something new only when the Bia{\l}ynicki-Birula strata are singular. The next lemma will be used for constructing interesting factorization algebras in this setting later. 

\begin{lemma}\label{factcomplete} Let $\cY$ be a $\G_m$-complete scheme equipped with a morphism $p: \cY \to \Conf$ and a factorization structure. Assume further that $p$ is $\G_m$-equivariant for the trivial action on $\Conf$, and that the factorization isomorphism is compatible with $\G_m$-actions on both sides. Then $\IC^{\frac{\infty}{2}}(\cY)$ has the structure of a factorization algebra. 
\end{lemma}
\begin{proof}
    This is an immediate consequence of point $3$ of Proposition \ref{tstructure}. Indeed, $\G_m$ acts on \[\cY \times_{\Conf} [\Conf \times \Conf]_{\disj}\] by acting on the left hand factor. The strata are of the form \[\cY^{\lambda,+}_{\mu} \times_{\Conf^{\lambda}} [\Conf^{\lambda'} \times \Conf^{\lambda''}]_{\disj}\] where $\cY^{\lambda,+}_{\mu}$ is an attracting locus in $\cY^{\lambda}$ and $\lambda=\lambda'+\lambda''$ (note here $\mu$ should simply be considered as indexing a stratum, not as a coweight). Now the result follows from the fact that \[\add: \Conf \times \Conf \to \Conf\] is \'{e}tale when restricted to the disjoint locus and point $3$ of Proposition \ref{tstructure} applied component-wise. 
\end{proof}

\subsection{The semi-infinite sheaf for compactified Zastava spaces.}\label{semiinfiniteforzastava} We will now consider the special case of the compactified Zastava $\pZ$, where we let $\G_m$ act via a regular dominant coweight $\G_m \to T$. Since $\pZ$ is proper over $\Spec(k)$, it is automatically $\G_m$-complete. Hence we may consider its t-structure together with its factorization algebra $\IC^{\frac{\infty}{2}}(\pZ)$. It is the goal of the next few sections to produce a canonical isomorphism \[\IC^{\frac{\infty}{2}}(\pZ) \overset{\sim}{\longrightarrow} \sIC \] of factorization algebras which extends the tautological identification over $\cZ^-$ (see Corollary \ref{cor: semiinfinite t-structures coincide}). Here the factorization structure on $\IC^{\frac{\infty}{2}}(\pZ)$ is obtained from Lemma \ref{factcomplete}.

\subsection{Attracting loci in the compactified Zastava.} Recall that we have a natural map \[\newoverline{\mathfrak{p}}: \pZ \longrightarrow \dBun_N \] which is $T$-equivariant. Define a stratification of $\pZ$ indexed by coweights in $\Lambda^-$ where the stratum $\pZ^{\lambda}_{=\mu}$ corresponding to $\mu$ in $\pZ^{\lambda}$ fits in a Cartesian square \[\begin{tikzcd}\label{basechangestrata}
	{\pZ^{\lambda}_{=\mu}} & {\pZ^{\lambda}} \\
	{\dBun^{=\mu}_N} & {\dBun_N}
	\arrow[from=1-1, to=1-2]
	\arrow[from=1-2, to=2-2]
	\arrow[from=1-1, to=2-1]
	\arrow[from=2-1, to=2-2]
\end{tikzcd}\] where we recall \begin{equation}\label{Nstratum} \dBun^{=\mu}_N \overset{\sim}{\longrightarrow} X^{\mu} \times_{\Bun_T} \Bun_B\end{equation} is the stratum of $\dBun_N$ classifying reductions with zeroes of total order exactly $\mu$. For example, the disjoint union 
\[\pZ_{=0} \coloneqq \coprod_{\lambda} \pZ_{=0}^{\lambda}\] of strata $\pZ^{\lambda}_{=0}$ is nothing other than the opposite affine Zastava $\cZ^-$. 

We also obtain a second stratification of $\pZ^{\lambda}$ whose strata are preimages of the strata $X^{\lambda-\mu} \times \Bun^{\mu}_{B^-}$ inside $\dBun_{B^-}$. Let ${_{=\mu}}{\pZ^{\lambda}}$ denote the $\mu$th stratum of $\pZ^{\lambda}$, by definition equal to the preimage of $X^{\lambda-\mu} \times \Bun^{\mu}_{B^-}$. In particular the $0$th stratum is $\cZ^{\lambda}$. We have the next easy but important lemma. 

\begin{lemma}\label{attractors}
    The strata $\pZ^{\lambda}_{=\mu}$ are the attracting loci in $\pZ^{\lambda}$. The corresponding repelling loci are the strata ${_{=\mu}}{\pZ^{\lambda}}$.
\end{lemma}
\begin{proof}
    By the isomorphism \ref{Nstratum}, we have an isomorphism \[\pZ^{\lambda}_{=\mu} \overset{\sim}{\longrightarrow} X^{\mu} \times_{\Bun_T} \cZ^{-,\lambda-\mu}_{\Bun_T},\] where $\cZ^{-,\lambda-\mu}_{\Bun_T}$ is the relative Zastava from \cite{BFGM}. Now our choice of coweight $\G_m \to T$ contracts $\cZ^{-,\lambda-\mu}_{\Bun_T}$ to $\Bun_T \times X^{\lambda-\mu}$, from which we conclude that $\G_m$ contracts $\pZ^{\lambda}_{=\mu}$ to the scheme $X^{\mu} \times X^{\lambda-\mu}$, which is obviously fixed by $\G_m$. 
    
    The statement about repelling loci follows from a similar argument using the isomorphism \[{_{=\mu}}{\pZ^{\lambda}} \overset{\sim}{\longrightarrow} \cZ^{\mu} \times X^{\lambda-\mu}\] obtained by the definition of the stratification. Again the fact that $\G_m$ repels this scheme from $X^{\mu} \times X^{\lambda-\mu}$ follows from our choice of coweight. 
\end{proof}

Let us denote by $\pi_{=\mu}: \pZ_{=\mu} \to X^{\mu} \times \Conf$ and $\sigma_{=\mu}: X^{\mu} \times \Conf \to \pZ_{=\mu}$ the resulting projection from attracting locus to fixed points and the natural embedding, respectively. We denote by $\pi^-_{=\mu}$ and $\sigma^-_{=\mu}$ the corresponding maps for the repelling loci ${_{=\mu}}{\pZ}$. 

\subsection{Zastava to global comparison.} Recall the category $\operatorname{SI}^{\leq 0}_{\operatorname{glob}}$ from \cite{Ga2}. It is the full subcategory of $\mathfrak{D}(\newoverline{\Bun}_N)$ whose $!$-pullback to each stratum $\newoverline{\Bun}^{=\lambda}_N$ is the $!$-pullback of an object on $X^{\lambda}$. As shown in \emph{loc. cit.}, $\operatorname{SI}^{\leq 0}_{\operatorname{glob}}$ is compatible with the perverse t-structure on $\mathfrak{D}(\newoverline{\Bun}_N)$ and thus also acquires a t-structure. 

Let us compare the normalized perverse t-structure on $\operatorname{SI}^{\leq 0}_{\operatorname{glob}}$ (shifted by $d=\operatorname{dim}(\Bun_N)=(g-1)\operatorname{dim}(N)$) and the semi-infinite t-structure on $\pZ^{\lambda}$ for an arbitrary $\lambda \in \Lambda^-$. The basic idea is that the semi-infinite t-structures on $\pZ^{\lambda}$ and $\dBun_N$ are set up in such a way that they ``think" the projection $\pZ^{\lambda} \to \dBun_N$ is smooth. 

\begin{proposition}\label{texactzastava}
    Pullback $\newoverline{\mathfrak{p}}^!$ along the projection \[\newoverline{\mathfrak{p}}: \pZ \to \dBun_N\] induces a t-exact functor \[\operatorname{SI}^{\leq 0}_{\operatorname{glob}} \longrightarrow \operatorname{SI}(\pZ^{\lambda})\] for every coweight $\lambda$. 
\end{proposition}
\begin{proof}
    For every $\mu \leq \lambda$, the following diagram commutes: \[\begin{tikzcd}
	{X^{\mu} \times X^{\lambda-\mu}} & {\pZ_{=\mu}} & {\pZ^{\lambda}} \\
	{X^{\mu}} & {\dBun_N^{=\mu}} & {\dBun_N}
	\arrow[from=1-2, to=1-3]
	\arrow["{\newoverline{\mathfrak{p}}}", from=1-3, to=2-3]
	\arrow[from=1-2, to=2-2]
	\arrow[from=2-2, to=2-3]
	\arrow[from=1-2, to=1-1]
	\arrow["{\operatorname{pr}_{X^{\mu}}}"', from=1-1, to=2-1]
	\arrow[from=2-2, to=2-1]
\end{tikzcd}\] It follows that $\newoverline{\mathfrak{p}}^!$ induces a functor on semi-infinite categories, by construction. Moreover, the projection $\dBun^{=\mu}_N \to X^{\mu}$ is a fibration with fiber at $D \in X^{\mu}$ given by the fiber $\Bun^{\mathscr{P}^0_T(-D)}_N$ of $\Bun_B$ over $\mathscr{P}^0_T(-D) \in \Bun_T$ by Corollary $2.2.9$ of \cite{FGV}. According to the discussion in \emph{loc.cit.}, we have \[ \operatorname{dim}(\Bun^{\mathscr{P}^0_T(-D)}_N)=\operatorname{dim}(\Bun_N)+\langle \check{\rho}, 2\mu \rangle.\] It follows that for an object $c \in \operatorname{SI}^{= \mu}_{\operatorname{glob}}$ sitting in semi-infinite degrees $\geq -\operatorname{dim}(\Bun_N)$ which is pulled back from $c' \in \fD(X^{\mu})$, then $c'$ lies in cohomological degrees $\geq \langle \check{\rho},2\mu \rangle$. In other words, $c'[\langle \check{\rho},2\mu \rangle ]$ is coconnective in the perverse t-structure. 

Now the relative dimension of $\operatorname{pr}_{X^{\mu}}$ is $\langle \check{\rho},\mu-\lambda \rangle$, and therefore we have \[\operatorname{pr}^!_{X^{\mu}}(c') \in \fD(X^{\mu} \times X^{\lambda-\mu})^{\geq \langle \check{\rho}, 2\mu-\mu+\lambda \rangle}\] where we emphasize $\langle \check{\rho},2\mu-\mu+\lambda \rangle= \langle \check{\rho},\lambda+\mu \rangle$. Since the relative dimension of \[ \pi_{=\mu}: \pZ_{=\mu} \longrightarrow X^{\mu} \times X^{\lambda-\mu}\] is also $\langle \check{\rho},\mu-\lambda \rangle$ and the object $\operatorname{pr}_{X^{\mu}}^!(c')[\langle \check{\rho}, \mu -\lambda \rangle ]$ lies in lies perverse cohomological degrees \[ \geq \langle \check{\rho},\lambda+\mu - \mu + \lambda \rangle = \langle \check{\rho}, 2\lambda \rangle = -\operatorname{dim}(\pZ^{\lambda})\]we conclude that $\newoverline{\mathfrak{p}}^!$ is left t-exact as desired. 

To see that $\newoverline{\mathfrak{p}}^!$ is right t-exact, consider the following diagram \[\begin{tikzcd}
	& {X^{\mu} \times X^{\lambda-\mu}} & {\cZ^{\mu} \times X^{\lambda-\mu}} & {\pZ^{\lambda}} \\
	{X^{\mu}} & {\dBun^{=\mu}_N} & {\dBun_N} & {\dBun_N}
	\arrow["{j^-_{=\mu}}", from=1-3, to=1-4]
	\arrow["{\newoverline{\mathfrak{p}}}", from=1-4, to=2-4]
	\arrow["{\mathfrak{p} \circ \operatorname{pr}_{\cZ^{\mu}}}", from=1-3, to=2-3]
	\arrow["{\operatorname{id}_{\dBun_N}}"', from=2-3, to=2-4]
	\arrow["{j_{=\mu}}"', from=2-2, to=2-3]
	\arrow["{\sigma^-_{=\mu}}", from=1-2, to=1-3]
	\arrow["{\iota_{=\mu} \circ \operatorname{pr}_{X^{\mu}}}", from=1-2, to=2-2]
	\arrow[from=2-2, to=2-1]
	\arrow[from=1-2, to=2-1]
\end{tikzcd}\] Let $c \in \operatorname{SI}^{\leq 0}_{\operatorname{glob}}$ be connective. By Braden's theorem, the $*$-restriction $j^*_{=\mu}\newoverline{\mathfrak{p}}^!(c)$ is canonically equivalent to $(\sigma^-_{=\mu})^*(\mathfrak{p} \circ \operatorname{pr}_{\cZ^{\mu}})^!(c)$. By Section $3.7.8$ of \cite{Ga2}, the latter is equivalent to $(\iota_{=\mu} \circ \operatorname{pr}_{X^{\mu}})^!j^*_{=\mu}(c)$ (see also the proof of Proposition \ref{parabolictstr} for a sketch of the argument in the case of a general parabolic). Now the argument follows a similar one as the one for coconnective objects by flipping the inequalities.
\end{proof}

%\begin{corollary}
   % There is a canonical isomorphism $\IC^{\frac{\infty}{2}}(\pZ^{\lambda}) \simeq \sIC$ over $\pZ^{\lambda}$ extending the tautological isomorphism over $\cZ^{-,\lambda}$. 
%\end{corollary}
%\begin{proof}
  % Note that the projection $\newoverline{\mathfrak{p}}: \pZ \to \dBun_N$ commutes with the defect action of $\Conf$ on $\pZ$ and the trivial action on $\dBun_N$, and therefore upgrades to a morphism $\newoverline{\mathfrak{p}}: \pZ^{\to} \to \dBun_N$. Hence we have a commutative diagram \[\begin{tikzcd}
	%& {\pZ^{\to}_{\corr}} \\
	%{\pZ^{\to}} && {\newoverline{S}^{0}_{\Ran}} \\
	%& {\newoverline{\Bun}_N}
	%\arrow[from=1-2, to=2-1]
	%\arrow[from=1-2, to=2-3]
	%\arrow["{\newoverline{\mathfrak{p}}}"', from=2-1, to=3-2]
	%\arrow["{\newoverline{\mathfrak{p}}^0_{\Ran}}", from=2-3, to=3-2]
%\end{tikzcd}\]
%Recall from the proof of Proposition \ref{effectivetounital} that $\pZ^{\to}_{\corr}$ can be endowed with the structure of a lax prestack whose straightening is equivalent to both $(\pZ^{\to})_{\str}$ and $(\pZ^{\to}_{\Ran})_{\str}$. By the proof of Theorem \ref{effisun} the latter is equivalent to $(\newoverline{S}^0_{\Ran})_{\str}$.  

%It remains to check that $\sICR$ is isomorphic to $\mathscr{F}(\sIC)$ as a factorization algebra. However, it is easy to check that the isomorphism $\mathscr{F}(\sIC) \overset{\sim}{\to} \sICR$ as sheaves extends the tautological isomorphism of factorization algebras over $S^0_{\Ran}$. Since the factorization structure on $\sICR$ is uniquely determined by its restriction to $S^0_{\Ran}$, we are done. 
%\end{proof}

\subsection{Semi-infinite t-structure on effective sheaves.} For each $\mu \in \Lambda^-$, the morphisms \[\pi_{=\mu}: \pZ_{=\mu} \longrightarrow X^{\mu} \times \Conf, \,\,\, \text{and} \,\,\, \sigma_{=\mu}: X^{\mu} \times \Conf \longrightarrow \pZ_{=\mu}\] are $\Conf$-equivariant, where $\Conf$ acts on $X^{\mu} \times \Conf$ on the second factor. The pullback $(\pi_{=\mu})^!$ therefore induces a fully-faithful functor \[\fD(X^{\mu}) \overset{\sim}{\longrightarrow} \fD_{\eff}(X^{\mu} \times \Conf) \longrightarrow \fD_{\eff}(\pZ_{=\mu})\] of invariant categories. Let $\operatorname{SI}^{=\mu}_Z$ be the essential image of this functor and write $f_{\mu}$ for the tautological equivalence between $\fD(X^{\mu})$ and $\operatorname{SI}^{=\mu}_Z$.

%Note that the $\G_m$-action and the $\Conf$-action on $\pZ$ commute. Hence $\operatorname{SI}(\pZ)$ inherits an action of $\fD(\Conf)$, where we view the latter as a symmetric monoidal category via convolution. As a result, we may consider the category \[\operatorname{SI}^{\leq 0}_Z \coloneqq \operatorname{SI}_{\eff}(\pZ) \coloneqq \Hom_{\fD(\Conf)}(\operatorname{Vect},\operatorname{SI} \] of categorical \emph{invariants}, where $\fD(\Conf)$ is considered a monoidal category via convolution. Equivalently, $\operatorname{SI}^{\leq 0}_Z$ may be thought of as the full subcategory of $\fD_{\eff}(\pZ)$ which is the preimage of $\operatorname{SI}(\pZ)$ under the forgetful functor $\fD_{\eff}(\pZ) \to \fD(\pZ)$. 

%For each $\mu$ we obtain an analogous subcategory $\operatorname{SI}^{=\mu}_Z$ of effective sheaves on $\pZ_{=\mu}$. Since the projection \[\pZ_{=\mu} \longrightarrow X^{\mu} \times \Conf\] is $\Conf$-equivariant, we obtain equivalences \[\fD(X^{\mu}) \overset{\sim}{\longrightarrow} \fD_{\eff}(X^{\mu} \times \Conf) \overset{\sim}{\longrightarrow} \operatorname{SI}^{=\mu}_Z\] where the first map is induced by pullback along the projection to $X^{\mu}$, and the second is the tautological equivalence. 

Using the equivalence \[f_{\mu}: \fD(X^{\mu}) \overset{\sim}{\longrightarrow} \operatorname{SI}^{=\mu}_Z\] of the previous paragraph, we may define a t-structure on $\operatorname{SI}^{=\mu}_Z$ by declaring an object $c$ to be connective/coconnective if $f_{\mu}^{-1}(c)[\langle \check{\rho}, 2\mu \rangle ]$ is connective/coconnective\footnote{We emphasize here that $\mu$ is a \emph{negative} coweight.}. It follows that the forgetful functor \[\operatorname{SI}^{=\mu}_Z \longrightarrow \operatorname{SI}(\pZ_{=\mu})\] is t-exact. 

Since the $\G_m$ and $\Conf$ actions on $\pZ$ commute, we have an action of $\fD(\Conf)$ on the monodromic category $\fD(\pZ)^{\G_m-\operatorname{mon}}$, where we view $\fD(\Conf)$ as a symmetric monoidal category via convolution. Hence we may consider the category \[\fD_{\eff}(\pZ)^{\G_m-\operatorname{mon}} \coloneqq \Hom_{\fD(\Conf)}\big(\operatorname{Vect},\fD(\pZ)^{\G_m-\operatorname{mon}}\big)\] of invariants. Equivalently, $\fD_{\eff}(\pZ)^{\G_m-\operatorname{mon}}$ is the full subcategory of $\fD_{\eff}(\pZ)$, equal to the preimage of $\fD(\pZ)^{\G_m-\operatorname{mon}}$ under the forgetful functor. Now the embeddings $j_{\mu}: \pZ_{=\mu} \hookrightarrow \pZ$ are $\Conf$-equivariant for every $\mu$, and so we define a full subcategory \[ \operatorname{SI}^{\leq 0}_Z \longhookrightarrow \fD_{\eff}(\pZ)^{\G_m-\operatorname{mon}}\] consisting of objects $c$ such that $j^!_{=\mu}(c)$ belongs to $\operatorname{SI}^{=\mu}_Z$ for every $\mu$. 

\begin{proposition}
    The left adjoint of the pushforward \[(j_{\mu})_*: \operatorname{SI}^{=\mu}_Z \longrightarrow \operatorname{SI}^{\leq 0}_Z \] is defined and as a result we obtain a t-structure on $\operatorname{SI}^{\leq 0}_Z$ such that the forgetful functor $\operatorname{SI}^{\leq 0}_Z \to \operatorname{SI}(\pZ)$ is t-exact.
\end{proposition}
\begin{proof}
    By construction, it is enough to verify that the functor $(j_{\mu})_*(\pi^-_{\mu})^!$ has a left adjoint. Note that the embeddings \[{_{=\mu}}{\pZ} \longhookrightarrow \pZ \,\,\, \text{and} \,\,\, \pZ_{=\mu} \longhookrightarrow \pZ \] are maps of $\Conf$-spaces. The result then follows from $\Conf$-equivariance together with an application of Braden's theorem. 
\end{proof}

We will consider its perverse t-structure, renormalized so that, for example, $\omega_{\Bun_N}$ is in the heart of $\operatorname{SI}^{=0}_{\operatorname{glob}}$. More precisely, an object $c$ is connective/coconnective if $c[-\operatorname{dim}(\Bun_N)]$ is connective/coconnective. it is clear from the definitions that the pullback $\newoverline{\mathfrak{p}}^!$ induces a functor \[\operatorname{SI}^{\leq 0}_{\operatorname{glob}} \longrightarrow \operatorname{SI}^{\leq 0}_Z.\] We are finally ready to prove the main result of this section. 

\begin{theorem}\label{texactequivalence}
    The functor \[\newoverline{\mathfrak{p}}^!: \operatorname{SI}^{\leq 0}_{\operatorname{glob}} \longrightarrow \operatorname{SI}^{\leq 0}_Z\] is a t-exact equivalence. In particular, we have a canonical isomorphism \[\newoverline{\mathfrak{p}}^!(\IC^{\operatorname{ren}}_{\dBun_N}) \simeq \IC^{\frac{\infty}{2}}(\pZ)\] extending the tautological identifications over $\cZ^-$. 
\end{theorem}
\begin{proof}
    By Proposition \ref{torsor}, the functor $\newoverline{\mathfrak{p}}^!$ is fully faithful. Hence it suffices to show that it is essentially surjective. The category $\operatorname{SI}^{\leq 0}_Z$ is generated under colimits by objects of the form $(j_{\mu})_!(c)$, where $c$ is pulled back along the projection \[ \operatorname{pr}_{X^{\mu}} \circ \pi^-_{\mu}: \pZ_{=\mu} \longrightarrow X^{\mu}.\] Moreover, the functor \[\newoverline{\mathfrak{p}}_{\mu}^!: \operatorname{SI}^{=\mu}_{\operatorname{glob}} \longrightarrow \operatorname{SI}^{=\mu}_Z \] induced by pullback along the natural map $\newoverline{\mathfrak{p}}_{\mu}: \pZ_{=\mu} \to \dBun^{=\mu}_N$ is clearly a t-exact equivalence. Now the result follows from an argument similar to Proposition \ref{texactzastava} or by reduction to that case.
    
    %By a Cousin argument it therefore suffices to show that the natural map $\newoverline{\mathfrak{p}}_{\mu}^! \circ j^*_{\mu} \to j^*_{\mu} \circ \newoverline{\mathfrak{p}}^!$ induced by base-change along the diagram \eqref{basechangestrata} is an isomorphism. Here we are abusing notation and using $j_{\mu}$ to also denote the inclusion \[\dBun^{=\mu}_N \longhookrightarrow \dBun_N\] of the $\mu$th stratum on Drinfeld's compactification.

    %It suffices to show that the base-change map is an isomorphism after restriction to $\G_m$-fixed points. By Braden's theorem the functor $j^*_{\mu} \circ \newoverline{\mathfrak{p}}^!$ followed by restriction to fixed points is equivalent to $\sigma^* \circ (j^-_{\mu})^! \circ \newoverline{\mathfrak{p}}^!$ where the maps are given in the following diagram: \[\begin{tikzcd}
	%{X^{\mu} \times \Conf} & {\cZ^{\mu} \times \Conf} & \pZ \\
	%{\dBun^{=\mu}_N} & {\dBun_N} & {\dBun_N}
	%\arrow["{j^-_{\mu}}", from=1-2, to=1-3]
	%\arrow["{\newoverline{\mathfrak{p}}}", from=1-3, to=2-3]
	%\arrow[from=1-1, to=2-1]
	%\arrow["{\sigma_{\mu}}", from=1-1, to=1-2]
	%\arrow["{\mathfrak{p} \circ \operatorname{pr}_{\cZ^{\mu}}}"', from=1-2, to=2-2]
	%\arrow["{=}"', from=2-2, to=2-3]
	%\arrow["{j_{\mu}}"', from=2-1, to=2-2]
%\end{tikzcd}\] Now the result follows from a standard Zastava argument for $\cZ^{\mu}$ mapping to $\dBun_N$ (see e.g. Section $3.7.8$ of \cite{Ga2}). 
\end{proof}

As a corollary, we obtain an alternative construction of Gaitsgory's semi-infinite t-structure on the unital subcategory of $\operatorname{SI}^{\leq 0}_{\Ran}$. 

\begin{corollary}\label{cor: semiinfinite t-structures coincide}
    The equivalence of Theorem \ref{effisun} is t-exact when restricted to semi-infinite categories. In particular, we have a canonical isomorphism $\sIC \simeq \IC^{\frac{\infty}{2}}(\pZ)$ of factorization algebras.  
\end{corollary}
\begin{proof}
    We have a commutative diagram \[\begin{tikzcd}
	{\pZ/\Conf} & {\newoverline{S}^0_{\Ran}/\Ran} \\
	& {\dBun_N}
	\arrow[from=1-1, to=1-2]
	\arrow[from=1-2, to=2-2]
	\arrow[from=1-1, to=2-2]
\end{tikzcd}\] where the horizontal arrow is an equivalence of fppf sheaves. 

The equivalence of Theorem \ref{effisun} is obtained by $!$-pullback along the horizontal arrow, whereas the equivalence of Theorem \ref{texactequivalence} is obtained by $!$-pullback along the diagonal arrow. By Theorem $4.3.2$ of \cite{Ga2}, pullback along the vertical arrow induces a t-exact equivalence on semi-infinite categories. The fact that the resulting isomorphism $\sIC \simeq \IC^{\frac{\infty}{2}}(\pZ)$ is compatible with factorization can be checked simply by noting that it extends the tautological isomorphism over $\cZ^-$. Indeed, the factorization structures on both sides are uniquely determined by the fact that they extend the tautological one on $\omega_{\cZ^-}$.
\end{proof}

\section{Resolution of attracting loci via Kontsevich Zastava spaces}\label{kontsevich}

In this section we will introduce a new version $\pZ_K$ of the compactified Zastava space, together with its factorization structure. We will equip the space $\pZ_K$ with a proper birational map $\mathfrak{r}_Z: \pZ_K \to \pZ$ and show that the fiber of $\mathfrak{r}_Z$ over an attracting locus $\pZ_{=\mu}$ is a resolution of singularities. Lastly, we will conclude that the semi-infinite IC sheaf on $\pZ$ pulls back to the ordinary (renormalized) IC sheaf on $\pZ_K$. 

\subsubsection{}\label{introofKont} Recall the resolution $\dBun^K_B \to \dBun_B$ of singularities of $\dBun_B$ over $\Bun_G$ by stable maps introduced in \cite{C}. All results mentioned here about $\dBun^K_B$ can be found in \emph{loc. cit.} An $S$-point of $\dBun^K_B$ is a commutative diagram 
\begin{equation}\begin{tikzcd}\label{eq: kontsevich}
	C & {\pt/B} \\
	{X \times S} & {\pt/G}
	\arrow[from=1-2, to=2-2]
	\arrow[from=1-1, to=1-2]
	\arrow["p"', from=1-1, to=2-1]
	\arrow["\sP_G", from=2-1, to=2-2]
\end{tikzcd}\end{equation}
where $C \to S$ is a flat family of projective nodal curves of arithmetic genus $g=\operatorname{genus}(X)$, $p$ has degree $1$, and the induced map $C \to (G/B)_{p^*\sP_G}$ is stable over each geometric point of $S$. One shows that $\dBun^K_B$ is a smooth, locally of finite-type algebraic stack containing $\Bun_B$ as a dense open substack. 

The stack $\dBun^K_B$ is equipped with a proper morphism \[r: \dBun^K_B \longrightarrow \dBun_B\] which is an isomorphism when restricted to $\Bun_B$. To define $r$, take an $S$-point of $\dBun^K_B$ as in \eqref{eq: kontsevich}. Then the map $C \to \pt/B$ induces nondegenerate Pl\"{u}cker maps \[p^{\chlam}_K:\chlam(\mathscr{P}_T) \longrightarrow V^{\chlam}_{p^*\sP_G}\] for every dominant weight $\chlam \in \check{\Lambda}^+$. Define a point of $\dBun_B$ where the Pl\"{u}cker map corresponding to $\chlam$ is given by \[\operatorname{det}(p^{\chlam}_K): \operatorname{det}(p_*\chlam(\mathscr{P}_T)) \longrightarrow V^{\chlam}_{\sP_G}\] where $\operatorname{det}$ denotes the determinant operation on perfect complexes. Clearly, over the locus where $p$ is an isomorphism, the maps $\operatorname{det}(p^{\chlam}_K)$ are embeddings of vector bundles. 

Let $\ovF^K_{N,B^-}$ denote the fiber product $\dBun_N \times_{\Bun_G} \dBun^K_{B^-}$. Define the \emph{Kontsevich compactified Zastava space} $\pZ_K$ by requiring that it sits in a Cartesian square
\[\begin{tikzcd}
	{\pZ_K} & {\ovF^K_{N,B^-}} \\
	\pZ & {\ovF_{N,B^-}}
	\arrow[hook, from=1-1, to=1-2]
	\arrow["\operatorname{id}_{\dBun_N} \times r", from=1-2, to=2-2]
	\arrow["\mathfrak{r}_Z"', from=1-1, to=2-1]
	\arrow[hook, from=2-1, to=2-2]
\end{tikzcd}\] and the bottom arrow is the usual open embedding. There are the analogous versions $\cZ_K$ and $\cZ^-_K$ together with open embeddings $j_K$ and $j^-_K$, respectively, into $\pZ_K$. Note that $\mathfrak{r}_Z$ is proper and birational, and the open locus on which the map $\mathfrak{r}_Z$ is an isomorphism is given by the affine Zastava $\cZ \simeq \cZ_K$. 

\begin{remark} The space $\cZ^-_K$ appeared in \cite{FFKM} as a resolution for $\cZ^-$ over the affine curve $\mathbb{A}^1$. However, to our knowledge the present paper is the first time the space $\pZ_K$ has been considered in the literature. Proposition \ref{kontsevichfactorization} also appears to be new, even for $\cZ^-_K$.
\end{remark}

Note there is a map $\pZ_K \to \Conf$, obtained by composing $\mathfrak{r}_Z$ with the defect map $\pZ \to \Conf$. The next proposition establishes the factorization property of the Kontsevich Zastava spaces relative to $\Conf$, and will be an important ingredient in the remainder of the paper. 

\begin{proposition}\label{kontsevichfactorization}
The Kontsevich Zastava space $\pZ_K$ has a factorization structure which commutes with that of $\pZ$. The same holds for the open subspaces $\cZ_K$ and $\cZ^-_K$, and their open embeddings into $\pZ_K$ are compatible with factorization. 
\end{proposition}
\begin{proof}
Let $S$ be an affine scheme. For all $\lambda, \mu \in \Lambda^-$ we will only define a map 
\[[ \pZ^{\lambda}_K \times \pZ^{\mu}_K]_{\disj} \longrightarrow \pZ^{\lambda+\mu}_{K,\disj}\] since by the construction it will be clear how define its inverse. Consider an $S$-point $(z,z')$ of the left hand side. Then $z$ consists of a $G$-bundle $\sP_G \to X \times S$, a generalized $N$-reduction $\{\kappa^{\chlam}\}_{\chlam \in \check{\Lambda}^+}$ of $\sP_G$, and a point $z_K$ of $\dBun^K_{B^-}$ whose underlying generalized $B^-$-reduction of $\sP_G$ is generically transverse to $\{\kappa^{\chlam}\}_{\chlam \in \check{\Lambda}^+}$. Write the point $z_K$ of $\dBun^K_{B^-}$ as a commutative diagram
\[\begin{tikzcd}
	C & {\pt/B} \\
	{X \times S} & {\pt/G}
	\arrow[from=1-2, to=2-2]
	\arrow[from=1-1, to=1-2]
	\arrow["p"', from=1-1, to=2-1]
	\arrow[from=2-1, to=2-2]
\end{tikzcd}\]
where $C$ is a flat, connected family of nodal curves over $S$, the map $p$ has degree $1$, and $C \to (G/B)_{\sP_G}$ is stable over every geometric point of $S$. Generic transversality identifies $z_K$ with the trivial diagram over a domain $U_z \subseteq X \times S$, complement to the $S$-family of colored divisors associated to $z$. For the point $z'$ we get a similar curve $p': C' \to X \times S$ with locus $U_{z'}$ over which $p'$ is an isomorphism. 

Let $C_{U_{z'}}$ denote the fiber product $U_{z'} \times_{X \times S} C$. Note that by disjointness, the curve $C_{U_{z'}}$ contains all the nodes of $C$. Likewise, define $C'_{U_z}=U_z \times_{X \times S} C'$ which contains all the nodes of $C'$. Write $C_{U_z \cap U_{z'}}$ (resp. $C'_{U_z \cap U_{z'}}$) for the fiber product $U_z \cap U_{z'} \times_{X \times S} C$ (resp. $U_z \cap U_{z'} \times_{X \times S} C'$). There is a unique isomorphism \[C_{U_z \cap U_{z'}} \overset{\sim}{\longrightarrow} C'_{U_z \cap U_{z'}}\] over $U_z \cap U_{z'}$ since the latter is contained in the locus over which $p$ and $p'$ are both isomorphisms. 

By the assumption that the projections of $z$ and $z'$ to the configuration space have disjoint support we obtain a pushout diagram
\[\begin{tikzcd}
	{C_{U_z \cap U_{z'}} } & {C_{U_{z'}}} \\
	{C'_{U_{z}}} & C_{z,z'}
	\arrow[from=1-2, to=2-2]
	\arrow[from=1-1, to=1-2]
	\arrow[from=1-1, to=2-1]
	\arrow[from=2-1, to=2-2]
\end{tikzcd}\]
with a map $p_{z,z'}: C_{z,z'} \to X \times S$. Note that $C_{z,z'}$ is flat over $S$ by construction. Furthermore, since $C'_{U_z} \subseteq C'$ (resp. $C_{U_{z'}} \subseteq C$) intersects each fiber of $C' \to S$ (resp. $C \to S$) nontrivially and in the nonsingular locus, the scheme $C_{z,z'}$ is a proper family of nodal curves relative to $S$. 

From the fact that $C \to \pt/B$ and $C' \to \pt/B$ factor through the canonical map $\pt \to \pt/B$ when restricted to $C_U$, we obtain a commutative diagram 
\[\begin{tikzcd}
	C_{z,z'} & {\pt/B} \\
	{X \times S} & {\pt/G}
	\arrow[from=1-2, to=2-2]
	\arrow[from=1-1, to=1-2]
	\arrow["p'"', from=1-1, to=2-1]
	\arrow[from=2-1, to=2-2]
\end{tikzcd}\]
with $C_{z,z'} \to (G/B)_{\sP_G}$ stable over each geometric point $s \in S$. Indeed, the locus along which we glue $C_{U_z}$ and $C_{U_{z'}}$ does not contain any nodes.
\end{proof}

We will also record the following lemma for future use. 

\begin{lemma}\label{zastavaisdm} The Kontsevich Zastava $\pZ_K$ is a Deligne-Mumford stack. In particular, it has a schematic and proper diagonal. Moreover, $\pZ_K$ is a union of connected components $\pZ^{\lambda}_K$ of dimension $-\langle \lambda, 2 \check{\rho} \rangle$, where $\lambda \in \Lambda^-$.
\end{lemma}
\begin{proof} It is known that the fiber of the map $\dBun_{B^-}^K \to \Bun_G$ over an affine scheme $S$ is a Deligne-Mumford stack. Hence the same is true of the fiber of the resolution $\dBun_{B^-}^K \to \dBun_{B^-}$, as well as of the map $\newoverline{\mathscr{F}}^K_{N,B^-} \to \newoverline{\mathscr{F}}_{N,B^-}$. Since $\pZ$ is a scheme, we see that $\pZ_K$ is a Deligne-Mumford stack. Lastly, the dimension calculation for $\pZ_K^{\lambda}$ follows from the corresponding well-known fact for $\pZ^{\lambda}$.
\end{proof}

\subsubsection{} By \cite{C}, taking \emph{defects} of $k$-points of $\dBun^K_{B^-}$ gives us a stratification indexed by elements of $\Lambda^-$. More precisely, in \emph{loc. cit.} the defect of a $k$-point \[\eta: \Spec(k) \longrightarrow \dBun^K_{B^-}\] is defined as follows. Recall that to $\eta$ we can associate a canonical section of the projection $p:C \to X$. Identify $X$ with its image in $C$ and let $\overset{\circ}{X}$ be the intersection of $X$ with the smooth locus of $C$. Denote by $C_1, \ldots , C_n$ the connected components of the curve $C \setminus \overset{\circ}{X}$ and by $x_1, \ldots, x_n$ the points $C_i \cap X$. The defect is then the colored divisor $-\sum_{i=1}^n \operatorname{deg}(\mathscr{P}_{B^-} |_{C_i}) x_i$, where $\mathscr{P}_{B^-}$ is the $B^-$-bundle on $C$ associated to $\eta$ and $\mathscr{P}_{B^-}|_{C_i}$ is its restriction to $C_i$.  

Denote by ${_{\leq \mu}}{\dBun^{K,\lambda}_{B^-}}$ the open locus of $\dBun^{K,\lambda}_{B^-}$ which is a union of strata indexed by coweights greater than or equal to $\mu$. We therefore obtain open subsets ${_{\leq \mu}}{\pZ^{\lambda}_K}$ and ${_{\leq \mu}}{\cZ^{-,\lambda}_K}$ of the Kontsevich Zastavas. 

Now let $\pZ_{K,=\mu}$ denote the fiber product \[ \pZ_{K,=\mu} \coloneqq \pZ_K \times_{\pZ} \pZ_{=\mu} \] where we recall that $\pZ_{=\mu}$ is an attracting locus for the standard $\G_m$-action on $\pZ$. We have an equivalence \[\pZ_{K,=\mu} \overset{\sim}{\longrightarrow} X^{\mu} \times_{\Bun_T} \cZ^-_{K,\Bun_T} \] where $\cZ^-_{K,\Bun_T}$ is the open locus of $\Bun_B \times_{\Bun_G} \dBun^K_{B^-}$ determined by generic transversality. In what follows, we will also denote by $\oZ_{\Bun_T}$ the open locus of $\Bun_B \times_{\Bun_G} \Bun_{B^-}$ determined by generic transversality (recall the open Zastava $\oZ$ defined in Section \ref{definition of zastava}). 

Let $\mathfrak{r}_{Z,=\mu}$ denote the projection from $\pZ_{K,=\mu}$ to $\pZ_{=\mu}$. We therefore obtain a stratification of $\pZ_K$ into locally closed subschemes compatible with the stratification of $\pZ$. Proposition \ref{zasissmooth} below shows that this stratification is a genuine stratification into smooth subschemes. 

\begin{proposition}
\label{zasissmooth}
The strata $\pZ_{K,=\mu}$ are smooth for every $\mu \in \Lambda^-$, and therefore the maps \[\mathfrak{r}_{Z,=\mu}: \pZ_{K,=\mu} \longrightarrow \pZ_{=\mu}\] are resolutions of singularities. The loci over which $\mathfrak{r}_{Z,=\mu}$ is an equivalence is given by the open subscheme $X^{\mu} \times_{\Bun_T} \oZ_{\Bun_T}$ of $\pZ_{=\mu}$.
\end{proposition}
\begin{proof} For simplicity, first consider the case $\mu=0$ so that $\pZ_{=\mu}=\cZ^-$ and $\pZ_{K,=\mu}=\cZ^-_K$. Fix $\nu \in \Lambda^-$. We claim that for $\lambda$ sufficiently dominant (see \cite{C} for a precise bound), the open locus ${_{\leq \nu}}{\cZ^{-,\lambda}_K}$ of $\cZ^{-,\lambda}_K$ is smooth. Indeed, by the hypothesis on $\lambda$, the projection ${_{\leq \nu}}{\dBun^{K,\lambda}_{B^-}} \to \Bun_G$ is smooth. By base-change, so is the map ${_{\leq \nu}}{\cZ^{-,\lambda}_K} \to \Bun_N$. But $\Bun_N$ is smooth, and hence so is ${_{\leq \nu}}{\cZ^{-,\lambda}_K}$. 

Now let $\nu'$ be large enough so that $\nu'+\nu$ is sufficiently dominant. Factorization provides an \'{e}tale morphism \[ [\oZ^{\nu'} \times \cZ^{-,\nu}_K]_{\disj} \longrightarrow \cZ^{-,\nu'+\nu}_K\] and since the degeneracy of the $B^-$-reduction is bounded by $\nu$, this map factors through ${_{\leq \nu}}{\cZ^{-,\nu'+\nu}_K}$. We can therefore conclude that the left side is smooth. To finish, note the projection \[[\oZ^{\nu'} \times \cZ^{-,\nu}_K]_{\disj} \longrightarrow \cZ^{-,\nu}_K\] is smooth and surjective, whence the result. 

For a general stratum, note that $\pZ_{K,=\mu}$ has a natural factorization structure relative to the base $X^{\mu} \times \Conf$, where we view the latter as a commutative algebra in correspondences \emph{over} $X^{\mu}$ via the diagram \[\begin{tikzcd}
	& {X^{\mu} \times [\Conf \times \Conf]_{\disj}} \\
	{(X^{\mu} \times \Conf) \times_{X^{\mu}} (X^{\mu} \times \Conf)} && {X^{\mu} \times \Conf}
	\arrow["{\operatorname{id}_{X^{\mu}} \times \operatorname{add}}"', from=1-2, to=2-3]
	\arrow["{ }", from=1-2, to=2-1]
\end{tikzcd}\] and where the unitality diagram is constructed in the obvious way. The proof then follows the same strategy as the case of $\cZ_K^-$ above. In other words, fixing $\nu \in \Lambda^-$ such that $\nu \geq \mu$, there is $\lambda$ large enough so that the projection \[{_{\leq \nu}}{\cZ^{-,\lambda}_{K,\Bun_T}} \longrightarrow \Bun_B\] is smooth. Since the map $\Bun_B \to \Bun_T$ is also smooth, we conclude that the composition \[{_{\leq \nu}}{\cZ^{-,\lambda}_{K,\Bun_T}} \longrightarrow \Bun_B \longrightarrow \Bun_T\] is as well. Since $X^{\mu}$ is smooth, so is $X^{\mu} \times_{\Bun_T} {_{\leq \nu}}{\cZ^{-,\lambda}_{K,\Bun_T}}$. 

Choose $\nu' \geq \mu$ so that $\nu'+\nu$ is sufficiently dominant. By factorization, we obtain an \'{e}tale map \[[\big(X^{\mu} \times_{\Bun_T} \oZ^{\nu'}_{\Bun_T}\big) \times_{X^{\mu}}\big( X^{\mu} \times_{\Bun_T} \cZ^{-,\nu}_{K,\Bun_T}\big)]_{\disj} \longrightarrow X^{\mu} \times_{\Bun_T} {_{\leq \nu}}{\cZ^{-,\nu'+\nu}_{K,\Bun_T}}\] and so it is enough to show that $\alpha: \oZ^{\nu'}_{\Bun_T} \to \Bun_T$ is smooth for all $\nu'$. By \cite{BFGM}, the fibers of the map $\alpha$ are all \'{e}tale-locally equivalent to $\oZ^{\nu'}$. But $\alpha$ is schematic, and thus the result follows from miracle flatness and smoothness of $\oZ^{\nu'}$. 
\end{proof}

\begin{remark}\label{density} An argument similar to the proof of Proposition \ref{zasissmooth} shows that $\cZ^-_K$ is dense in $\pZ_K$. Indeed, the key point is that the preimage of a dense open substack under a smooth (and therefore open) map is dense.
\end{remark}

\subsubsection{} We are ready to prove our main theorem of this section. We will start by fixing some notation. In what follows, denote by \[\newoverline{\mathfrak{p}}_K: \pZ_K \longrightarrow \dBun_N\] the projection obtained via base change from the map $\dBun_{B^-}^K \to \Bun_G$. 

To make the exposition cleaner, for an algebraic stack $\cY$ with a map $\pi_{\cY}$ to $\Conf$ such that each preimage \[\cY^{\lambda} \coloneqq \pi^{-1}_{\cY}(\Conf^{\lambda})\] is connected and of finite type, let us denote by $\IC^{\operatorname{ren}}_{\cY}$ the sheaf which is the renormalized IC sheaf on each component of $\cY$. In particular, when $\cY=\pZ_K$ we have \[ \IC^{\operatorname{ren}}_{\pZ_K^{\lambda}}=\IC_{\pZ^{\lambda}_K}[-\langle \lambda, 2\check{\rho} \rangle ], \] where we recall that $-\langle \lambda, 2\check{\rho} \rangle $ is a nonnegative integer, and positive whenever $\lambda \neq 0$. Note that $\IC^{\operatorname{ren}}_{\pZ_K}$ is naturally a factorization algebra. 

\begin{theorem}
\label{sicisic}
The pullback $\mathfrak{r}_Z^! (\sIC)$ is canonically isomorphic to the renormalized intersection cohomology sheaf $\IC^{\operatorname{ren}}_{\pZ_K}$ as a factorization algebra. 
\end{theorem}
\begin{proof}
Fix $\mu \in \Lambda^-$ and let $\lambda$ be sufficiently dominant for any $\nu \leq \mu$. Then by \cite{C} the projection \[{_{\leq \mu}}{\dBun^{K,\lambda}_{B^-}} \longrightarrow \Bun_G\] is smooth, and hence so is the restriction \[{_{\leq \mu}}{\newoverline{\mathfrak{p}}}^{\lambda}_K: {_{\leq \mu}}{\pZ^{\lambda}_K} \longrightarrow \dBun_N\] of $\newoverline{\mathfrak{p}}^{\lambda}_K$ to ${_{\leq \mu}}{\pZ^{\lambda}_K}$ by base-change. We conclude that $({_{\leq \mu}}{\ofp_K^{\lambda}})^! (\IC_{\dBun_N})$ is the IC sheaf up to a cohomological shift. The statement of the theorem is therefore true on the latter space (note the cohomological shift is absorbed into the definition of $\sIC$). 

For the general case, let $\lambda$ be arbitrary and let $\nu$ be big enough such that $\lambda+\nu$ is sufficiently dominant for all $\mu \leq \nu$. Factorization provides a commutative square \[\begin{tikzcd}
	{[\oZ^{\nu}_K \times \pZ^{\lambda}_K]_{\disj}} & {{_{\leq \lambda}}{\pZ^{\lambda+\nu}_{K,\disj}}} \\
	{[\oZ^{\nu} \times \pZ^{\lambda}]_{\disj}} & {{_{\leq \lambda}}{\pZ^{\lambda+\nu}_{\disj}}}
	\arrow[from=1-1, to=1-2]
	\arrow[from=1-2, to=2-2]
	\arrow[from=1-1, to=2-1]
	\arrow[from=2-1, to=2-2]
\end{tikzcd}\] where both horizontal arrows are open immersions (see also the proof of Proposition \ref{zasissmooth}). Hence there is a natural isomorphism \begin{equation}\label{disjiso}\alpha: \IC^{\operatorname{ren}}_{\oZ^{\nu}_K} \boxtimes_{\disj} \IC^{\operatorname{ren}}_{\pZ^{\lambda}_K} \overset{\sim}{\longrightarrow} \IC^{\operatorname{ren}}_{\oZ^{\nu}} \boxtimes_{\disj} \mathfrak{r}_Z^!(\sIC)\end{equation} where $(-) \boxtimes_{\disj} (-)$ denotes the external product restricted to the disjoint locus. 

Since $\oZ^{\nu}_K$ is smooth, the left hand side (resp. right hand side) of \eqref{disjiso} is the $!$-pullback of $\IC^{\operatorname{ren}}_{\pZ^{\lambda}_K}$ (resp. $\mathfrak{r}_Z^!(\sIC)$) along the projection \[h: [\oZ^{\nu}_K \times \pZ^{\lambda}_K]_{\disj} \to \pZ^{\lambda}_K\] up to a cohomological shift. Since $h$ is schematic and surjective, $h^!$ is conservative and t-exact up to a cohomological shift. It follows that $\mathfrak{r}_Z^!(\sIC)$ is perverse up to a cohomological shift\footnote{For us ``perverse" means a holonomic D-module concentrated in cohomological degree $0$.}. Since the geometric fibers of $h$ are connected, the functor $h^!$ is fully faithful when restricted to perverse sheaves, and hence the isomorphism $\alpha$ descends to an isomorphism $\beta: \IC^{\operatorname{ren}}_{\pZ^{\lambda}_K} \overset{\sim}{\to} \mathfrak{r}_Z^! (\sIC)$ in $\fD(\pZ^{\lambda}_K)$. 

It remains to check that $\beta$ is an isomorphism of factorization algebras. A priori, we have \emph{two} factorization structures on $\IC^{\operatorname{ren}}_{\pZ_K}$: one by definition and the other obtained via transport of structure from $\beta$. However, the equivalence $\beta$ extends the tautological isomorphism \[\omega_{\cZ_K^-}=\IC^{\operatorname{ren}}_{\pZ_K}|^!_{\cZ_K^-} \overset{\sim}{\longrightarrow} \mathfrak{r}_Z^!(\sIC)|^!_{\cZ^-}=\omega_{\cZ_K^-} \] of factorization algebras over $\cZ_K^-$ obtained by Proposition \ref{zasissmooth}. The latter condition uniquely determines the factorization structure on $\IC^{\operatorname{ren}}_{\pZ_K}$ by Remark \ref{density}.   
\end{proof}

%\begin{corollary} We have a canonical isomorphism 
%\[\mathfrak{r}_Z^! (\sIC) \longrightarrow j^-_{K,!*} (\omega_{\cZ^-_K})\] extending the tautological identification over $\cZ^-_K$, where by $j^-_{K,!*}(\omega_{\cZ^-_K})$ we mean the appropriate shift of the usual intermediate extension.
%\end{corollary}
%\begin{proof} This is a direct consequence of Proposition \ref{zasissmooth} and Theorem \ref{sicisic}. 
%\end{proof}

\begin{remark}
    Note that in the course of the proof of Theorem \ref{sicisic} we have showed that $\pZ_K$ is locally in the smooth topology isomorphic to $\dBun_N$. From this perspective, it is clear why the intersection cohomology of $\pZ_K$ should model the semi-infinite intersection cohomology sheaf. 
\end{remark}

\subsection{Recovering the effective category.} The goal of this section is to show that the effective category $\fD_{\eff}(\pZ)$ (and consequently the unital category $\fD_{\untl}(\newoverline{S}^0_{\Ran})$) can be recovered from $\fD(\pZ_K)$ after imposing equivariance with respect to a certain groupoid. Under this equivalence, the renormalized intersection cohomology sheaf $\IC^{\operatorname{ren}}_{\pZ_K}$ will match up with $\sICR$. 

Recall that a prestack $\cY$ is QCA if it is quasi-compact with affine stabilizers \cite{DG}. Moreover, a stack is called \emph{safe} if the neutral component of any stabilizer of a geometric point is a unipotent algebraic group \cite{DG}. For example, any Deligne-Mumford stack is safe. We have the following general lemma. 

\begin{lemma}\label{conserv} Let $\varphi: \cX \to \cY$ be a proper (not necessarily schematic) morphism of safe QCA algebraic stacks which is surjective on geometric points. Then $\varphi^!$ is monadic, and in particular admits a left adjoint. 
\end{lemma}
\begin{proof} We wish to show that $\varphi^!$ satisfies the hypotheses of the Barr-Beck-Lurie theorem. First we will show that $\varphi^!$ is conservative. Let $c \in \fD(\cY)$ be an arbitrary nonzero D-module on $\cY$. By definition, we have an equivalence \[\fD(\cY) \overset{\sim}{\longrightarrow} \operatorname{lim}_{S \in \Aff_{/\cY}} \fD(S)\] so we may choose an affine scheme $S$ with a map $f:S \to \cY$ such that $f^!(c)$ is not zero. 

The fiber product $\cX \times_{\cY} S$ is algebraic, so there exists a smooth cover \[g: T \longrightarrow \cX \times_{\cY} S\] by a scheme $T$. Letting $p_S: \cX \times_{\cY} S \to S$ denote the projection, we see it suffices to show that $(p_S \circ g)^!(c)$ is nonzero. 

By assumption, $p_S \circ g$ is surjective on geometric points. Then the right vertical arrow in the commutative diagram \[\begin{tikzcd}
	{\fD(S)} & {\operatorname{IndCoh}(S)} \\
	{\fD(T)} & {\operatorname{IndCoh}(T)}
	\arrow["{\operatorname{oblv}_S}", from=1-1, to=1-2]
	\arrow["{(p_S \circ g)^!}"', from=1-1, to=2-1]
	\arrow["{\operatorname{oblv}_T}"', from=2-1, to=2-2]
	\arrow["{(p_S \circ g)^{!,\operatorname{IndCoh}}}", from=1-2, to=2-2]
\end{tikzcd}\] is conservative by Proposition 8.1.2 of \cite{Ga-.5}. Since the forgetful functor from D-modules to indcoherent sheaves is conservative, we see that $(p_S \circ g)^!$ is conservative as well. 

we will now verify that $\varphi^!$ admits a left adjoint. Since $\cX$ and $\cY$ are QCA, their categories of D-modules are compactly generated. Moreover, since they are both safe we have equalities \[\fD(\cX)^c = \fD_{\operatorname{coh}}(\cX) \,\,\, \text{and} \,\,\, \fD(\cY)^c = \fD_{\operatorname{coh}}(\cY)\] where for a stack $\cY'$ we denote by $\fD_{\operatorname{coh}}(\cY')$ the category of D-modules $c$ such that for every scheme $S$ with a smooth map $f:S \to \cY'$ the pullback $f^!(c)$ has coherent cohomologies. 

By Proposition 10.6.2 of \cite{DG}, the de Rham pushforward functor \[\varphi_*: \fD(\cX) \longrightarrow \fD(\cY)\] preserves the coherent subcategories, and hence preserves compact objects by the discussion above. Now the category of D-modules for any QCA stack is dualizable, and under Verdier duality we have $\varphi^! \overset{\sim}{\to} (\varphi_*)^{\vee}$. Hence by \cite{Ga-1} the functor $\varphi^!$ admits a left adjoint $\varphi_!$. Since $\varphi^!$ is continuous, it is monadic by the Barr-Beck-Lurie theorem and the associated monad has underlying functor given by the composition $\varphi^!\varphi_!: \fD(\cX) \to \fD(\cX)$.
\end{proof}

Note that, unless $\varphi$ is schematic, in general $\varphi_!$ will differ from $\varphi_*$. In fact, the difference is related to the pseudo-identity functor of Gaitsgory. Nevertheless, we can relax the schematicity assumption in the following case. A map $\varphi: \cX \to \cY$ is called \emph{Deligne-Mumford} if for every affine scheme $S$ with a map $S \to \cY$, the fiber product $\cX \times_{\cY} S$ is a Deligne-Mumford stack. We have the following lemma.  

\begin{lemma}\label{dmpushforward} Let $\varphi: \cX \to \cY$ be a proper Deligne-Mumford morphism of safe QCA stacks. Then the left adjoint to $\varphi^!$ is defined and coincides with the de Rham pushforward $\varphi_*$.
\end{lemma}
\begin{proof} Since the question is local on $\cY$, we may assume that $\cY$ is an affine scheme. The fact that the left adjoint $\varphi_!$ to $\varphi^!$ is defined follows from Lemma \ref{conserv}. Since the categories of D-modules on $\cX$ and $\cY$ are dualizable and canonically self-dual, the functor $\varphi_!$ is given by a kernel $K \in \fD(\cX \times \cY)$. 

By definition, \[K= (\operatorname{id}_{\fD(\cX)} \otimes \varphi_!)\big((\Delta_{\cX})_*(\omega_{\cX})\big) \] where $\operatorname{id}_{\fD(\cX)} \otimes \varphi_!$ is the functor \[\fD(\cX \times \cX) \simeq \fD(\cX) \otimes \fD(\cX) \longrightarrow \fD(\cX) \otimes \fD(\cY) \simeq \fD(\cX \times \cY)\] given by the tensor product of the identity with $\varphi_!$. Here, $\Delta_{\cX}$ denotes the diagonal of $\cX$.

Since \[\operatorname{id}_{\cX} \times \varphi: \cX \times \cX \longrightarrow \cX \times \cY\] is proper, $K$ also identifies with $(\operatorname{id}_{\cX} \times \varphi)_!\big((\Delta_{\cX})_*(\omega)\big)$. By assumption, $\Delta_{\cX}$ is schematic and proper, and hence $(\Delta_{\cX})_!=(\Delta_{\cX})_*$. 

We therefore obtain that $K$ is given by $(\gamma_{\varphi})_!(\omega_{\cX})$, where $\gamma_{\varphi}$ denotes the inclusion \[\Gamma_{\varphi} \hookrightarrow \cX \times \cY \] of the graph of $\varphi$. But the latter is proper and schematic, and therefore $(\gamma_{\varphi})_!$ is equivalent to $(\gamma_{\varphi})_*$. Since $(\gamma_{\varphi})_*(\omega_{\cX})$ is the kernel of $\varphi_*$, we are done. 
\end{proof}

%Call a morphism $\cX \to \cY$ \emph{Deligne-Mumford} if for every affine scheme $S$ with a map $f:S \to \cY$, the fiber product $\cX \times_{\cY} S$ is a Deligne-Mumford stack. 

%\begin{lemma} Let $\varphi: \cX \to \cY$ be a proper Deligne-Mumford morphism of QCA algebraic stacks. Then the left adjoint $\varphi_!$ of \[\varphi^!: \fD(\cY) \longrightarrow \fD(\cX)\] is defined and coincides with the de Rham pushforward $\varphi_*: \fD(\cX) \to \fD(\cY)$.
%\end{lemma}
%\begin{proof} By Drinfeld-Gaitsgory, the functor \[\varphi_*: \fD(\cX) \longrightarrow \fD(\cY)\] sends $\fD_{\operatorname{coh}}(\cX)$ to $\fD_{\operatorname{coh}}(\cY)$ and therefore preserves compact objects. 
%\end{proof}

Denote by $\mathfrak{r}_Z^{\gen}$ the composition \[\pZ_K \overset{\mathfrak{r}_Z}{\longrightarrow} \pZ \overset{\mathfrak{t}^{\gen}_Z}{\longrightarrow} \cZ^{\gen}\]We are now ready to prove the following reconstruction theorem.

\begin{theorem}\label{recovery} There is a monad $\mathscr{M}_{K}$ acting on $\fD(\pZ_K)$ such that \begin{enumerate} \item $\mathscr{M}_K$ sends factorization algebras to factorization algebras; \item $\IC^{\operatorname{ren}}_{\pZ_K}$ is a module for $\mathscr{M}_K$. \end{enumerate} Moreover, there is an equivalence \[(\mathfrak{r}^{\gen}_Z)^!_{\operatorname{enh}}: \fD(\cZ^{\gen}) \overset{\sim}{\longrightarrow} \mathscr{M}_K-\operatorname{mod}\big(\fD(\pZ_K)\big)\] such that $(\mathfrak{r}^{\gen}_Z)^!_{\operatorname{enh}}(\sIC)$ is canonically isomorphic to $\IC^{\operatorname{ren}}_{\pZ_K}$ as a factorization algebra. 
\end{theorem}
\begin{proof} We will begin with the construction of the monad and a proof of the stated equivalence. Recall that the projection $\dBun_{B^-} \to \Bun_{G,B^-}^{\gen}$ is proper and schematic, as well as fppf locally surjective. Since these properties are preserved under base change, the same holds for the map \[\mathfrak{t}^{\gen}_Z: \pZ \longrightarrow \cZ^{gen}.\] In particular, the de Rham pushforward functor $(\mathfrak{r}^{\gen}_Z)_*$ is left adjoint to the conservative and continuous functor $(\mathfrak{r}^{\gen}_Z)^!$.

Fix a coweight $\lambda \in \Lambda^-$. Then $\pZ^{\lambda}$ is quasicompact, and hence $\pZ^{\lambda}_K$ is a safe QCA stack. Since $\mathfrak{r}^{\lambda}_Z: \pZ^{\lambda}_K \to \pZ^{\lambda}$ is proper and surjective on geometric points for every $\lambda$, by Lemma \ref{conserv} the functor \[(\mathfrak{r}^{\lambda}_Z)^!: \fD(\pZ^{\lambda}) \longrightarrow \fD(\pZ^{\lambda}_K)\] is monadic with left adjoint $(\mathfrak{r}^{\lambda}_Z)_!$. 

By Lemma \ref{zastavaisdm}, the morphism $\mathfrak{r}^{\lambda}_Z$ is Deligne-Mumford, and hence by Lemma \ref{dmpushforward} the left adjoint $(\mathfrak{r}^{\lambda}_Z)_!$ coincides with $(\mathfrak{r}^{\lambda}_Z)_*$. Since $\pZ_K$ is the disjoint union of the connected components $\pZ^{\lambda}_K$, by the previous paragraph the functor $(\mathfrak{r}_Z^{\gen})^!$ admits a left adjoint given by the de Rham pushforward functor $(\mathfrak{r}_Z^{\gen})_*$. 

Let $\mathscr{M}_K$ denote the monad \[(\mathfrak{r}_Z^{\gen})^! \circ (\mathfrak{r}_Z^{\gen})_* : \fD(\pZ_K) \longrightarrow \fD(\pZ_K) \] acting on $\fD(\pZ_K)$. Since $(\mathfrak{r}_Z^{\gen})^!$ is conservative, by Barr-Beck-Lurie we conclude that there is an equivalence of categories \[(\mathfrak{r}^{\gen}_Z)^!_{\operatorname{enh}}: \fD(\cZ^{\gen}) \overset{\sim}{\longrightarrow} \mathscr{M}_K - \operatorname{mod}\big(\fD(\pZ_K)\big)\] whose composition with the forgetful functor \[\oblv_{\mathscr{M}_K}: \mathscr{M}_K - \operatorname{mod}(\fD(\pZ_K)) \longrightarrow \fD(\pZ_K)\] canonically identifies with $(\mathfrak{r}^{\gen}_Z)^!$.

Note that since $\mathfrak{r}_Z$ is a morphism of factorization spaces, the monad $\mathscr{M}_K$ preserves factorization algebras, proving point $1$. To show point $2$, note that by Theorem \ref{effisun} the Zastava semi-infinite IC sheaf $\sIC$ has an effective structure, and hence may be upgraded to an object of $\fD(\cZ^{\gen})$, which we denote by the same name. By Theorem \ref{sicisic} there is a canonical isomorphism \[\IC^{\operatorname{ren}}_{\pZ_K} \overset{\sim}{\longrightarrow} (\mathfrak{r}^{\gen}_Z)^!_{\operatorname{enh}}( \sIC),\] from which we conclude that $\IC^{\operatorname{ren}}_{\pZ_K}$ has a natural structure of an $\mathscr{M}_K$-module and is canonically equivalent to $(\mathfrak{r}^{\gen}_Z)^!_{\operatorname{enh}}(\sIC)$. 
\end{proof}

The monad $\mathscr{M}_K$ has the following concrete description. Consider the Cartesian diagram \[\begin{tikzcd}
	{\pZ_K \times_{\cZ^{\gen}} \pZ_K} & {\pZ_K} \\
	{\pZ_K} & {\cZ^{\gen}}
	\arrow["{\operatorname{pr}_2}", from=1-1, to=1-2]
	\arrow["{\mathfrak{r}_Z^{\gen}}", from=1-2, to=2-2]
	\arrow["{\operatorname{pr}_1}"', from=1-1, to=2-1]
	\arrow["{\mathfrak{r}_Z^{\gen}}"', from=2-1, to=2-2]
\end{tikzcd}\] where $\operatorname{pr}_i$ for $i \in \{1,2\}$ denotes the projection onto the $i$th factor. Base change along this diagram identifies the underlying functor of the monad $\mathscr{M}_K$ with the composition of functors $(\operatorname{pr}_2)_* \circ (\operatorname{pr}_1)^!$. 

\section{Parabolic semi-infinite IC sheaves}\label{parabolicicsheaves}

In this section, we consider an application of the theory developed in this paper. Namely, we will construct an analogue of $\sICR$ for arbitrary parabolic subgroups of $G$. In the following, let $P \subseteq G$ be a parabolic subgroup containing $B$ with Levi factor $M$ and unipotent radical $N_P$. To $P$ we may associate a subdiagram of the Dynkin diagram of $G$ spanned by some set of vertices $\mathcal{I}_M \subseteq \mathcal{I}_G$. Let $P^-$ denote the opposite parabolic satisfying $P \cap P^- =M$.

\subsection{Parabolic Zastava spaces.} One can consider the \emph{parabolic compactified Zastava space}, the moduli space \[\cZ_{\widetilde{N}_P,\newoverline{P}^-} \coloneqq (\widetilde{\Bun}_{N_P} \times_{\Bun_G} \dBun_{P^-})^0\] of generically transverse pairs of generalized reductions of a $G$-bundle to $N_P$ and $P^-$. Here, $\widetilde{\Bun}_{N_P}$ denotes the fiber product \[\widetilde{\Bun}_P \times_{\Bun_M} \{\mathscr{P}^0_M\}\] where $\widetilde{\Bun}_P$ is the compactification of $\Bun_P$ considered in \cite{BFGM}, and $\dBun_{P^-}$ denotes the ``naive" compactification of $\Bun_{P^-}$ also considered in \emph{loc. cit.}. 

Note there is a natural defect of transversality map \[\cZ_{\widetilde{N}_P,\newoverline{P}^-} \longrightarrow \Conf_P,\] where $\Conf_P$ is the configuration space of points in $X$ associated to the lattice $\Lambda^-_P$ of negative cocharacters \[ \G_m \longrightarrow M/[M,M].\] In other words, we consider the negative span of the images of the coroots $\alpha_i$ for $i \in \mathcal{I}_G \setminus \mathcal{I}_M$ under the map $\Lambda \to \Lambda_P$. Just as in the principal case, there is an action of $\Conf_P$ on $\cZ_{\widetilde{N}_P, \newoverline{P}^-}$ making the latter into an augmented monoid over $\Conf_P$. Hence we may consider the associated lax prestack $\cZ^{\to}_{\widetilde{N}_P, \newoverline{P}^-}$ and its correspondence version \[\cZ^{\to}_{\widetilde{N}_P, \newoverline{P}^-,\corr} \longrightarrow \Conf^{\to}_{P,\corr}.\] 

Let us now denote by $\widetilde{S}^0_{P,\Ran}$ the following closed subfunctor of $\Gr_{G,\Ran}$. A map $S \to \widetilde{S}^0_{P,\Ran}$ is equivalent to a point $x_I$ of $\Ran$, a $G$-bundle $\mathscr{P}_G$ over $X \times S$, and a trivialization $\alpha$ away from $\Gamma_{x_I}$ such that the following holds. For every representation $V$ of $G$, we require that the composition of the maps \[V^{N_P}_{\mathscr{P}^0_M} \longrightarrow V_{\mathscr{P}^0_G} \overset{\alpha}{\longrightarrow} V_{\mathscr{P}_G}\] extends (possibly with zeroes) to $X \times S$. Here we regard the invariants $V^{N_P}$ with its natural structure as an $M$-representation.

\begin{proposition}\label{paraboliceffisun}
    The straightening $\cZ_{\widetilde{N}_P,\newoverline{P}^-, \Ran}$ is naturally equivalent to $\widetilde{S}^0_{P,\Ran}$ as a factorization space.
\end{proposition}
\begin{proof}
    The argument is the same as in the principal case (see the proof of \ref{effisun}).
\end{proof}

Note for the proof of Proposition \ref{paraboliceffisun} to work, it is crucial that we compactify $\Bun_{P^-}$ ``naively"; there is no action of $\Conf_P$ on $\widetilde{\Bun}_{P^-}$. 

\subsection{The sheaf.} Choose a coweight $\G_m \to Z(M)$ so that it contracts $N_P$ to the identity $1 \in N_P$. Consider the resulting $\G_m$-action on $\cZ_{\widetilde{N}_P, \newoverline{P}^-}$. For $\theta \in \Lambda^-_P$ the fixed points of the connected component of $\cZ^{\theta}_{\widetilde{N}_P,\newoverline{P}^-}$ identify with $\Gr^{+,\theta'}_{M, \Conf} \times X^{\theta}$, where $\Gr^{+,\theta'}_{M,\Conf}$ is a closed subscheme of the affine Grassmannian $\Gr_{M,\Conf}$ (see \cite{BFGM} for the definition) and $\theta' \leq \theta$. The corresponding attracting locus is the scheme \[\cZ^-_{P,=\theta} \coloneqq \Gr^{+,\theta}_{M,\Conf} \times_{\Bun_M} \cZ_{P,\newoverline{P}^-}, \] where $\cZ_{P,\newoverline{P}^-}$ is the relative Zastava $(\Bun_P \times_{\Bun_G} \dBun_{P^-})^0$. For example, when $\theta=0$ the fixed locus is $\Conf$ and the attracting locus is $\cZ_{N_P,\newoverline{P}^-}=(\Bun_{N_P} \times_{\Bun_G} \dBun_{P^-})^0$. The repelling loci are the schemes $\cZ^{\theta}_{\widetilde{N}_P,P^-} \times \Conf_P$. 

Since $\cZ_{\widetilde{N}_P, \newoverline{P}^-}$ is a proper scheme, each of whose connected components are of finite-type, it is $\G_m$-complete and we may consider the semi-infinite category $\operatorname{SI}(\cZ_{\widetilde{N}_P, \newoverline{P}^-})$ together with its factorization algebra $\IC^{\frac{\infty}{2}}(\cZ_{\widetilde{N}_P, \newoverline{P}^-})$. As in the principal case, everything is compatible with the $\Conf$ action and we obtain a semi-infinite category $\operatorname{SI}^{\leq 0}_P$ of effective sheaves on $\cZ_{\widetilde{N}_P, \newoverline{P}^-}$ together with a factorization algebra $\IC^{\frac{\infty}{2}}_P$. Explicitly, $\IC^{\frac{\infty}{2}}_P$ is the intermediate extension of the dualizing sheaf $\omega_{\cZ_{N_P,\newoverline{P}^-}}$ in the semi-infinite t-structure on $\operatorname{SI}(\cZ_{\widetilde{N}_P,\newoverline{P}^-})$.

Now by Proposition \ref{paraboliceffisun} and Proposition \ref{effectivetounital} we obtain a unital factorization algebra denoted $\IC^{\frac{\infty}{2}}_{P,\Ran}$ on $\Gr_{G,\Ran}$.

\begin{definition}\label{parabolicicdefinition} We will call the factorization algebra $\IC^{\frac{\infty}{2}}_{P,\Ran}$ the \emph{parabolic semi-infinite intersection cohomology sheaf}. 
\end{definition}

Let us consider the extreme cases. When $P=B$, we recover the semi-infinite intersection cohomology sheaf discussed in the bulk of this text. At the other extreme we have $P=G$. In this case, the support of $\IC^{\frac{\infty}{2}}_{P,\Ran}$ is the \emph{positive locus} of $\Gr_{G,\Ran}$. Explicitly, this is the moduli space of triples $(x_I,\mathscr{P}_G,\alpha)$ as in the definition of $\Gr_{G,\Ran}$ such that for every $V \in \operatorname{Rep}(G)$ the a priori meromorphic map \[V_{\mathscr{P}^0_G} \overset{\alpha}{\longrightarrow} V_{\mathscr{P}_G}\] is regular. Since $G$ is semisimple, this induces a nonzero endomorphism of the trivial line bundle on $X$ by taking determinants, and hence $\alpha$ itself must extend to the disc. It follows that $\IC^{\frac{\infty}{2}}_{P,\Ran}$ is nothing more than the factorizable unit for the chiral category $\fD(\Gr_{G,\Ran})$, obtained as the pushforward of the dualizing sheaf along the canonical section \[\Ran \longrightarrow \Gr_{G,\Ran}\] induced by the unital structure of $\Gr_{G,\Ran}$. 

\subsection{Local-to-global parabolic compatibility.} We would like to show that the Zastava parabolic semi-infinite IC sheaf is compatible with the global IC sheaf of parabolic Drinfeld's compactification. To do this, we first need some preparations. The constructions that follow can be found in \cite{ABBGM} for the principal case and in the proof of Proposition $4.2.21$ of \cite{DL} for a general parabolic.

For a coweight $\theta \in \Lambda^-_P$, let $\widetilde{\Bun}_{N_P,\operatorname{level},\theta}$ denote the moduli space whose $S$-points consists of a point $D \in X^{\theta}(S)$, a generalized reduction $(\mathscr{P}_G,\kappa) \in \widetilde{\Bun}_{N_P}$ whose locus of degeneracy is disjoint from the support of $D$, and a trivialization of the resulting $N_P$ bundle on the formal neighborhood of $\Gamma_D$. Letting $\widetilde{\Bun}_{N_P, \operatorname{good},\theta}$ be the moduli of a divisor $D$ as before and a generalized $N_P$-reduction which is nondegenerate at $D$, we see there is a map \[\widetilde{\Bun}_{N_P,\operatorname{level},\theta} \longrightarrow \widetilde{\Bun}_{N_P,\operatorname{good},\theta}\] which is a torsor for the group $\mathfrak{L}^+_{X^{\theta}}(N_P)$ over $X^{\theta}$ of arcs from the multidisc into $N_P$. By regluing, the action of $\mathfrak{L}^+_{X^{\theta}}(N_P)$ extends to an action of $\mathfrak{L}_{X^{\theta}}(N_P)$ which does not preserve the projection to $\widetilde{\Bun}_{N_P,\operatorname{good},\theta}$. We therefore obtain a groupoid \[\begin{tikzcd}\label{groupoidaction}
	{\mathscr{H}^{\theta}_{N_P} \coloneqq \mathfrak{L}_{X^{\theta}}(N_P) \times^{\mathfrak{L}^+_{X^{\theta}}(N_P)} \widetilde{\Bun}_{N_P,\operatorname{level},\theta}} & {\widetilde{\Bun}_{N_P,\operatorname{good},\theta}}
	\arrow[shift right, from=1-1, to=1-2]
	\arrow[shift left=2, from=1-1, to=1-2]
\end{tikzcd}\] where the two arrows are given by projection and action. Note the category given by imposing $\mathscr{H}_{N_P}$-equivariance on $D$-modules over $\widetilde{\Bun}_{N_P,\operatorname{good},\theta}$ is a full subcategory of $\mathfrak{D}(\widetilde{\Bun}_{N_P,\operatorname{good},\theta})$ by unipotence of $N_P$. 

Define $\operatorname{SI}^{\leq 0}_{\operatorname{glob},P}$ to be the full subcategory of $\fD(\widetilde{\Bun}_{N_P})$ consisting of sheaves whose $!$-pullback to $\widetilde{\Bun}_{N_P,\operatorname{good},\theta}$ is $\mathscr{H}^{\theta}_{N_P}$-equivariant for all $\theta \in \Lambda^-_P$. Since the action of $\mathscr{H}_{N_P}$ preserves the strata of $\widetilde{\Bun}_{N_P}$, we obtain categories $\operatorname{SI}^{=\theta}_{\operatorname{glob},P}$ for every stratum $\Gr^{+,\theta}_{M,\Conf}\times_{\Bun_M} \Bun_P$. Moreover, since for each $\theta$ the projection \[\widetilde{\Bun}_{N_P,\operatorname{good},\theta} \longrightarrow \widetilde{\Bun}_{N_P} \] is smooth, the categories $\operatorname{SI}^{=\theta}_{\operatorname{glob},P}$ and $\operatorname{SI}^{ \leq 0}_{\operatorname{glob},P}$ acquire t-structures, simply by restricting the perverse t-structure up to the usual shift by $\operatorname{dim}(\Bun_P)$. The following lemma is standard. 

\begin{lemma}\label{groupoidequiv}
    An object of $\fD(\widetilde{\Bun}_{N_P})$ belongs to $\operatorname{SI}^{\leq 0}_{\operatorname{glob},N_P}$ if and only if its $!$-pullback to a stratum $\Gr^{+,\theta}_{M,\Conf} \times_{\Bun_M} \Bun_P$ is $!$-pulled back from $\Gr^{+,\theta}_{M,\Conf}$. 
\end{lemma}

Let us pause for a moment to fix some notation. Recall that for a factorization space $\cY$ over $\Conf$ there is a notion of \emph{weak factorization module space} for $\cY$. This is a space $\cX$ equipped with homotopy associative morphisms \[\varphi: [\cY \times \cX]_{\disj} \longrightarrow \cX_{\disj}\] which are compatible with the factorization structure on $\cY$ in a suitable sense. Given a factorization algebra $A$ in $\mathfrak{D}(\cY)$ we may form the category \[A-\operatorname{mod}^{\operatorname{fact}}(\mathfrak{D}(\cX))\] of factorization modules for $A$ over $\cX$. Roughly, these are sheaves $c$ in $\mathfrak{D}(\cX)$ equipped with isomorphisms \[ A \boxtimes c \overset{\sim}{\longrightarrow} \varphi^!(c)\] which are compatible with the factorization structure on $A$. For an overview on factorization modules see \cite{R1}.

Note that $\cZ_{\widetilde{N}_P,P^-}$ is naturally a weak factorization space for $\cZ_{N_P,P^-}$, and let $\mathfrak{p}: \cZ_{\widetilde{N}_P,P^-} \to \widetilde{\Bun}_{N_P}$ denote the usual map from the affine Zastava. Note that since $\cZ_{N_P,P^-}$ is smooth, its renormalized intersection cohomology sheaf $\IC^{\operatorname{ren}}_{\cZ_{N_P,P^-}}$ is just the dualizing sheaf $\omega_{\cZ_{N_P,P^-}}$. In particular, it is endowed with a canonical structure of a factorization algebra, and we can consider its category of factorization modules in $\mathfrak{D}(\cZ_{\widetilde{N}_P,P^-}$. We have the following important lemma. 

\begin{lemma}\label{weakfactorization}
    The functor of $!$-pullback \[\mathfrak{p}^!: \operatorname{SI}^{\leq 0}_{\operatorname{glob},P} \longrightarrow \fD(\cZ_{\widetilde{N}_P,P^-}) \] upgrades to a functor \[\mathfrak{p}^!_{\operatorname{enh}}: \operatorname{SI}^{\leq 0}_{\operatorname{glob},P} \longrightarrow \IC^{\operatorname{ren}}_{\cZ_{N_P,P^-}} - \operatorname{mod}^{\operatorname{fact}}(\fD(\cZ_{\widetilde{N}_P,P^-}))\] from the semi-infinite category to factorization modules for the renormalized intersection cohomology sheaf of the open Zastava.
\end{lemma}
\begin{proof}
    The proof is almost identical to one direction of Proposition $4.2.6$ of \cite{ABBGM}. The idea is as follows. We can view the disjoint locus $[\cZ_{N_P,P^-} \times \cZ_{\widetilde{N}_P,P^-}]_{\disj}$ as a groupoid acting on $\cZ_{\widetilde{N}_P,P^-}$, where one map is given by projection and the other by factorization. By generic transversality, there is a map \[[\cZ_{N_P,P^-} \times \cZ_{\widetilde{N}_P,P^-}]_{\disj} \longrightarrow \coprod_{\theta \in \Lambda^-_P} \mathscr{H}^{\theta}_{N_P}\] commuting with the action maps on $\cZ_{\widetilde{N}_P,P^-}$ and $\coprod_{\theta \in \Lambda^-_P}\widetilde{\Bun}_{N_P,\operatorname{good},\theta}$. 
    
    The result now follows from the fact that the renormalized intersection cohomology sheaf of $\cZ_{N_P,P^-}$ is the dualizing sheaf as a factorization algebra. Indeed, for a sheaf $c \in \mathfrak{D}(\cZ_{\widetilde{N}_P,P^-})$, the structure of a factorization module on $c$ for the factorization algebra $\IC^{\operatorname{ren}}_{\cZ_{N_P,P^-}}$ is therefore equivalent to an equivariant structure on $c$ with respect to the groupoid $[\cZ_{N_P,P^-} \times \cZ_{\widetilde{N}_P,P^-}]_{\disj}$ acting on $\cZ_{\widetilde{N}_P,P^-}$. 
\end{proof}

We will denote by \[\widetilde{\mathfrak{p}}: \cZ_{\widetilde{N}_P, \newoverline{P}^-} \longrightarrow \widetilde{\Bun}_{N_P} \] the projection to Drinfeld's compactification. One can show that every object of $\operatorname{SI}^{\leq 0}_{\operatorname{glob},P}$ is $\G_m$-monodromic, and hence the functor $\widetilde{\mathfrak{p}}^!$ factors through $\operatorname{SI}(\cZ_{\widetilde{N}_P,\newoverline{P}^-}$) by Lemma \ref{groupoidequiv}. The next proposition establishes the basic local-to-global property for the sheaf $\IC^{\frac{\infty}{2}}_P$. 

\begin{proposition} \label{parabolictstr}
    For every $\theta \in \Lambda^-_P$, the functor \[\widetilde{\mathfrak{p}}^!: \operatorname{SI}^{\leq 0}_{\operatorname{glob},P} \longrightarrow \operatorname{SI}(\cZ^{\theta}_{\widetilde{N}_P,\newoverline{P}^-}) \] is t-exact. Moreover, there is a canonical equivalence \[\widetilde{\mathfrak{p}}^!(\IC^{\operatorname{ren}}_{\widetilde{\Bun}_{N_P}}) \overset{\sim}{\longrightarrow} \IC^{\frac{\infty}{2}}_P\] extending the tautological isomorphism over $\cZ_{N_P,\newoverline{P}^-}$. 
\end{proposition}
\begin{proof}
    The proof proceeds in the same way as Proposition \ref{texactzastava}. The main observation is that the projection \[\Gr^{+,\theta'}_{M,\Conf} \times_{\Bun_M} \Bun_P \longrightarrow \Gr^{+,\theta'}_{M,\Conf} \] has relative dimension \[\operatorname{dim}(\Bun_{N_P}) + \langle \check{\rho}_G-\check{\rho}_M,2\theta' \rangle\] where we view $\check{\rho}_G-\check{\rho}_M$ as an element of the lattice dual to $\Lambda_P$. Indeed, the fiber over $\Bun_P \to \Bun_M$ at $\mathscr{P}_M$ is the moduli stack of torsors for the twisted group scheme $(N_P)_{\mathscr{P}_M}$ over $X$. To compute the dimension of this stack, one computes its shifted Euler characteristic. Similarly, the dimension of $\cZ^{\theta}_{\widetilde{N_P},P^-}$ is \[\langle \check{\rho}_G-\check{\rho}_M,2\theta \rangle.\] It follows that, just as in the proof of Proposition \ref{texactzastava}, the shifts match up. 

    Now the argument for left exactness follows from a diagram chase around the commutative diagram 
	\[\begin{tikzcd}
	{\Gr^{+,\theta'}_{M,\Conf} \times X^{\theta-\theta'}} & {\Gr^{+,\theta'}_{M,\Conf} \times_{\Bun_M} \cZ^{\theta-\theta'}_{P,\newoverline{P}^-} } & {\cZ^{\theta}_{\widetilde{N}_P,\newoverline{P}^-}} \\
	{\Gr^{+,\theta'}_{M,\Conf}} & {\Gr^{+,\theta'}_{M,\Conf} \times_{\Bun_M} \Bun_P} & {\widetilde{\Bun}_{N_P}}
	\arrow[from=1-3, to=2-3]
	\arrow[from=2-2, to=2-3]
	\arrow[from=1-2, to=1-3]
	\arrow[from=1-2, to=2-2]
	\arrow[from=1-2, to=1-1]
	\arrow[from=1-1, to=2-1]
	\arrow[from=2-2, to=2-1]
\end{tikzcd}\] using Lemma \ref{groupoidequiv} and the above observation on shifts. 

To check right exactness, as in the proof of Proposition \ref{texactzastava}, it suffices to check the following claim. Consider the Cartesian diagram \[\begin{tikzcd}
	{\Gr^{+,\theta}_{M,\Conf}} & {\cZ^{\theta}_{\widetilde{N}_P,P^-}} \\
	{\Gr^{+,\theta}_{M,\Conf} \times_{\Bun_M} \Bun_P} & {\widetilde{\Bun}_{N_P}}
	\arrow["{\sigma_{=\theta}}", from=1-1, to=1-2]
	\arrow["{\mathfrak{p}_{\theta}}", from=1-2, to=2-2]
	\arrow["{\iota_{\theta}}"', from=1-1, to=2-1]
	\arrow["{j_{=\theta}}"', from=2-1, to=2-2]
\end{tikzcd}\] Base change along this Cartesian square induces a natural transformation \[\sigma_{=\theta}^* \circ \mathfrak{p}^!_{\theta} \longrightarrow \iota^!_{\theta} \circ j_{=\theta}^*\] which we would like to show is an isomorphism when evaluated on objects of $\operatorname{SI}^{\leq 0}_{\operatorname{glob},P}$. This is a standard argument (see e.g. Section $3.7.8$ of \cite{Ga2} and Proposition $4.2.21$ of \cite{DL}) which we will sketch. 

For $\theta \in \Lambda^-_P$, let $\cZ^{\theta,r}_{\widetilde{N}_P,P^-}$ denote the open locus of $\cZ^{\theta}_{\widetilde{N}_P,P^-}$ from Section $3.6$ of \cite{BFGM} which maps smoothly to $\widetilde{\Bun}_{N_P}$. The desired isomorphism clearly holds over this subscheme. As in the proof of Proposition $4.2.21$ of \cite{DL}, the subscheme $\cZ^{\theta,r}_{\widetilde{N}_P,P^-}$ is compatible with factorization in a suitable sense. In other words, the factorization map \[[\cZ^{\theta'}_{N_P,P^-} \times \cZ^{\theta}_{\widetilde{N}_P,P^-}]_{\disj}^r\longhookrightarrow \cZ^{\theta'+\theta}_{\widetilde{N}_P,P^-,\disj}\] lands in \[\cZ^{\theta'+\theta,r}_{\widetilde{N}_P,P^-} \coloneqq \cZ^{\theta'+\theta,r}_{\widetilde{N}_P,P^-} \times_{\Conf^{\theta'+\theta}_P} [\Conf^{\theta'}_P \times \Conf^{\theta}_P]_{\disj} \]for every $\theta'$. Moreover, for $\theta'$ large enough, the projection \[[\cZ^{\theta'}_{N_P,P^-} \times \cZ^{\theta}_{\widetilde{N}_P,P^-}]_{\disj}^r \longrightarrow \cZ^{\theta}_{\widetilde{N}_P,P^-}\] is surjective. Fix such a $\theta'$.

Consider the diagram \[\begin{tikzcd}
	{[\cZ^{\theta'}_{N_P,P^-} \times \Gr^{+,\theta}_{M,\Conf}]^r_{\disj}} & {[\cZ^{\theta'}_{N_P,P^-} \times \cZ_{\widetilde{N}_P,P^-}^{\theta}]^r_{\disj} } \\
	{\Gr^{+,\theta}_{M,\Conf}} & {\cZ_{\widetilde{N}_P,P^-}^{\theta}} \\
	{\Gr^{+,\theta}_{M,\Conf} \times_{\Bun_M} \Bun_P} & {\widetilde{\Bun}_{N_P}}
	\arrow[from=1-1, to=1-2]
	\arrow[from=1-1, to=2-1]
	\arrow[from=1-2, to=2-2]
	\arrow[from=2-1, to=2-2]
	\arrow[from=3-1, to=3-2]
	\arrow[from=2-1, to=3-1]
	\arrow[from=2-2, to=3-2]
\end{tikzcd}\] where the top two vertical arrows are the projections. By $\mathscr{H}_{N_P}$-equivariance, $!$-pulling an object $c \in \operatorname{SI}^{\leq 0}_{\operatorname{glob},P}$ back along the outer vertical composition is the same as $!$-pulling back along the composition \[[\cZ^{\theta'}_{N_P,P^-} \times \cZ^{\theta}_{\widetilde{N}_P,P^-}]^r_{\disj} \overset{\operatorname{fact}}{\longrightarrow} \cZ^{\theta'+\theta,r}_{\widetilde{N}_P,P^-} \longrightarrow \widetilde{\Bun}_{N_P}\] by Lemma \ref{weakfactorization}, and where here $\operatorname{fact}$ denotes the factorization morphism. We conclude that the desired isomorphism holds along the outer square of the diagram. But $\cZ^{\theta'}_{N_P,P^-}$ is smooth, and hence the same is true on the lower square. 
\end{proof}

In the following, write $\mathfrak{p}_{\Ran}$ for the map \[\widetilde{S}^0_{P,\Ran} \longrightarrow \widetilde{\Bun}_{N_P}.\] We have the following corollary to Proposition \ref{parabolictstr}.

\begin{corollary}\label{localtoglobalrancomp}
    There is a canonical equivalence $\mathfrak{p}_{\Ran}^!(\IC^{\operatorname{ren}}_{\widetilde{\Bun}_{N_P}}) \overset{\sim}{\longrightarrow} \IC^{\frac{\infty}{2}}_{P,\Ran}$ extending the tautologial isomorphism over $S^0_{P,\Ran}$. 
\end{corollary}

%\subsection{The parabolic Kontsevich Zastava.} Note that the same definition of $\dBun^K_{B^-}$ goes through for the \emph{naive} compactification $\dBun_{P^-}$. Hence we obtain a resolution of singularities \[\dBun^K_{P^-} \longrightarrow \dBun_{P^-}\] which is an isomorphism over $\Bun_{P^-}$. As before, we obtain the parabolic Kontsevich Zastava $\cZ^K_{\widetilde{N}_P,\newoverline{P}^-}$ and the factorization algebra \[\IC^{\operatorname{ren}}_P \coloneqq \IC^{\operatorname{ren}}_{\cZ^K_{\widetilde{N}_P,\newoverline{P}^-}}\] obtained as its renormalized intersection cohomology sheaf. It is easy to see that analogues of Proposition \ref{kontsevichfactorization} and Theorem \ref{sicisic} also hold in this case. 

\section{Quiver Zastava semi-infinite IC sheaves}\label{quivericsheaves}

we will begin by reviewing the Mirkovi\'{c} \emph{quiver Zastava spaces}\cite{MYZ}, a generalization of the usual theory. At the end, we will construct a semi-infinite IC sheaf for these spaces which agrees with the sheaf in the group case. From now onwards, we will drop the requirement that the curve $X$ is proper. 

\subsubsection{} Pick a graph $Q$ without edge loops and with a totally ordered set of vertices $I$ called ``colors". Consider the "coweight" lattice $\Lambda$ with generators $\alpha_i$ indexed by elements $i \in I$ and let $\Lambda^- \subseteq \Lambda$ denote the negative cone. In other words, $\Lambda^-$ consists of sums $\sum_{i \in I} n_i \alpha_i$ such that $n_i \leq 0$. The adjacency matrix of $Q$ gives a symmetric bilinear form $\kappa$ on $\Lambda$. Explicitly, $\kappa(\alpha_i,\alpha_j)$ is the number of edges between $i$ and $j$ whenever $i \neq j$, and is equal to $-2$ when $i=j$. 

Let $\Conf$ denote the moduli space of divisors on $X$ with coefficients in $\Lambda^-$. Equivalently, $\Conf$ is the moduli space of anti-effective divisors on the curve $X \times I$. This is a scheme with connected components given by partially symmetrized powers $X^{\lambda}$ indexed by $\lambda \in \Lambda^-$. In each $X^{\lambda}$ we have an incidence divisor $\Delta^{\lambda}$ with irreducible components $\Delta^{\lambda}_{i \leq j}$ where $\leq$ denotes the ordering on $I$. Explicitly, $\Delta^{\lambda}_{i \leq j}$ is the divisor where a point of color $i$ is equal to a point of color $j$. Write $\Delta=\cup_{\lambda \in \Lambda^-} \Delta^{\lambda}$ for the union of all the incidence divisors.

Define a factorizable line bundle on $\Conf$ as \[\mathscr{L}_Q \coloneqq \mathcal{O}_{\Conf}(-\kappa\Delta)\] where $-\kappa\Delta$ means that we multiply $\Delta^{\lambda}_{i\leq j}$ by $-\kappa(\alpha_i,\alpha_j)$ whenever $i < j$ and by $1$ whenever $i=j$. Now consider the two-step flag bundle $\Conf^+$ over $\Conf$ consisting of pairs $(D,E)$ of divisors such that $D$ is scheme-theoretically supported on $E$. We have a correspondence \[\begin{tikzcd}
	& {\Conf^+} \\
	\Conf && \Conf
	\arrow["q"', from=1-2, to=2-1]
	\arrow["p", from=1-2, to=2-3]
\end{tikzcd}\] where $q((D,E))=D$ and $p((D,E))=E$. We have the following lemma. 

\begin{lemma}\label{finitelocallyfree}
    The projection $p: \Conf^+ \to \Conf$ is finite locally free, and the quasicoherent sheaf $p_*q^*(\mathscr{L}_Q)$ is a factorizable vector bundle.
\end{lemma}
\begin{proof}
    There is an isomorphism \[\Conf^+ \overset{\sim}{\longrightarrow} \Conf \times \Conf\] sending a pair $(D,E)$ to $(D, E-D)$. Under this isomorphism $q$ is sent to the first projection and $p$ is sent to the addition map, which is obviously finite locally free. Hence $p_*q^*(\mathscr{L}_Q)$ is a vector bundle. Since $p$ and $q$ are compatible with factorization, it also inherits a factorization structure. 
\end{proof}

By Lemma \ref{finitelocallyfree} we can consider the affine bundle given as a relative spectrum \[\mathscr{V}_Q \coloneqq \underline{\Spec}\big(p_*q^*(\mathscr{L}_Q)^{\vee}\big)\] as a factorization space lying over $\Conf$. Now although the projectivization $\mathbb{P}(\mathscr{V}_Q)$ of $\mathscr{V}_Q$ is \emph{not} factorizable in general, its restriction $s^*\mathbb{P}(\mathscr{V}_Q)$ to the simple locus \[s: \Conf_{\operatorname{simple}} \longhookrightarrow \Conf\] will contain a factorizable subscheme $\mathbb{P}^{\operatorname{loc}}(\mathscr{V}_Q)$. Here $\Conf_{\operatorname{simple}}$ is, by definition, the complement of the divisor $\Delta$. 

Using the tautological trivialization of $\mathscr{L}_Q$ over $\Conf_{\operatorname{simple}}$, we can describe $\mathbb{P}^{\operatorname{loc}}(\mathscr{V}_Q)$ as the subscheme which is fiberwise equal to the Segre embedding. In other words, over a divisor \[D=\sum_j \lambda_jx_j \in \Conf_{\operatorname{simple}} \] where each $\lambda_j$ is equal to some $-\alpha_i$ and $x_j \neq x_{j'}$ for $j \neq j'$, the fiber $\mathbb{P}^{\operatorname{loc}}(\mathscr{V}_Q)$ is equal to the image of the Segre embedding \[ \prod_j \mathbb{P}^1 \longhookrightarrow \mathbb{P}(\mathscr{V}_Q)_D\] where we use the tautological isomorphism $\mathbb{P}(\mathscr{V}_Q)_{-\alpha_ix_j} \simeq \mathbb{P}^1$ obtained from the tautological trivialization $\mathscr{L}_{Q,-\alpha_i x_j} \simeq k$ of the fiber of $\mathscr{L}_Q$ over $-\alpha_i x_j$. Note we are using here that the fiber of $p_*q^*(\mathscr{L}_Q)$ over $-\alpha_i x_j$ is canonically isomorphic to $k \oplus \mathscr{L}_{Q,-\alpha_ix_j}$, where the trivial factor comes from viewing the empty set as a subdivisor of $-\alpha_ix_j$. 

\begin{definition}
    Define the \emph{compactified quiver Zastava} $\pZ_Q$ associated to $Q$ to be the closure of $\mathbb{P}^{\operatorname{loc}}(\mathscr{V}_Q)$ inside the projective bundle $\mathbb{P}(\mathscr{V}_Q)$. 
\end{definition}

By construction, $\pZ_Q$ is a proper factorizable scheme over $\Conf$. Moreover, when $Q$ is a graph associated to a Dynkin diagram of type ADE, there is an isomorphism of factorization spaces \[\pZ_Q \overset{\sim}{\longrightarrow} \pZ_G\] where $\pZ_G$ is the usual compactified Zastava associated to the semisimple and simply-connected algebraic group of ADE type connected to $Q$ \cite{MYZ}. Hence we may consider the spaces $\pZ_Q$ to be a generalization of the theory of Drinfeld's Zastavas to arbitrary Kac-Moody algebras with symmetric Cartan matrix. In particular, when $Q$ is of affine ADE type, we obtain a space which may be considered a Zastava space for the corresponding affine Lie algebra.  

\subsection{Quiver Zastava semi-infinite IC sheaves.} The action of the $\Conf^+$ group scheme $\G_m \times \Conf^+$ on the line bundle $q^*(\mathscr{L}_Q)$ induces an action of the Weil restriction $\G \coloneqq \operatorname{Res}^{\Conf^+}_{\Conf}(\G_m \times \Conf^+)$ on $\mathbb{P}(\mathscr{V}_Q)$. By definition, the fiber of $\G$ at a divisor $D$ of degree $\lambda$ is the group of maps from the moduli $\Conf^+_D$ of subdivisors of $D$ to $\G_m$. Note that we have a decomposition \[\Conf^+_D = \coprod_{\mu \leq \lambda} \Conf^{+,\mu}_D\] into connected components. Hence we have \[ \G_D \overset{\sim}{\longrightarrow} \prod_{\mu \leq \lambda} \G^{\mu}_D\] where $\G^{\mu}_D$ is the group of maps $\Conf^{+,\mu}_D \to \G_m$. Define a homomorphism $\G_m \to \G_D$ by requiring that its composition with each projection $\G_D \to \G^{\mu}_D$ sends a point $t \in \G_m$ to the constant function of degree $\langle \check{\rho},\mu \rangle =\sum_{i \in I} n_i$ where $\mu= \sum_{i \in I} n_i\alpha_i$ with $n_i \leq 0$. Clearly, this defines a homomorphism $h: \G_m \to \G$ whose action on $\mathbb{P}(\mathscr{V}_Q)$ preserves the Segre equations. 

Since $\pZ_Q$ is defined as a closure of a Segre embedding, the homomorphism $h$ from the previous section induces an action of $\G_m$ on $\pZ_Q$. Moreover, $\pZ_Q$ is $\G_m$-complete by \cite{D}. Hence we may consider the semi-infinite category $\operatorname{SI}(\pZ_Q)$ as well as the sheaf $\IC^{\frac{\infty}{2}}(\pZ_Q)$. To see that this sheaf coincides with $\sIC$ under the equivalence $\pZ_Q \simeq \pZ_G$, it suffices to match up the $\G_m$-actions on both sides. But this is clear: on the simple locus the $\G_m$-action on $\pZ_G$ is compatible with factorization and hence is easily seen to be the correct one. Hence we may view the construction above as assigning a ``semi-infinite IC sheaf" to each Kac-Moody algebra with a symmetric Cartan matrix.


\begin{thebibliography}{10}

\bibitem{ABBGM} S. Arkhipov, R. Bezrukavnikov, A. Braverman, D. Gaitsgory and I. Mirkovi\'{c}, Modules over the quantum
group and semi-infinite flag manifold, \emph{Transformation Groups} 10 (2005), 279–362.

\bibitem{Bar} J. Barlev, $D$-modules on spaces of rational maps, \emph{Compositio Mathematica} 150 (2014), 835–876. 

\bibitem{BB} A. Beilinson and J. Bernstein, Localisation de $\mathfrak{g}$-modules, C.R.A.S. Paris (1981), t.292.

%\bibitem{BF} A. Braverman and M. Finkelberg, Pursuing the double affine Grassmannian I: transversal slices via instantons on Ak-singularities, \emph{Duke Mathematical Journal} 152 (2010), no. 2, 175–206.

\bibitem{B} A. Bia{\l}ynicki-Birula, Some theorems on actions of algebraic groups, \emph{Ann. of Math.} (2) 98 (1973), 480–497.  

\bibitem{Br} T. Braden, Hyperbolic localization of intersection cohomology, \emph{Transform. Groups} 8 (2003), 209–216.

\bibitem{BFGM} A. Braverman, M. Finkelberg, D. Gaitsgory, I. Mirkovi\'{c}, Intersection cohomology of Drinfeld's compactifications, \emph{Selecta Mathematica. New Series} 8 (2002), no. 3, 381-418. 

%\bibitem{BFN} A. Braverman, M. Finkelberg, and H. Nakajima, Coulomb branches of 3d N = 4 quiver gauge theories and slices in the affine Grassmannian (with two appendices by A. Braverman, M. Finkelberg, J. Kamnitzer, R. Kodera,
%H. Nakajima, B. Webster, and A. Weekes), \emph{Advances in Theoretical Mathematical Physics}
%23 (2019), no. 1, 75–166.

\bibitem{Bu} D. Butson, Equivariant localization in factorization homology and applications in mathematical physics II: Gauge theory applications. \url{https://arxiv.org/abs/2011.14978} (2020).

\bibitem{C} J. Campbell, A resolution of singularities for Drinfeld's compactification by stable maps, \emph{J. Algebraic Geom.} 28 (2019), 153–167.

\bibitem{Cis} D. Cisinski, Higher categories and homotopical algebra, \emph{Cambridge Stud. Adv. Math.}, 180
Cambridge University Press, Cambridge, 2019.

\bibitem{DL} G. Dhillon and S. Lysenko, Semi-infinite parabolic IC-sheaf. \url{https://arxiv.org/abs/2310.06386} (2023).

\bibitem{D} V. Drinfeld, On algebraic spaces with an action of $\G_m$. \url{https://arxiv.org/abs/1308.2604} (2015).

\bibitem{DG0} V. Drinfeld and D. Gaitsgory, On a theorem of Braden, \emph{Transform. Groups} 19 (2014), 313-358.  

\bibitem{DG} V. Drinfeld and D. Gaitsgory, On some finiteness questions for algebraic stacks, \emph{Geom. Funct. Anal.} 23 (2013), 149–294.

\bibitem{FFKM} B. Feigin, M. Frenkel, A. Kuznetsov, I. Mirkovi\'{c}, Semi-infinite flags. II. Local and global intersection cohomology of quasimaps' spaces, \emph{ Differential topology, infinite-dimensional Lie algebras, and applications}, 113-148, Amer. Math. Soc. Transl. Ser. 2, 194, Adv. Math. Sci., 44, \emph{Amer. Math. Soc., Providence, RI}, 1999.

\bibitem{FF} B. Feigin and E. Frenkel, Affine Kac-Moody algebras and semi-infinite flag manifolds, \emph{Communications in Mathematical Physics} 128 (1990) 161–189.

%\bibitem{FG} E. Frenkel and D. Gaitsgory, D-modules on the affine Grassmannian and representations of affine Kac–Moody algebras, Duke Math. J. 125 (2004) 279–327.

\bibitem{FG2} E. Frenkel and D. Gaitsgory, Localization of $\hat{\mathfrak{g}}$-modules on the affine Grassmannian, \emph{Ann. of Math. (2)} 170 (2009), 1339–1381.

\bibitem{FGV} E. Frenkel, D. Gaitsgory, and K. Vilonen, Whittaker patterns in the geometry of moduli spaces of bundles on curves, \emph{Ann. of Math.} (2) 153 (2001), 699–748.

\bibitem{FM} M. Finkelberg and I. Mirkovi\'{c}, Semi-infinite flags I. Case of a global curve $\mathbb{P}^1$, \emph{Differential topology, infinite-dimensional Lie algebras, and applications}, 81–112, Amer. Math. Soc. Transl. Ser. 2, 194, Adv. Math. Sci., 44, \emph{Amer. Math. Soc., Providence, RI}, 1999.

\bibitem{Ga-2} D. Gaitsgory, Contractibility of the space of rational maps, \emph{Invent. Math.} 191 (2013), no.1, 91–196.

\bibitem{Ga-1} D. Gaitsgory, Functors given by kernels, adjunctions and duality, \emph{J. Algebraic Geom.} 25 (2016), 461–548.

\bibitem{Ga-.5} D. Gaitsgory, Ind-coherent sheaves. \url{https://arxiv.org/abs/1105.4857} (2012). 

\bibitem{Ga0} D. Gaitsgory, On factorization algebras arising in the quantum geometric Langlands theory, \emph{Advances in Mathematics} 391 (2021) 107962. 

\bibitem{Ga1} D. Gaitsgory, The semi-infinite intersection cohomology sheaf, \emph{Advances in Mathematics} 327 (2018), 789–868.

\bibitem{Ga2} D. Gaitsgory, The semi-infinite intersection cohomology sheaf-II: The Ran space version. \url{https://arxiv.org/abs/1708.07205?context=math} (2021).

\bibitem{GR} D. Gaitsgory and N. Rozenblyum, Crystals and $D$-modules. \url{https://arxiv.org/abs/1111.2087} (2014). 

\bibitem{Lu} J. Lurie, Higher topos theory, \emph{Ann. of Math. Stud.}, 170
Princeton University Press, Princeton, NJ, 2009.

\bibitem{M} I. Mirkovi\'{c}, Some extensions of the notions of loop Grassmannians.  Available at \url{https://people.math.umass.edu/~mirkovic/A.Notes/LocalSpaces/SibeExtededLoopGrassmannians/Sibe.pdf} (2017).

\bibitem{MYZ} I. Mirkovi\'{c}, Y. Yang, and G. Zhao, Loop Grassmannians of quivers and affine quantum groups. \url{https://arxiv.org/abs/1810.10095} (2020). 

%\bibitem{Mu} S. Mukai, Duality between $D(X)$ and $D(\hat{X})$ with its application to Picard sheaves, \emph{Nagoya Mathematical Journal} 81 (1981), 153-175.

%\bibitem{N} H. Nakajima, Towards geometric Satake correspondence for Kac-Moody algebras – Cherkis bow varieties and affine Lie algebras of type A. \url{https://arxiv.org/abs/1810.04293}.

\bibitem{R1} S. Raskin, Chiral Categories. Available at \url{https://gauss.math.yale.edu/~sr2532/chiralcats.pdf} (2019).

\bibitem{R2} S. Raskin, Chiral principal series categories I: Finite dimensional calculations, \emph{Advances in Mathematics} 388 (2021). 

\bibitem{S} J. Steinebrunner, Locally (co)Cartesian fibrations as realisation fibrations and the classifying space of cospans, \emph{J. Lond. Math. Soc. (2)} 106 (2022), 1291–1318.

\bibitem{Zhu} X. Zhu, An introduction to affine Grassmannians and the geometric Satake equivalence, \emph{Geometry of moduli spaces and representation theory}, pp. 59–154, IAS/Park City Mathe-
matics Series vol. 24, \emph{Amer. Math. Soc., Providence}, 2017.

\end{thebibliography}
\end{document}